\documentclass[reqno]{amsart}
\usepackage{amssymb, amsfonts, stmaryrd, color,fancyhdr, amsmath,amsthm,amsbsy,amsfonts, titletoc,  mathrsfs, mathabx, verbatim,extarrows, wasysym, euscript, enumerate, latexsym, mathtools, epsfig, subfigure, mathabx, graphicx, color, wrapfig, hyperref,graphicx,float, bm, setspace}
\usepackage{CJKutf8}
\usepackage[pagewise,displaymath,mathlines]{lineno}
\newcommand*\patchAmsMathEnvironmentForLineno[1]{%
   \expandafter\let\csname old#1\expandafter\endcsname\csname #1\endcsname
   \expandafter\let\csname oldend#1\expandafter\endcsname\csname end#1\endcsname
   \renewenvironment{#1}%
      {\linenomath\csname old#1\endcsname}%
      {\csname oldend#1\endcsname\endlinenomath}}%
\newcommand*\patchBothAmsMathEnvironmentsForLineno[1]{%
   \patchAmsMathEnvironmentForLineno{#1}%
   \patchAmsMathEnvironmentForLineno{#1*}}%
\AtBeginDocument{%
\patchBothAmsMathEnvironmentsForLineno{equation}%
\patchBothAmsMathEnvironmentsForLineno{align}%
\patchBothAmsMathEnvironmentsForLineno{flalign}%
\patchBothAmsMathEnvironmentsForLineno{alignat}%
\patchBothAmsMathEnvironmentsForLineno{gather}%
\patchBothAmsMathEnvironmentsForLineno{multline}%
}

\numberwithin{equation}{section}
\theoremstyle{plain}
\DeclareMathAlphabet{\mathpzc}{OT1}{pzc}{m}{it}

\def\mcU{\mathcal U}
\def\mcV{\mathcal V}

\def\mbn{\mathbf n}

\def\mrJ{\mathrm J}
\def\mrF{\mathrm F}

\def\msF{\mathscr F}
\def\bE{ {\bf E}}

\def\mbE{\mathbb E}
\def\msG{\mathscr G}
\def\mbR{{\mathbb R}}
\def\msR{\mathscr R}
\def\mbQ{\mathbb Q}
\def\oB{\overline{B}}

\def\mbM{\mathbb M}
\def\mrN{\mathrm N}
\def\mrp{{\mathrm p}}
\def\mrq{{\mathrm q}}
\def\mrL{{\mathrm L}}
\def\mcL{{\mathcal L}}

\def\mbp{{\mathbf p}}
\def\mbq{{\mathbf q}}
\def\msI{{\mathscr I}}
\def\mcI{{\mathcal I}}
\def\mbV{{\mathbb  V}}
\def\mbL{{\mathbb L}}
\def\mrd{{\mathrm d }}
\def\mcZ{{\mathcal Z}}
\def\mbfn{{\mathbf{0}}}

\def\cG{\mathcal G}
\def\cD{\mathcal D}

\def\cO{\mathcal O}
\def\mbT{\mathbb T}
\def\cM{\mathcal M}
\def\cMb{\mathcal M_b}
\def\rodd{\rm odd}
\def\reven{\rm even}
\def\Htwoin{ H^2_{\rm in}}
\def\muckmuck{distribution }
\def\muckmucks{distributions }
\def\kpp{}

\newcommand{\beq}{\begin{equation}}
\newcommand{\eeq}{\end{equation}}
\newcommand{\beqs}{\begin{equation*}}
\newcommand{\eeqs}{\end{equation*}}
\newcommand{\h}{\hspace}
\newcommand{\normmm}[1]{{\left\vert\kern-0.15ex\left\vert\kern-0.15ex\left\vert #1 
    \right\vert\kern-0.15ex\right\vert\kern-0.15ex\right\vert}}

\newcommand{\dd}{\, \mathrm{d}}
 
\newcommand{\ep}{\epsilon}
\newcommand{\varep}{\varepsilon}
\def\eps{\varepsilon}
\newcommand{\p}{\partial}
\newcommand{\la}{\lambda} 
\newcommand{\La}{\Lambda}

\newtheorem{thm}{Theorem}[section]
\newtheorem{lemma}[thm]{Lemma}
\newtheorem{cor}[thm]{Corollary}
\newtheorem{remark}[thm]{Remark}
\newtheorem{prop}[thm]{Proposition}
\newtheorem{defn}[thm]{Definition}

\newtheorem{notation}{Notation}[section]

\begin{document}
\title{\textbf{Manifolds of Amphiphilic Bilayers: Stability up to the Boundary} }
\author{Yuan Chen}
\address{Department of Mathematics, Michigan State University, East Lansing, MI 48824, U.S.}
\curraddr{}
\email{chenyu60@msu.edu}

\author{Keith Promislow}
\address{Department of Mathematics, Michigan State University, East Lansing, MI 48824, U.S.}
\email{promislo@msu.edu}
\thanks{K.P. acknowledges support from NSF grants DMS-1409940 and DMS-1813203}


\date{\today}

\dedicatory{}


\begin{abstract} 

We consider the mass preserving $L^2$-gradient flow of the strong scaling of the functionalized Cahn Hilliard gradient flow and establish the nonlinear stability of a manifold comprised of quasi-equilibrium bilayer \muckmucks up to the manifold's boundary. In the limit of thin but non-zero interfacial width, $\varep\ll1,$ the bilayer manifold is parameterized by meandering modes that describe the interfacial evolution and  ``pearling'' modes that control the structure of the profile near the interface. The pearling modes are weakly damped and can lead to the dynamic rupture of the interface. Amphiphilic interfaces can lengthen to decrease energy. We introduce an implicitly defined parameterization of the interfacial shape that uncouples this growth from the parameters describing the shape and introduce a nonlinear projection onto the manifold from a surrounding neighborhood. The bilayer manifold has asymptotically large but finite dimension tuned to maximize normal coercivity while preserving the wave-number gap between the meandering and the pearling modes. Modulo a pearling stability assumption, we show that the manifold attracts nearby orbits into a tubular neighborhood about itself so long as the interfacial shape remains sufficiently smooth and far from self-intersection. In a companion paper, \cite{CP-nonlinear}, we identify open sets of initial data whose orbits converge to circular equilibrium after a significant transient, and derive a singularly perturbed interfacial evolution comprised of motion against curvature regularized by an asymptotically weak Willmore term.   

\end{abstract}
\maketitle


\section{Introduction}

\noindent

The functionalized Cahn-Hilliard (FCH) free energy models the free energy of mixtures of amphiphilic molecules and solvent. Amphiphilic
molecules are formed by chemically bonding two components whose individual interactions with the solvent are energetically favorable and unfavorable, respectively. 
When blended with the solvent,  amphiphilic molecules have a propensity to phase separate, forming thin amphiphilic rich domains that are generically the thickness of two molecules in at least one direction. 
On a periodic domain $\Omega\subset\mbR^2$ the FCH free energy is given in terms of the volume fraction $u-b_-$ of the amphiphilic molecule 
\begin{equation}\label{FCH}
\mathcal F(u):=\int_\Omega \frac{\varep}{2} \left(\Delta u-\frac{1}{\varep^2}W'(u)\right)^2-\varep^{p-1}\left(\frac{\eta_1}{2}|\nabla u|^2+\frac{\eta_2}{\varep^2} W(u)\right) \mrd x,
\end{equation}
where $W:\mathbb R\mapsto\mathbb R$ is a smooth tilted double well potential with local minima at $u=b_{\pm}$ with $b_-<b_+$, $W(b_-)=0>W(b_+),$ and $W''(b_-)>0$.  The state $u\equiv b_-$ corresponds
to pure solvent, while $u\equiv b_+$ denotes a maximum packing of amphiphilic molecules. The system parameters $\eta_1>0$ and $\eta_2$ 
characterize key structural properties of the amphiphilic molecules. The small positive parameter $\varep\ll 1$ characterizes the ratio of the length of the molecule to the domain size,  and $p=1$ or $2$ selects a balance between the Willmore-type residual of the dominant squared term and the amphiphilic structure terms.  We select the strong scaling $p=1$, in which the amphiphilic structure terms dominate the Willmore residual.  The FCH energy  was introduced in \cite{GHPY-11}, motivated by the work of Gommper \cite{GS-90, GK-93a, GK-93b}. In particular the form of the energy in \cite{GG-94} corresponds to the FCH with $\eta_1=0$ and 
$\eps^p\eta_2=-f_0$, where $f_0$ is a key bifurcation parameter.  
A central feature of the functionalized Cahn-Hilliard  energy  \eqref{FCH} is that its approximate minima include vast families of saddle points of a Cahn-Hilliard type energy. Within the FCH the competitors for minima include codimension one bilayers, codimension two pores, and codimension three micelles that are the building blocks of many biologically relevant organelles.

 The chemical potential of  $\mathcal F$, denoted $\mathrm F=\mrF(u)$, is a rescaling of its variational derivative
\beq
\label{eq-FCH-L2-p}
\mathrm F(u):= \varep^3\frac{\delta \mathcal F}{\delta u}=(\varep^2\Delta -W''(u))(\varep^2\Delta u-W'(u))+\varep^p (\eta_1\varep^2\Delta u-\eta_2 W'(u)).
\eeq
We take the strong, $p=1$, scaling of the FCH and consider the mass-preserving $L^2$ gradient flow
 \begin{equation}\label{eq-FCH-L2}
\p_t u=-\Pi_0  \mathrm F(u),
\end{equation}
subject to periodic boundary conditions on $\Omega\subset \mathbb R^2$. Here $\Pi_0$ is the zero-mass projection given by
\begin{equation}
 \Pi_0 f:=f-\left\langle f\right\rangle_{L^2}, 
\end{equation}
where we have introduced the averaging operator
\beq
\label{def-massfunc}
\left\langle f\right\rangle_{L^2}:=\frac{1}{|\Omega|}\int_\Omega f\dd x.
\eeq

We consider the mass-preserving $L^2$ gradient flow of the strong scaling of the FCH free energy, \eqref{eq-FCH-L2-p}-\eqref{eq-FCH-L2}, and construct a bilayer manifold, $\cM_b$, with boundary contained in $H^2(\Omega)$ that is comprised of quasi-equilibrium of the system, called bilayer \muckmucks\!\!. Each bilayer \muckmuck is associated to an immersed interface in $\Omega$, and varies predominantly through the $\varep$-scaled signed distance to that interface.  The bilayer manifold has a nonlinear projection that maps an open neighborhood of the bilayer manifold onto itself and decomposes functions $u$ in the open neighborhood of the bilayer manifold into a point on the manifold (a bilayer \muckmuck\!\!) and a perturbation that is orthogonal to the tangent plane of $\cM_b$. The bilayer \muckmuck is parameterized by a finite but asymptotically large set of ``meander modes'' that characterize the shape of its associated interface and a single bulk density parameter that characterizes the excess amphiphilic mass in the bulk. The orthogonal perturbation is further decomposed, through a linear projection, into an asymptotically large but finite dimensional set of ``pearling modes'' and an infinite dimensional set of ``fast modes.'' The pearling modes modify the internal structure of the bilayer \muckmuck near its interface and are weakly damped, subject to a pearling stability condition. The fast modes are uniformly damped under the flow.  The meander modes perturb the shape of a predefined base interface, so that $\cM_b$ accommodates interfaces whose range of shapes is independent of $\varep.$

The bilayer manifold is defined as a graph over a bounded domain of meander modes. We show that initial data that start asymptotically close  to $\cM_b$ will remain close unless the meander modes become sufficiently large that the they hit the boundary of the domain. This domain is selected to insure that the associated interfaces do not self-intersect, that their curvatures remain uniformly bounded, independent of $\eps$, and that the pearling stability condition holds uniformly. Establishing the stability of the manifold up to the meander mode boundary requires two classes of sharp bounds. The first are upper bounds on the coupling of the evolution of the interfacial geometry, characterized by the meander modes, upon the pearling and the fast modes. The second are lower bounds on the coercivity of the second variation of the energy evaluated at points on the manifold when restricted to act on the pearling and meander spaces. 

In a companion paper, \cite{CP-nonlinear}, we rigorously analyze the evolution of the interfaces. In particular we identify an open set of initial data of \eqref{eq-FCH-L2} whose projection defines interfaces that are sufficiently close to circular, and show that the evolving curvature of the interfaces satisfies the curvature bounds for all $t>0$, and that after a transient in which the deviation of the interface from circularity may \emph{grow} by an $o(1)$ amount, the interface ultimately converges to a nearly circular equilibrium.  Together these results partially validate the formal  results obtained in \cite{CKP-18}, where the authors applied multiscale analysis to the $H^{-1}$ gradient flow of the strong FCH free energy.   They considered the evolution of bilayer \muckmuck with high density of amphiphilic material that separate bulk regions of low density via a codimension one interface $\Gamma$.  On the $\varep^{-3}$ time scale they formally derived the evolution of the curvature $\kappa$ of $\Gamma$  
\begin{equation}\label{RCL-flow}
\p_t \kappa=-(\Delta_s+\kappa^2)V,
\end{equation}
in terms of the $\varep$-scaled normal velocity
\begin{equation}\label{RCL-flow-V}
V=(\sigma(t)-\sigma_1^*)\kappa +\varep \Delta_s \kappa + \mathrm {H.O.T.}.
\end{equation}
The normal velocity is proportional to the curvature $\kappa$ through a time-dependent coefficient that can be positive or negative depending upon the initial data.  Here $\Delta_s$ is the Laplace-Beltrami  operator associated to $\Gamma$, and for simplicity we have omitted positive constants that are independent of time and $\ep$. 
The bulk density parameter $\sigma=\sigma(t)$ controls the spatially constant density of amphiphilic material in the bulk. This couples strongly to the length of the interface in a relation that is determined by conservation of mass. The critical value  $\sigma_1^*$ , is a constant depending only upon the system parameters $\eta_1, \eta_2,$ and the well $W$.  When the bulk density is above this critical value, $\sigma>\sigma_1^*$, the interface absorbs mass from the bulk, and moves {\sl against} curvature,  in a singularly perturbed meandering or buckling motion that is regularized by the higher order diffusion, $\Delta_s$. This regime is called regularized curve lengthening (RCL), and the weak surface diffusion plays an essential role in the local existence. Conversely, when $\sigma<\sigma_1^*$, the interface releases mass to the bulk and contracts under a mean curvature driven flow (MCF). In both cases the flow drives $\sigma$ towards $\sigma_1^*$ and the curve attains an equilibrium length set by the mass of the initial data.

In the absence of a maximum principle for the fourth-order system \eqref{eq-FCH-L2}, we use energy estimates and modulation methods. 
The modulation methods for extended manifolds, \cite{ZM, GK-20} consider the linearization of the flow about points on the manifold and establish lower bounds on 
decay rates based upon coercivity estimates of the linearization restricted to subspaces that are approximately tangent and approximately normal to the manifold. 
As a form of normal hyperbolicity they require a spectral gap between eigenvalues associated to the tangent plane of the manifold, the slow modes, and those associated to the normal direction, the fast modes. These results refine earlier estimates in \cite{DHPW-14, HP-15, NP-17}, which introduced the slow space comprised of pearling and meander modes. In particular \cite{NP-17}, conducted spectral analysis of the linearized operators restricted to the slow space and,  modulo the pearling stability condition, used a slow space with dimension  $O(\eps^{-1/2})$ to establish a fast space coercivity that scales with $\sqrt{\varep}$. However, these results are too rough to close our nonlinear estimates.  

Our analysis requires two significant modifications. First, the ansatz that defines the  manifold is constructed implicitly in the parameters that define the shape of the interface. A single parameter controls interfacial length, uncoupling those that control the shape and making the natural basis modes of the tangent plane of the ansatz substantially closer to orthogonal. This allows us to extend the size of the slow spaces while preserving the diagonal dominance of the correlation matrix obtained from restricting the linearization of the FCH equation to the slow space. We combine the implicit ansatz with higher-order corrections to the slow space to build the improved estimates for the slow-fast, and pearling-meander coupling. We enlarge the dimension of the slow space, to $O(\rho\eps^{-1})$, where the spectral cut-off $\rho\ll1$ is independent of $\eps.$ This yields a fast space coercivity that scales with $\rho$ and is independent of $\eps$. Indeed, the choice of the size of the slow space requires a delicate balance. A larger slow space allows for stronger coercivity of the fast space. However the pearling modes have asymptotically short in-plane wave-length, while the meander modes are relatively long in-plane wave-length. \emph{The asymptotically large gap between the in-plane wave lengths of the pearling and meander modes decreases the strength of their coupling by one order of magnitude in $\eps$}. 


It is illuminating to compare the estimates derived here for the bilayer manifold to the classic results that establish rigorous results for 
front evolution in the scalar Cahn-Hilliard (CH) equation, such as  \cite{ABC-94, AF-03, AFK-04}. The bilayer \muckmucks are not fronts. An immediate distinction is that the limiting curve motion for the FCH is singularly perturbed and ill-posed in the $\eps\to0$ limit, while the CH interfacial motion is locally well posed in this sharp interface limit.  This requires us to fix 
$\eps>0$ small but nonzero, and to perform detailed analysis in the regions near the interface.  A second distinction is that the bilayer manifold has asymptotically weak relaxation rates to perturbations that incite the pearling modes that regulate the width of the interface. The coercivity associated to fronts in CH is uniform with respect to $\eps$. For the FCH the pearling modes have the capacity to destabilize the interfacial structure, modulating its width to the point that it can perforate. Their amplitude must be controlled through tight bounds on the coupling between the front evolution and the pearling modes. There is no analogue for these structural dynamics of the interface within the fronts of the scalar CH models.

The singular nature of the interface motion and the weak damping of the internal pearling modes generate significant technical obstacles whose resolution requires restrictions. The most striking of these is that the interfaces must be sufficiently close to a base point interface $\Gamma_0$ that is far from self intersection. Self intersection in the RCL regime is a real possibility.  Numerical benchmark calculations have identified bulk parameter values for initial data that initiate the formation of defects within the bilayer \muckmucks\!\!, \cite{Bench-20}. These results show that an initial bulk density state $\sigma$ that deviates from the scaled equilibrium  $\sigma_1^*$ by an $O(1)$ amount can lead to defect formation, suggesting that the restriction $|\sigma-\sigma_1^*|\ll1$ we require here, see (\ref{assump-Ap}) and Lemma\,\ref{lem-sigma}, is not far from optimal. 

The remainder  of this article is organized as follows. In Section \ref{sec-pre}, we present the local coordinates and estimates on the variation of the interface through the meander parameters. In particular we introduce the implicitly defined perturbed interfaces $\Gamma_\mbp$ in Definition\,\ref{def-P-interface} and show that they are well posed in Lemma\,
\ref{lem-def-gamma-p}. In Lemma\,\ref{lem-def-Phi-p} of Section \ref{sec-profile} we construct the quasi-equilibrium bilayer \muckmucks $\Phi_\mbp$ as the dressing (Definition\,\ref{def-dressing}) of the perturbed interface $\Gamma_{\mathbf p}$, and estimate their residual $\mrF(\Phi_\mbp)$. The bilayer  manifold $\cMb$ is presented in Definition\,\ref{def-bM0} as the graph of the map $\mbp\mapsto \Phi_\mbp(\cdot;\sigma(\mbp))$ with the bulk density parameter $\sigma$ slaved to constrain the mass of $\Phi_\mbp.$  Section \ref{sec-linear}  introduces the slow space in Definition\,\ref{def-slow-space} and the spectral cut-off parameter 
$\rho$. Tbe modified slow space is presented in Lemma\, \ref{lem-def-Z*}, followed by a characterization of the spectrum of the operator  $\Pi_0\mbL$ arising from the linearization of the flow (\ref{eq-FCH-L2-p}) about $\Phi_\mbp.$  In Lemma\,\ref{lem-eigen-bM} we establish the $O(\eps)$ weak coercivity of the pearling spaces modulo the pearling stability condition and in Theorem\,\ref{thm-coupling est} we give sharp bounds on the pearling-meander and fast-slow coupling required for closure of the nonlinear estimates. In Theorem\,\ref{lem-coer} we establish the strong coercivity for the fast modes in terms of the spectral cut-off parameter $\rho$. In Section\,\ref{s:nonstab-BLM}, we define the bilayer manifold that includes the pearling modes, and the nonlinear projection onto the manifold. It concludes with the main result, Theorem\,\ref{thm:Main}, which establishes the nonlinear stability of the bilayer manifold  up to its boundary. We emphasize that there are three small parameters used in this work, $\varep_0, \rho, \delta>0.$ The first, $\varep_0$ sets the upper bound on the size of the dominant small parameter $\varep.$ The spectral parameter $\rho>0$ controls the dimension of slow spaces, while $\delta$ is a technical parameter used to close the nonlinear estimates. We first fix $\delta$ sufficiently small in Lemma\,\ref{lem-Manifold-Projection}, and then $\rho$ sufficiently small in Theorem\,\ref{thm-coupling est} and Lemma\,\ref{lem-IC-est}. The value $\varep_0$ in set in terms of these fixed values in Theorem\,\ref{thm:Main}.

The companion paper \cite{CP-nonlinear} establishes the unconditional stability of the bilayer manifold built form a circular base point interface, and recovers the evolution of the meander modes and associated interfacial motion, as well as the scope of the transient and the rate of convergence to equilibrium. 

\subsection{Notation}\label{ssec-Notation} We present some general notation. 
\begin{enumerate}
\item  The symbol $C$ generically denotes a positive constant whose value depends only on the system parameters $\eta_1, \eta_2$, the domain $\Omega$, and 
geometric quantities of initial curve $\Gamma_0$. In particular its value is independent of $\varep, \rho,$ and $\delta$, so long as they are sufficiently small. 
The value of $C$  may vary line to line without remark.  In addition, $A\lesssim B$ indicates that quantity $A$ is less than quantity $B$ up 
to a multiplicative constant $C$ as above, and $A\sim B$ if $A\lesssim B$ and $B\lesssim A$.  The notation $A\wedge B$ denotes the minimum of $A$ and $B$.   The expression $f=O(a)$  indicates the existence of a constant $C$, as above,
and a norm $|\cdot|$ for which
\begin{equation*}
| f| \leq C |a|.
\end{equation*}
We also use $f=O(a,b)$ for the case $f\leq C|a|+C|b|$. 
\item The quantity $\nu$ is a positive number, independent of $\varep$, that denote an exponential decay rate. It may vary from line to line.
\item If a function space $X(\Omega)$ is comprised of functions defined on the whole spatial domain $\Omega$, we will drop the symbol $\Omega$.
\item We use $\bm 1_{E}$ as the characteristic function of an index set $E\subset \mathbb N$,  i.e. $\bm 1_{E}(x)=1$ if $x\in E$;  $\bm 1_{E}(x)=0$ if $x\notin E$.
We denote the usual Kronecker delta by 
\beqs
\delta_{ij}=\left\{\begin{array}{ll} 1,& i=j \\
                                                       0, & i\neq j
                                                       \end{array} \right.
                                                       \eeqs
\item For a vector $\mathbf q=(\mathrm q_j)_j$,  we denote the norms
\begin{equation*}
\|\mathbf q\|_{l^k}=\left(\sum_j  |\mathrm q_j|^k\right)^{1/k}, \qquad 
 \hbox{for $k\in \mathbb N^+$}, 
 \end{equation*}
 and  $\|\mathbf q\|_{l^\infty}=\max_j |\mathrm q_j|.$ For a matrix $\mbQ=(\mbQ_{ij})_{ij}$ as a map from $l^2$ to $l^2$ has operator norm $l^2_*$ defined by
 \begin{equation*}
 \|\mbQ\|_{l^2_*}=\sup_{\{\|\mbq\|_{l^2=1}\}} \|\mathbb Q \mbq\|_{l^2},\quad
 \hbox{for $k\in \mathbb N^+$}, 
 \end{equation*}
We write
\beqs \mrq_j = O(a)e_j,\quad \mbQ_{ij}=O(a) \mbE_{ij},\eeqs
where ${\bm e}=(e_j)_j$ is a vector with $\|{\bm e}\|_{l^2}=1$ or ${\mbE}$ is a matrix with operator norm $\|\mbE\|_{l^2_*}=1$ 
to imply that $\|\mbq\|_{l^2} = O(a)$ or $\|\mbQ\|_{l^2_*}=O(a)$ respectively. See \eqref{def-e-i}-\eqref{est-e-i,j} of Notation\,\ref{Notation-e_i,j} for usage.
\item  The matrix $e^{\theta\mathcal R}$ denotes rotation through the  angle $\theta$ with the generator $\mathcal R$.  More explicitly, 
\begin{equation*}\mathcal R=\left(\begin{array}{cc}0& -1\\1&0\end{array}\right), \quad e^{\theta\mathcal R}= \left(\begin{array}{cc}\cos \theta& -\sin \theta\\\sin \theta &\cos \theta \end{array}\right).
\end{equation*}
\end{enumerate}

\section{Coordinates and preliminary estimates}\label{sec-pre}
In this section we recall the local coordinates associated to a general smooth interface and use them to define a finite dimensional family of perturbations of the interface. In particular we establish bounds controlling the variation of the interface in terms of the parameters.  
\subsection{ The local coordinates} 
\label{ss-wc}
We consider a closed, smooth, non-intersecting curve  $\Gamma\subset \mathbb R^2$ which divides $\Omega=\Omega_+\cup \Omega_-$ into an exterior $\Omega_+$ and an interior $\Omega_-$.   The interface $\Gamma$ is given parametrically as 
\begin{equation*}
\Gamma=\big\{\bm \gamma(s): s\in \mathscr I \subset \mathbb R \big\},
\end{equation*}
with tangent vector $\mathbf T(s)\in \mathbb R^2$ 
\begin{equation}\label{re-T-gamma}
\mathbf T(s):= \bm \gamma'/|\bm \gamma'|.
\end{equation}
Denoting the outer normal to $\Gamma$ by $\mathbf n(s)$ we have the relations
\begin{equation}\label{re-T-n}
 \mathbf T'=\kappa \mathbf n, \quad  \mathbf n'=-\kappa \mathbf T,
\end{equation}
where $\kappa=\kappa(s)$ is the curvature of $\Gamma$ at ${\bm \gamma}(s)$. 
By the implicit function theorem, there exists an open set $\mathcal N$ containing $\Gamma$ such that for each $x\in \mathcal N$ 
we may write
\begin{equation}\label{coord-wh}
  x=\bm \gamma(s)+r\mathbf n(s) 
\end{equation}
where $r=r(x)$ is the well-known signed distance of the point $x$ to the curve $\Gamma$ and $s=s(x)$ is determined by the choice of parameterization $\bm \gamma.$  In this neighborhood, we define the scaled signed distance  $z=r/\varep$ and  ``whiskers'' of length $\ell$: 
\begin{equation*}
w_{\ell}(s):=\big\{\bm \gamma(s)+\varep z \mathbf n(s): r\in [-\ell, \ell]\big\},
\end{equation*}
and the $\ell$-reach of $\Gamma$, 
\begin{equation}
\label{def-reach}
\Gamma^\ell=\bigcup_{s\in \msI } w_\ell (s). 
\end{equation}
In the following, $(z,s)$ will be referred as the local coordinates near $\Gamma$.  

\subsubsection{Dressings}

We say that a curve $\Gamma$ is $\ell$-far from self intersection if none of the whiskers of length $\ell$ intersect each other nor with $\p\Omega$, and if the set $\Gamma^\ell$ contains all points of $\Omega$ whose distance to $\Gamma$ is at most $\ell.$
 We introduce the following class of curves and show they have a uniform distance from self-intersection.

\begin{defn}
\label{def-GKl}
Given $K, \ell>0$ the class $\cG_{K,\ell}^k$ consists of those curves $\Gamma$ whose parameterization $\bm\gamma$ satisfies (a) $\|\bm\gamma\|_{W^{k,\infty}(\mathscr I)}\leq K$ and (b) if points two points on $\Gamma$ satisfy $|s_1-s_2|_\msI>1/(2K)$ then $|\bm \gamma(s_1)-\bm\gamma(s_2)|>\ell$. Here $|\cdot|_{\msI}$ denotes the periodic distance 
$|s|_\msI=\min\left\{\Big|s- |\msI|k\Big|\, : k\in \mathbb Z\right\}.$
\end{defn}

\begin{lemma}
If $\tilde\ell<\ell$ then $\cG^{k}_{K,\tilde \ell}\subset \cG^{k}_{K,\ell}$. If $\ell\leq\sqrt{2}\pi/K$ then every curve in $\cG_{K,\ell}^{2}$ is $\ell$-far from self-intersection.
\end{lemma}
\begin{proof}
The first statement follows directly from the definition of $\cG_{K,\ell}^k.$ 
Pick $\Gamma\in\cG_{K,\ell}^{2}$ with parameterization $\bm \gamma$. The reach $\Gamma^\ell$ contains the set of points whose distance to $\Gamma$ is less than or equal to $\ell$. Indeed, if $x\in\Omega$ lies within $\ell$ of $\Gamma$, then there exists a least one point $s\in\mathscr I$ such that $\bm\gamma(s)$ is the closest point on $\Gamma$ to $x$. Since the tangent $\mathbf T$ has a smooth derivative, it follows that $(x-\bm\gamma(s))\cdot \mathbf T(s)=0.$ If not, then $(|x-\bm\gamma(s)|^2)^\prime\neq0$, which contradicts $\bm\gamma(s)$ being the closest point on $\Gamma$ to $x$. Consequently $x\in w_\ell(s)\subset\Gamma^\ell.$ To see that the whiskers of length $\ell$ do not intersect consider two points $s_1,s_2\in\mathscr I$. Since $ | \mathbf T(s)|=1$ and 
$$\left|\frac{\dd}{\dd s}  \mathbf T(s_1)\cdot \mathbf T(s)\right|\leq K,$$
we deduce that $|s_1-s_2|_\msI<1/(2K)$ implies $ \mathbf T(s_1)\cdot \mathbf T(s_2)>\frac12.$ It follows from planar geometry that \beq\label{e:kp1}
|\bm\gamma(s_1)-\bm\gamma(s_2)|\geq \frac{|s_1-s_2|_\msI}{\sqrt{2}}.
\eeq
If the whiskers from $\bm\gamma(s_1)$ and $\bm\gamma(s_2)$ intersect at a distance $\ell$ from $\Gamma$, then $\bm\gamma(s_1)$ and $\bm\gamma(s_2)$ lie on a circle of radius $\ell$. In particular the straight line distance between these two points must be less than the arc-length along that circle. If $\nu(s)$ denotes the angle of $\mathbf T(s)$ to the horizontal, then this distance inequality implies
$$ \frac{\ell}{2\pi} |\nu(s_1)-\nu(s_2)| > |\bm\gamma(s_1)-\bm\gamma(s_2)|.$$
Since $\nu^\prime=\kappa$, we have the Lipschitz estimate  $|\nu(s_1)-\nu(s_2)|\leq K|s_1-s_2|_\msI,$ and combining these estimates and dividing by $|s_1-s_2|_\msI$ yields the bound
$$\ell\geq \frac{\sqrt{2}\pi}{K}.$$
This shows that the whiskers of two points whose arc-length distance is less than $1/(2K)$ can only intersect at a distance of at least $\frac{\sqrt{2}\pi}{K}$ from $\Gamma.$  However since $\Gamma\in\cG_{K,\ell}^{2}$, if $|s_1-s_2|_\msI>1/(2K)$ then by condition (b) the whiskers through $\bm \gamma(s_1)$ and $\bm \gamma(s_2)$ of length $\ell$ cannot intersect as the points $\bm \gamma(s_1)$ and $\bm\gamma(s_2)$ are more than $2\ell$ apart. By assumption $\ell\leq\min\{\frac{\sqrt{2}\pi}{K},\ell\}$, and we deduce that $\Gamma$ is $\ell$-far from self intersection.
\end{proof}

\begin{defn}[Dressing]\label{def-dressing}
Fix a smooth cut-off function $\chi:\mathbb R\rightarrow \mathbb R$ satisfying: $\chi(r)=1$ if $r\leq 1$ and $\chi(r)=0$ if $r\geq 2$. Given an interface 
$\Gamma$ which is $2\ell$-far from self intersection and a smooth function $f(z): \mathbb R\rightarrow \mathbb R$ which tends to a constant $f^\infty$ and whose
derivatives tend to zero at an $\varep$ independent exponential rate as $z\rightarrow \pm \infty$, then we define the dressed function, $f^d:\Omega\mapsto \mathbb R$, of $f$ with respect to $\Gamma$ as
\begin{equation*}
f^d(x)=  f(z(  x)) \chi(\varep |z(  x)|/\ell)+f^\infty(1-\chi(\varep |z(  x)|/\ell)).
\end{equation*} 
\end{defn}
From this definition the dressed function satisfies
\begin{equation*}
f^d (  x)=\left\{\begin{aligned}
f(z(  x)), &\quad \hbox{if $|z(x)|\leq \ell/\varep$};\\
f^\infty,\quad\; &\quad \hbox{if $|z(x)|\geq 2\ell/\varep$}.
\end{aligned}\right.
\end{equation*}

\begin{defn}[Dressed operator]\label{def-Ld}
Let $\mrL: {\mathrm D}\subset L^2(\mbR)\mapsto L^2(\mbR)$  be a self-adjoint differential operator with smooth coefficients whose derivatives of all order decay to zero at an exponential rate
at $\infty.$ We define the space $\mathcal D $ to consist of the functions $f:\mathbb R\mapsto\mathbb R$ as in Definition \ref{def-dressing}, and the 
dressed operator $\mrL_d:\mathrm D\cap\mathcal D \mapsto L^2(\Omega)$ and its $r$'th power, $r\in\mathbb N$,
\beq
\mrL_{d}^r f := (\mrL^r f)^d .
\eeq
If $r<0$ then we assume that $f\in\mathcal R(\mrL)$ and the inverse $\mrL_d^{-1} f$ decays exponentially to a constant at $\pm\infty.$ For simplicity we abuse notation and drop the subscript '$d$' in both the dressed operator and the dressed function where the context is clear. 
\end{defn}

A function $f=f(  x)\in L^1(\Omega)$ is said to be {\it localized} near the interface $\Gamma$ if there exists $\nu>0$ such that for all $  x\in \Gamma^{2\ell}$, 
\begin{equation*}
|f(  x(s, z))|\lesssim e^{-\nu|z|}.
\end{equation*} 

\subsubsection{The Jacobian}
Let $\mathbb J(s, z)$ be the Jacobian matrix with respect to the change to the whiskered coordinates and denote the Jacobian by $\mathrm J(s, z)=\det \mathbb J(s, z)$.  In two dimensions, $\mathbf n'\parallel \bm\gamma'$ so that
\begin{equation*}
\begin{aligned}
&\mathbb J=\Big(\bm \gamma'(s)+\varep z \mathbf n'(s)\quad \varep \mathbf n(s) \Big)=\Big(\bm \gamma'(s)\quad  \mathbf n(s)\Big)^T\left(\begin{array}{cc}
1-\varep z \kappa(s) &0\\
0& \varep
\end{array}\right)
\end{aligned}
\end{equation*}
where the curvature
\beq\label{def-kappa-first}
 \kappa(s)=-\frac{\bm\gamma'(s)\cdot \mathbf n'(s)}{|\bm \gamma'(s)|^2}.
\eeq
We decompose the Jacobian as:  $\mathrm J(s, z)=\tilde{\mathrm J}(s, z)\mathrm J_0(s)$ where 
\begin{equation}
\tilde{\mathrm J}(s, z)
=\varep (1-\varep z\kappa(s)), \quad \mathrm J_0(s)= |\bm \gamma'(s)|. 
\end{equation}  
The metric tensor takes the form $\mathbb G=\mathbb J^T \mathbb J$.
 
If $f, g\in L^2(\Omega)$ have support in $\Gamma^{2\ell}$,  then the usual $L^2( \Omega)$-inner product can be rewritten as
\begin{equation}
\left<f, g\right>_{L^2}=\int_{\mathbb R_{2\ell}}\int_{\msI}  f(s, z) g(s, z) \mathrm J(s, z) \, \mrd s\mrd z. 
\end{equation}
where $\mathbb R_{a}:=[-a, a]$ for $a\in \mathbb R^+$.  If $\tilde s$ denotes the arc-length reparameterization of $\Gamma$ over the interval 
$\msI_{\Gamma}=[0, |\Gamma|]$,  then $\dd \tilde s=\mathrm J_0(s)\dd s$ and the $L^2$-inner product becomes
\begin{equation}
\left<f, g\right>_{L^2}=\int_{\mathbb R_{2\ell}}\int_{\msI_{\Gamma}}  f(s, z) g(s, z) \tilde{\mathrm J}(s, z) \, \mrd \tilde s\mrd z. 
\end{equation}
Moreover, if $f\in L^2$ is localized near the interface $\Gamma$, then
 \begin{equation*}
 \int_\Omega f\dd   x=\int_{\mathbb R_{2\ell}}\int_{\msI} f (  x(s,z))\mathrm J\dd s\mrd z + O(e^{-\nu \ell/\varep}).
 \end{equation*}

\subsubsection{Laplacian} 
The $\varep$-scaled Laplacian can be expressed in the local  coordinates near $\Gamma$ as 
\begin{equation}\label{eq-Lap-induced}
\varep^2 \Delta_{  x}=\mathrm J^{-1}\p_z(\mathrm J \,\p_z)+\varep^2\Delta_g= \p_z^2+\varep \mathrm H \p_z+\varep^2\Delta_g.
\end{equation}
Here $\mathrm H$  is the extended curvature
\begin{equation}\label{def-H}
\mathrm H(s,z):= -\frac{\kappa(s)}{1-\varep z\kappa(s)}=\frac{\p_z \mathrm J}{\varep \mathrm J},
\end{equation}
and $\Delta_g$ is the induced Laplacian under metric tensor $\mathbb G$, which can be decomposed into 
\begin{equation*}
\Delta_{g}:=\frac{1}{\sqrt{\det \mathbb G}} \p_s (\mathbb G^{11} \sqrt{\det \mathbb G}\, \p_s)= \Delta_s+\varep z  D_{s,2},\;\;  \;\; \mathbb G=\left(\begin{array}{cc}
|\bm \gamma'+\varep z\mathbf n'|^2 &0\\
0& \varep^2
\end{array}\right).
\end{equation*}
Here $\mathbb G^{ij}$ denotes the $(i,j)$-component of the inverse matrix $\mathbb G^{-1}$; $\Delta_s$ is the Laplace-Beltrami operator on   the surface $\Gamma$ and $D_{s,2}$ is a relatively bounded perturbation of $\Delta_s$. In particular, since $|\bm \gamma'+\varep z\mathbf n'|=|\bm \gamma'||1-\varep z\kappa|$, we have
\begin{equation}\label{def-Ds}
\Delta_s= \frac{1}{|\bm \gamma'(s)|} \p_s \left(\frac{1}{|\bm \gamma'(s)|}\, \p_s\right), \quad D_{s,2} =a(s, z)\Delta_s  + b(s, z) \p_s,
\end{equation}
where the coefficients $a, b$ have the explicit formulae
\begin{equation}\label{def-a,b}
a(s, z)=(\varep z)^{-1} \left( \frac{1}{|1-\varep z\kappa|^2}-1\right), \qquad b(s, z)=\frac{(\varep z)^{-1}}{2|\bm \gamma'|^2}\p_s a(s, z).
\end{equation}
\subsection{Perturbed interfaces}
We construct families of interfaces by perturbation from a fixed base curve which we label $\Gamma_0$ with parameterization $\bm \gamma_0$, curvature $\kappa_0(s)$ and length $|\Gamma_0|$.  Without loss of generality we assume that $s$ is the arc length parameterization of $\Gamma_0$ and takes values in $\msI=[0, |\Gamma_0|],$ and  that $\Gamma_0$ is centered at the origin in the sense that the average value of $\bm\gamma_0$ is $(0,0)$.
The effective radius
\beqs
R_0:=\frac{|\Gamma_0|}{2\pi},
\eeqs
forms a natural scaling parameter. 
 The Laplace-Beltrami operator $-\Delta_s: H^2(\msI)\rightarrow L^2(\msI)$ has in-plane wave numbers $\{\beta_i\}_{i=0}^\infty$ whose squares are the scaled eigenvalues of the $|\Gamma_0|$-periodic eigenfunctions $\{\Theta_i\}_{i=0}^\infty$
\begin{equation}\label{def-beta}
-\Delta_s  \Theta_i =\beta_i^2 \Theta_i\big/R_0^2.
\end{equation}
The ground state eigenmode is spatially constant: 
\beq
\label{def-Theta0}
\Theta_0=1/\sqrt{2\pi R_0}, \qquad \beta_0=0;
\eeq 
and for $k\geq 1$, we normalize the eigenmodes in $L^2(\msI)$,
\begin{equation}\label{def-Theta-beta}
\Theta_{2k-1}=\frac{1}{\sqrt{2\pi R_0}}\cos\left(\frac{ks}{R_0}\right),\quad \Theta_{2k}=\frac{1}{\sqrt{2\pi R_0}}\sin \left(\frac{ks}{R_0}\right); \quad \hbox{and}\;\;  \beta_{2k-1}=\beta_{2k}=k.
\end{equation}
To control the smoothness of the perturbed interface we introduce the weighted $\mathbf p$-norms.
\begin{defn}[Weighted $\mathbf p$-norms]\label{def-Vrk}
Given ${N_1}, k>0$, the weighted 
space $\mbV_k^r= \mbV_k^r(N_1)$ is defined on the $N_1$-vectors $\mathbf p=(\mathrm p_0,\cdots, \mathrm p_{N_1-1})^T\in \mathbb R^{N_1}$ as
\begin{equation*}
\|\mathbf p\|_{\mbV_k^r({N_1})}^r:= \sum_{j=0}^{{N_1}-1} \beta_j^{kr} |\mathrm p_j|^r<\infty.
\end{equation*}
The elements of $\mbp$ starting with $\mrp_3$ that control the shape of the interface are denoted 
\beq \label{def-hatp}
\hat{\mathbf p}:=( \mathrm p_3, \cdots , \mathrm p_{{N_1}-1})^T, 
\eeq
and by abuse of notation we apply the same norm to $\hat\mbp$, starting the sum with $j=3.$
When $r=1$, we omit the superscript $r$ and denote the space by $\mbV_k$.
\end{defn}


 The following definition introduces $\Gamma_{\mathbf p}$ the $\mathbf p$-variation of $\Gamma_0$, through an implicit construction that incorporates perturbations so that the change in length of $\Gamma_{\mbp}$ is controlled solely by $\mrp_0$ which scales the effective radius $R_0$. This definition is shown to be well-posed in Lemma\,\ref{lem-def-gamma-p}.
\begin{defn}[Perturbed interfaces]\label{def-P-interface}
Fix a smooth interface $\Gamma_0.$
\begin{enumerate}[(a)]
\item  Given  $\mathbf p\in \mbV_2$, we define the $\mathbf p$-variation of $\Gamma_0$, denoted by $\Gamma_{\mathbf p}$, through the  parametric form: 
\begin{equation}\label{def-gamma-p}
\bm\gamma_{\mathbf p}(s):=\frac{(1+\mrp_0)}{A(\mbp)} \bm\gamma_{\bar p}(s) +\mrp_1\Theta_0 \bE_1 +\mrp_2\Theta_0 \bE_2,  \quad \hbox{for} \quad s\in \msI,
\end{equation}
where $\{\bE_1, \bE_2\}$ are the canonical basis for $\mathbb R^2$, the scaling constant $A(\mbp)$ normalizes the length of $\bm \gamma_{\bar p}$
\beq\label{def-A(p)}
A(\mbp):=|\Gamma_0|^{-1} \int_{\msI} |\bm\gamma_{\bar p}'(s)|\dd s,
\eeq
and the perturbed curve $\bm\gamma_{\bar p}$ is
\beq\label{def-gamma-barp}
\bm\gamma_{\bar p}(s)= \bm \gamma_0 + \bar p(\tilde s) \mbn_0(s)
\eeq
where the vector $\mathbf n_0(s)$ denotes the outer normal vector of $\Gamma_0$ parameterized by  $s$.  The definition is made implicit through the relations
\begin{equation}\label{def-barp}
 \bar p(\tilde s):= \sum_{i=3}^{{N_1}-1} \mathrm p_i \tilde \Theta_i(\tilde s), \qquad \tilde \Theta_i(\tilde s):=\Theta_i\left(\frac{2\pi R_0 \tilde s}{|\Gamma_\mbp|}\right),\end{equation}   
where $\tilde s=\tilde s(s;\mbp)\in \msI_\mbp=[0,|\Gamma_\mbp|]$ is the  arc length parametrization of the perturbed curve $\bm \gamma_\mbp$ solving
\beq \label{def-ts}
\frac{\mrd \tilde s}{\mrd s}=|\bm \gamma_\mbp'|,\qquad \tilde s(0)=0.
\eeq
\end{enumerate}
\end{defn}

\begin{remark}
The parameters $\mathrm p_1$ and $\mathrm p_2$ rigidly translate the interface, and are the only terms
that contribute motion along the tangent to $\Gamma_0$. In particular their normal components recover $\{\Theta_1, \Theta_2\}$,
\begin{equation}\label{proj-n0-E-basis}
\Theta_0 \bE_1\cdot \mathbf n_0=\frac{1}{\sqrt{2\pi R_0}}\cos{\frac{s}{R_0}}=\Theta_1, \quad \Theta_0\bE_2\cdot \mathbf n_0=\frac{1}{\sqrt{2\pi R_0}}\sin\frac{s}{R_0}=\Theta_2.
\end{equation}
\end{remark}
As will be shown in Lemma \ref{lem-def-gamma-p} the curve $|\Gamma_\mbp|$ has length $(1+\mrp_0)|\Gamma_0|$, and the role of $\mrp_0$ is to scale the effective radius $R_\mbp:=(1+\mrp_0)R_0$ of $\Gamma_\mbp$. Indeed it follows from \eqref{def-beta} and \eqref{def-barp} that the arc-length scaled Laplace-Beltrami eigenmodes $\{\tilde\Theta_j\}_{j\geq0}$ of $\Gamma_\mrp$ satisfy
\beq \label{tTheta''}
-\tilde \Theta_j''(\tilde s)=\beta_{\mbp,j}^2 \tilde \Theta_j(\tilde s), \qquad \beta_{\mbp,j}=\frac{ \beta_j}{(1+\mrp_0)R_0},
\eeq 
where we address $^\prime$ of $\tilde \Theta_j$ always  denotes differentiation with respect to $\tilde s$.
The most significant contribution of the rescaling is that it renders the perturbed eigenmodes mutually orthogonal in $L^2(\mathscr I_\mbp)$,  satisfying
\beq\label{ortho-tTheta}
\int_{\msI_{\mbp}}\tilde \Theta_j\tilde\Theta_k|\bm \gamma_\mbp'|\dd  s=\int_{\msI_{\mbp}}\tilde \Theta_j\tilde\Theta_k\dd \tilde s=(1+\mrp_0)\delta_{jk}.
\eeq
The weighted norms are equivalent to usual Sobolev norms of $\bar p$. Indeed the orthogonality \eqref{ortho-tTheta} implies  
\beq\label{re-barp-hatp}
\|\hat{\mathbf p}\|_{\mathbb V_k^2} \sim \|\bar p\|_{H^k(\mathscr I_\mbp)}, \quad \|\bar p^{(k)}\|_{L^\infty(\mathscr I_\mbp)}\lesssim \|\hat{\mathbf p}\|_{\mathbb V_k},
\eeq
where the constants depend only upon the $|\Gamma_\mbp.|$ The following embeddings are direct results of H\"older's inequality and the bound $\beta_j\leq j/2$ for $j\geq 3$, details are omitted. 
\begin{lemma}\label{lem-est-V}
It holds that  
\begin{equation*}
\begin{aligned}
\|\hat{\mathbf p}\|_{\mbV_k}\lesssim  \|\hat{\mathbf p}\|_{\mbV_{k+1}^2}, \quad \|\hat\mbp\|_{\mbV_k}\lesssim  N_1^{1/2}\|\hat\mbp\|_{\mbV_k^2},\quad \|\hat{\mathbf p}\|_{\mbV_{k+1}^r}\lesssim N_1\|\hat{\mathbf p}\|_{\mbV_k^r}.
\end{aligned}
\end{equation*}
In addition, for any vector $\mathbf a \in l^2(\mathbb R^{m})(m\in \mathbb Z^+)$  we have the dimension dependent bound
\begin{equation}\label{est-l1-l2}
\|\mathbf a\|_{l^1}\leq {m}^{1/2}\|\mathbf a\|_{l^2}.
\end{equation}
\end{lemma} 

We assume that the base interface $\Gamma_0\in\cG_{K_0,\ell_0}^4$ with $\ell_0$ sufficiently small that $\Gamma_0$ is $\ell_0$-far from self intersection.  To insure that the implicit construction of $\Gamma_\mbp$ is well posed and the resultant curves are uniformly far from self-intersection we assume that 
the meander parameters $\mbp$ lies in the set
\beq\label{A-00}
\cD_{\delta}:=\left\{ \mbp\in\mbR^{N_1} \bigl |\, \|\hat{\mbp}\|_{\mbV_2}\leq C  ,\; \|\hat{\mbp}\|_{\mbV_1} \leq   C\delta,\; \mrp_0>-1/2\right\}
\eeq
for some positive constant $C\lesssim 1$. The quantity $\delta>0$ will be chosen sufficiently small, depending only upon the system parameters and the choice of  $\ell$ in $\cG^4_{K,\ell}$. 
The lower bound on $\mrp_0$ is chosen to prevent the curve being scaled to a point, any fixed value greater than $-1$ is sufficient.  We assume that $\mbp\in \cD_\delta $ throughout the sequel.
\begin{remark}
Dimension $N_1$ is asymptotically large in this article, in fact, $N_1\leq \varep^{-1}$.  The uniform $\mbV_2$ bound on $\hat\mbp\in\cD_\delta $ implies that 
\beqs
\|\hat\mbp\|_{\mbV_3}\lesssim \varep^{-1}, \qquad \|\hat\mbp\|_{\mbV_4}\lesssim \varep^{-2}. 
\eeqs
This affords finite but asymptotically large bounds on the third and fourth derivatives of $\bm\gamma_\mbp$ in Lemma\,\ref{lem-Gamma-p}. 
\end{remark}
\begin{lemma}\label{lem-def-gamma-p}
Suppose that $\Gamma_0\in\cG_{K_0,\ell_0}^{4}$ for some $K_0,\ell_0>0$. Then for all $\mbp\in\cD_\delta $ the system \eqref{def-ts} defined through \eqref{def-gamma-p} has a unique solution and the resulting interface $\Gamma_\mbp$ is well defined provided that $\delta$ is suitably small in terms of $K_0,\ell_0$. Moreover, the length of the curve $\Gamma_\mbp$ is 
\beq
\label{Gamma-p-length}
|\Gamma_\mbp|=(1+\mrp_0)|\Gamma_0|.
\eeq
\end{lemma}
\begin{proof}
The construction of $\bm\gamma_\mbp$ given in Definition \ref{def-P-interface} requires only that the ODE \eqref{def-ts} is well posed. The issue is that the right-hand side of this expression is implicit in $\tilde s$. To apply the general ODE existence theory we must establish a Lipschitz estimate on $|\bm\gamma_\mbp'|$. From the definition of $\bm \gamma_\mbp$ in \eqref{def-gamma-p}, we take the derivative with respect to $s$, obtaining
\beq\label{gamma'-1}
\bm \gamma_\mbp'=\frac{1+\mrp_0}{A(\mbp)} \left(\bm \gamma_0'+\bar p'(\tilde s) \frac{\dd\tilde s}{\dd s}\mbn_0 +\bar p(\tilde s) \mbn_0'(s) \right).
\eeq  
Here and below, primes of $\bar p$ denote derivatives with respect to $\tilde s$. Recalling \beq\label{bn0-bn0'}
\mathbf n_0'(s)=-\kappa_0\bm \gamma_0'(s),
\eeq
and combining \eqref{gamma'-1} with \eqref{def-ts} implies
\beq \label{est-gamma'}
\bm \gamma_\mbp'=\frac{1+\mrp_0}{A(\mbp)} \left(\bm\gamma_0'+\bar p'(\tilde s) |\bm\gamma_\mbp'| \mbn_0(s)-\kappa_0(s) \bar p(\tilde s)\bm \gamma_0'(s) \right).
\eeq
Since $\bm \gamma_0'$ tangent to $\Gamma_0$ while $\mathbf n_0$ is the outer normal, we have  $\bm \gamma_0'\cdot \mathbf n_0=0 $. Taking the norm of \eqref{gamma'-1}, squaring, expanding, and solving for $|\bm \gamma_\mbp'|$, find the equality
\beq\label{est-|gamma'|-0}
|\bm \gamma_\mbp'|=\frac{1+\mrp_0}{A(\mbp)}\left(1-\kappa_0\bar p(\tilde s) \right)\left(1-\left(\frac{1+\mrp_0}{A(\mbp)}\right)^2|\bar p'(\tilde s)|^2\right)^{-1/2} .
\eeq
Taking derivative with respect to $\tilde s$, and using the weighted-$\mbp$ bounds on the Sobolev norms of $\bar p$ from Lemma\,\ref{lem-est-V} we bound the $L^\infty$ norms of $\bar p$ and its derivatives, to deduce
\begin{equation}
\left|\p_{\tilde s}|\bm \gamma_\mbp'|\right|\lesssim \|\hat \mbp\|_{\mbV_1}(1+\|\hat\mbp\|_{\mbV_2})
\end{equation}
for $1+\mrp_0>1/2, \|\hat\mbp\|_{\mbV_1}<  A^2/(1+\mrp_0)^2$ and $A(\mbp)$ bounded as in the Appendix Lemma \ref{lem-A}. That is, $|\bm \gamma_\mbp'|$ is globally(uniformly) Lipschitz with respect to $\tilde s$ provided that $\hat{\mbp}\in \mbV_2$ satisfying $\|\hat\mbp\|_{\mbV_1}< A^2/(1+\mrp_0)^2 $ and $1+\mrp_0>1/2$. Hence by classical Picard–Lindel\"of existence theory, the system \eqref{def-ts} is solvable  on a  small interval with smallness depending on the Lipschitz constant only.  In addition, by construction the length of $\Gamma_\mbp$ satisfies \eqref{Gamma-p-length} which implies that $\tilde{s}$ is uniformly bounded independent of $\hat{\mbp}$, and the solution is extendable to the whole finite interval $\msI$  for all $\hat{\mbp}\in \mbV_2$ satisfying $\|\hat\mbp\|_{\mbV_1}< A^2/(1+\mrp_0)^2 $ and all $1+\mrp_0>1/2.$

\end{proof}

The following Lemma establishes uniform bounds on the smoothness and distance from self-intersection of the interfaces $\Gamma_\mbp$.
\begin{lemma}[Smoothness of $\Gamma_{\mathbf p}$]\label{lem-Gamma-p}  Suppose that $\Gamma_0\in\cG_{K_0,\ell_0}^{4}$ for some $K_0>0$ and $\ell_0\in(0,\sqrt{2}\pi/K_0)$.  Then there exist $K,\ell>0$ and $\delta$ suitably small depending on $\Gamma_0$, independent of $\varep>0$ such that for all $\mbp\in\cD_\delta $  the associated $\Gamma_{\mbp}$ resides in $\cG_{K,\ell}^{2}$ and is $\ell$-far from self-intersection.
Moreover the perturbed curves $\bm\gamma_\mbp$ satisfy the bounds 
\beq\label{est-de-gamma-p}
|\bm \gamma_\mbp^{(k)}|\lesssim 1+\sum_{l=1}^k|\bar p^{(l)}(\tilde s)|\lesssim 1+\|\hat\mbp\|_{\mbV_k}, \qquad k=1,2, \cdots 4. 
\eeq
The curvature and normal of $\Gamma_{\mathbf p}$, defined by 
\begin{equation}\label{def-kappa}
\kappa_{\mathbf p}:=\bm \gamma_{\mathbf p}''\cdot \mathbf n_{\mathbf p}/|\bm\gamma'_{\mathbf p}|^2, \qquad  \mathbf n_{\mbp}=e^{-\pi \mathcal R/2}\bm \gamma'_{\mathbf p}\big/|\bm \gamma'_{\mathbf p}|,
\end{equation}
admit the bounds
\begin{equation}
\label{eq-kappa-est}
 |\mathbf n_{\mathbf p}|\lesssim 1+\|\hat\mbp\|_{\mbV_1}; \qquad |\kappa_{\mathbf p}|\lesssim1+\|\hat{\mathbf p}\|_{\mbV_{2}}; \qquad \|\kappa_\mbp\|_{H^k(\msI_\mbp)}\lesssim 1+\|\hat\mbp\|_{\mbV_{k+2}^2}
\end{equation}\
for $k=1,2$. 
\end{lemma}
\begin{proof} 
We first establish the bounds on $\bm\gamma_\mbp$ and its curvatures. From \eqref{est-|gamma'|-0}, we have upper and lower bound of the metric $|\bm \gamma_\mbp'|$
\beq\label{Ubd-de-gamma-0}
\frac{1}{2} \leq |\bm \gamma_\mbp'|\leq 2(1+|\mrp_0|),
\eeq
 provided that $\|\hat\mbp\|_{\mbV_1}$ is suitably small. This further implies that the first derivative of the metric has the bound
 \beq\label{Ubd-de-|gamma'|-0}
 ||\bm \gamma_\mbp'|'| =  \frac{|\bm \gamma_\mbp' \cdot \bm \gamma_\mbp''|}{|\bm \gamma_\mbp'|} \lesssim |\bm \gamma_\mbp''|.
 \eeq
The higher derivatives of the metric $|\bm \gamma_\mbp'|$ enjoy the bounds
 \beq\label{Ubd-de-|gamma'|}
 \left||\bm \gamma_\mbp'|^{''} \right|\lesssim |\bm \gamma_\mbp''|^2 +|\bm \gamma_\mbp'''|, \qquad \left||\bm \gamma_\mbp'|'''\right| \lesssim |\bm \gamma_\mbp^{(4)}| + |\bm \gamma_\mbp''| |\bm \gamma_\mbp'''|. 
 \eeq
Moreover the definition \eqref{def-gamma-p} of $\bm\gamma_\mbp(s)$ with $\bar p=\bar p(\tilde s)$ and $\tilde s=\tilde s(s)$, and the smoothness of $\Gamma_0$ imply
 \beq\label{Ubd-de-gamma}
 \begin{aligned}
 & |\bm \gamma_\mbp''|\lesssim 1+|\bar p''| +|\bar p'|\cdot||\bm \gamma_\mbp'|'| ; \qquad |\bm \gamma_\mbp'''|\lesssim 1 +|\bar p'''|+ |\bar p'|\cdot||\bm \gamma_\mbp'|''| +|\bar p''|\cdot|  |\bm \gamma_\mbp'|'|; \\
 &|\bm \gamma_\mbp^{(4)}|\lesssim 1+|\bar p'''|+|\bar p^{(4)}|  +|\bar p'|||\bm \gamma_\mbp'|'''| +|\bar p''| \left(||\bm \gamma_\mbp'|'|^2 +||\bm \gamma_\mbp'|''|\right) +|\bar p'''|\cdot ||\bm \gamma_\mbp'|'|,
 \end{aligned}
 \eeq
 provided that $\hat\mbp \in \mbV_2$.  Combining the second estimate in \eqref{Ubd-de-gamma} with the estimate \eqref{Ubd-de-|gamma'|-0} yields the bound
 \beq\label{Ubd-de-gamma-1}
 |\bm \gamma_\mbp''|\lesssim 1+|\bar p''(\tilde s)|
 \eeq
 for $\|\hat\mbp\|_{\mbV_1}$ suitably small. In a similar manner, combining the last two estimates of \eqref{Ubd-de-gamma} with \eqref{Ubd-de-|gamma'|-0}-\eqref{Ubd-de-|gamma'|} and \eqref{Ubd-de-gamma-1} yields 
 \beq\label{Ubd-de-gamma-2}
 |\bm \gamma_\mbp'''|\lesssim 1+ |\bar p'''(\tilde s)|, \qquad |\bm \gamma_\mbp^{(4)}|\lesssim 1+ |\bar p^{(4)}(\tilde s)| +|\bar p'''(\tilde s)|
 \eeq
 for $\hat \mbp \in \mbV_2$. Now the curvature $\kappa_\mbp$ in \eqref{def-kappa}  admits bound
 \beqs
 |\kappa_\mbp|\lesssim |\bm \gamma_\mbp''|\lesssim 1 +|\bar p''(\tilde s)|, 
 \eeqs
 so that the $L^\infty$ and $L^2(\msI_\mbp)$ bounds of the curvature follow 
 from \eqref{re-barp-hatp}. Taking the derivative of the curvature and using the bounds \eqref{Ubd-de-gamma-0}, \eqref{Ubd-de-gamma-1} and \eqref{Ubd-de-gamma-2} with $\hat\mbp\in \mbV_2$ implies
 \beqs
 |\kappa_\mbp'| =\left|  \frac{e^{-\pi \mathcal R/2}\bm \gamma_\mbp''\cdot \bm \gamma_\mbp'' +e^{-\pi \mathcal R/2}\bm \gamma_\mbp''\cdot \bm \gamma_\mbp''' }{|\bm \gamma_\mbp'|^3} -3 \frac{e^{-\pi \mathcal R/2}\bm \gamma_\mbp''\cdot \bm \gamma_\mbp''}{|\bm\gamma_\mbp'|^5} \bm \gamma_\mbp' \cdot \bm \gamma_\mbp'' \right|\lesssim 1+|\bar p'''|. 
 \eeqs
 The $L^2(\msI_\mbp)$-bound of  $\kappa_\mbp'$ now follows from \eqref{re-barp-hatp}. The $L^2(\msI_\mbp)$ bound of $\kappa_\mbp''$ is obtained from similar calculations, the details are omitted. 
 
 To see that $\Gamma_\mbp\in\cG_{K,\ell}^2$, we remark from \eqref{est-de-gamma-p} that $\bm\gamma_\mbp$ is uniformly bounded in $W^{2,\infty}(\mathscr I)$ by some $K>K_0.$ In particular $\|\kappa_\mbp\|_{L^\infty}$ inherits this uniform bounded.  To establish condition (b) of Definition\,\ref{def-GKl} we first establish that $\Gamma_0\in\cG^2_{K,\tilde\ell}$ for some $\tilde\ell>0$. We have condition (b) for $\Gamma_0$ with $K_0$ and $\ell_0$. If $1/(2K)<|s_1-s_2|_\msI<1/(2K_0)$ then by \eqref{e:kp1} we have $|\bm\gamma(s_1)-\bm\gamma(s_2)|>1/(2\sqrt{2}K)$. Combining these cases we have $\Gamma_0\in\cG_{K,\tilde\ell}$ with $\tilde\ell=\min\{\ell_0,1/(2\sqrt{2}K)\}$. For $|s_1-s_2|_\msI>1/(2K)$, by \eqref{def-gamma-p} with $\Theta_0$ independent of $s$ and \eqref{def-barp} we derive
 \beqs
 \begin{aligned}
 |\bm \gamma_\mbp(s_1)-\bm \gamma_\mbp(s_2)|&= \frac{1+\mrp_0}{ A(\mbp)} |\bm \gamma_{\bar p}(s_1)-\bm \gamma_{\bar p}(s_2)|\geq  \frac{1+\mrp_0}{ A(\mbp)} \left(\tilde \ell - 2\|\hat \mbp\|_{\mbV_1}\right).
 \end{aligned}
 \eeqs
 Here we used \eqref{re-barp-hatp} to bound the $L^\infty$-norm $\bar p$ and $\|\hat\mbp\|_{\mbV_0}\leq \|\hat\mbp\|_{\mbV_1}$. Lemma\,\ref{lem-A} affords the bound $A(\mbp)=O(1+\|\hat\mbp\|_{\mbV_1})$, and we deduce that $\Gamma_\mbp\in \cG_{K,\ell}^2$ for $\ell$ less than $\tilde \ell/4$ for all $\mbp\in\cD_\delta $ by choosing $\delta$  suitably small. 
 We deduce from Lemma\,\ref{def-GKl} that each $\Gamma_\mbp$ is $\ell$-far from self-intersection.
\end{proof}

The dressing of interfaces requires a $2\ell$ reach. From Lemma\,\ref{lem-Gamma-p} we may choose $\ell>0$ such that the collection of perturbed interfaces belong to $\cG_{K,2\ell}^2$ with associated reach $\Gamma_\mbp^{2\ell}$. To each $\mbp\in\cD_\delta $ this allows us to introduce the local whiskered coordinates $(s_{\mathbf p}, z_{\mathbf p})$ associated to $\Gamma_\mbp$. 
Similarly, the geometric structures $\mbn_\mbp, \bm\gamma_\mbp$ and $\kappa_\mbp$ associated to $\Gamma_\mbp$ have natural extensions to $\Gamma_\mbp^{2\ell}$.
The domain $\Gamma_\mbp^{2\ell}$ of $(s_\mbp,z_\mbp)$ overlaps with the domain $\Gamma_0^{2\ell_0}$ of the local coordinates $(s,z)$ associated to the base point $\Gamma_0.$  On the interface $\Gamma_\mbp$, corresponding to $z_\mbp=0$, the whiskered variable $s_\mbp$ reduces to $s$, that is $s_\mbp\big|_{z_\mbp=0}= s$.  The quantity $\tilde s$, and not $s$, corresponds to arc-length on $\Gamma_\mbp$. In the sequel the term ``local coordinates of $\Gamma_\mbp$'' refers to $(s_\mbp,z_\mbp)$ on $\Gamma_\mbp^{2\ell}$, however  it is convenient to introduce $\tilde s_\mbp$, the extension of $\tilde s$ to $\Gamma_\mbp^{2\ell}$, as this is the natural variable for the  Laplace-Beltrami eigenmodes $\{\tilde\Theta_j\}_{j\geq 0}$ of $\Delta_{s_\mbp}$, and of their integrals.

\begin{notation}\label{Notation-h}
To simplify the presentation of the subsequent calculations, we will use the blanket notation $h(z_\mbp, \bm \gamma_{\mathbf p}^{(k)})$ for any smooth function defined in $\Gamma_\mbp^{2\ell}$ that depends upon $s_{\mbp}$ only through the first $k$ derivatives of $\bm \gamma_{\mbp}$. If the function is independent of $z_\mbp$ we will denote it by $h(\bm \gamma_{\mathbf p}^{(k)}).$
\end{notation} 

The following Lemma presents a common use of Notation\,\ref{Notation-h}.
\begin{lemma}\label{lem-e_ij}
If  function $f=f(s_\mbp)$ defined on $\Gamma_\mbp^{2\ell}$ depends upon $s_\mbp$ only through $|\bm \gamma_\mbp'|,  \kappa_\mbp,  \mbn_\mbp\cdot \mbn_0$, $\varep^k\nabla^k_{ s_\mbp}$, and their derivatives, then under the assumptions \eqref{A-00} there exists $h=h(\bm \gamma_\mbp'')$ in the sense of Notation \ref{Notation-h} such that $f(s_\mbp)=h(\bm \gamma_\mbp'')$,
where $h$ satisfies
\begin{equation}\label{est-h-p-V-2+2}
\| h(\bm \gamma_{\mbp}'')\|_{L^2(\msI_{\mbp})} +\|h(\bm \gamma_\mbp'')\|_{L^\infty} \lesssim 1;
\end{equation}
and for $l\geq 1$,
\begin{equation}\label{est-h-gamma''-de}
\left\|\varep^{l-1}\nabla_{ s_{\mbp}}^{l} h( \bm \gamma_{\mbp}'')\right\|_{L^2(\msI_{\mbp})}\lesssim 1+ \|\hat{\mbp}\|_{\mbV_{3}^2}, \qquad \|\varep^{l-1}\nabla_{s_\mbp}^l h(\bm \gamma_\mbp'')\|_{L^\infty}\lesssim 1+\|\hat\mbp\|_{\mbV_3}.
\end{equation}
\end{lemma}
\begin{proof}
The estimates \eqref{est-h-p-V-2+2}-\eqref{est-h-gamma''-de}
are direct results of Lemma \ref{lem-Gamma-p}. 
\end{proof}
The following lemma is used frequently to establish bounds on vector and operator norms in various error terms.
\begin{lemma}\label{Notation-e_i,j}
Recalling the notation of section\,\ref{ssec-Notation}, if $f\in L^2(\msI_\mbp)$, then there exists a unit vector $\bm e=(e_i)\in l^2$ such that
\beq\label{def-e-i}
\int_{\msI_\mbp}f \tilde \Theta_i \dd \tilde s_\mbp=O(\|f\|_{L^2(\msI_\mbp)})e_i.
\eeq
If  in addition $f\in L^\infty$ on $\msI_\mbp$, then for any vector $\mathbf a=(\mathrm a_j)\in l^2$, we have
\beq\label{est-e-i,j}
\left|\sum_{j}\int_{\msI_\mbp} f\tilde \Theta_i \mathrm a_j\tilde \Theta_j\dd \tilde s_\mbp\right|\lesssim \|\mathbf  a\|_{l^2}\|f\|_{L^\infty}e_i,
\eeq
and there exists a matrix $\mbE=(\mbE_{ij})$ with operator norm $l^2_*$ norm equal to one, such that
\beq\label{def-e-i,j}
\int_{\msI_\mbp}f\tilde \Theta_i \tilde \Theta_j \dd \tilde s_\mbp =O(\|f\|_{L^\infty}) \mbE_{ij}.
\eeq
\end{lemma}
\begin{proof} The estimates follow from Plancherel and classic applications of Fourier theory.
\end{proof}

\section{Quasi-Equilibrium Profiles and the Bilayer  Manifold}\label{sec-profile}
Fix $K_0,\ell_0>0$ and a base interface $\Gamma_0\in\cG_{K_0,2\ell_0}^4$. 
We  associate the collection of perturbed interfaces $\{\Gamma_\mbp\}_{\mbp\in\cD_\delta}$ and construct the bilayer manifold as the graph of the quasi-equilibrium bilayer \muckmuck $\Phi_{\mathbf p}$ over the set $\cD_\delta.$ The bulk density parameter $\sigma$ is slaved to  the meander parameters to enforce a prescribed total mass constraint. The  construction of the quasi-equilibrium bilayer \muckmuck begins with $\phi_0$ defined on $L^2(\mathbb R)$ as the nontrivial solution of 
\begin{equation}\label{def-phi0n2}
\p_{z_{\kpp}}^2 \phi_{0}-W'(\phi_{0})=0,
\end{equation} that is homoclinic to the left well $b_-$ of $W$. In particular $\phi_0$ is unique up to translation, even about its maximum, and  $\phi_0-b_-$ converges to $0$ as $z_{\kpp}$ tends to $\pm \infty$ at the exponential rate $\sqrt{W''(b_-)}>0.$ 

The linearization  $\mathrm L_{\kpp 0}$ of  \eqref{def-phi0n2} about $\phi_0$, 
\begin{equation}\label{def-rL-p0}
\mathrm L_{\kpp 0}:=-\p_{z_{\kpp}}^2+W''(\phi_{0}(z_{\kpp})),
\end{equation}
is a Sturm-Liouville operator on the real line whose coefficients decay exponentially fast to constants at $z_{\kpp}=\infty$. 
The following Lemma follows from classic results and direct calculations, see for example Chapter 2.3.2 of \cite{KP-13}. 
\begin{lemma}\label{lem-L0}
The spectrum of \,$\mathrm L_{\kpp 0}$ is real, and uniformly positive except for two point spectra: $\la_0<0$ and $\la_1=0$ and associated ortho-normal eigenmodes $\psi_0$ and $\psi_1.$
Moreover, it holds that
\begin{equation*}
\mathrm L_{\kpp 0} \phi_{0}'=0, \quad \mathrm L_{\kpp 0}\phi_{0}''=-W'''(\phi_{0})\left| \phi_{0}' \right|^2, \quad \mathrm L_{\kpp 0}\left( z_{\kpp}\phi_{0}' \right)=-2\phi_{0}''.
\end{equation*}
The ground state eigenmode $\psi_0$ is even and positive. The  kernel of $L_{\kpp 0}$ is spanned by $\psi_1=\phi_0^\prime/\|\phi_0^\prime\|_{L^2}.$ The operator $\mrL_0$ is invertible  on the $L^2$ perp of its kernel, and both $\mrL_{\kpp 0}$ and its inverse preserve parity.
\end{lemma}

Some care must be taken to distinguish between functions in $L^2(\mathbb R)$ and their dressings that reside in $L^2(\Omega)$.  As an example, since $1$ is $L^2(\mathbb R)$ orthogonal to $\phi_0^\prime$ we may define $B_{\kpp k}=\mathrm L_{\kpp 0}^{-k} 1\in L^\infty(\mathbb R)$
and its dressing subject to $\Gamma_\mbp$,
\beq \label{def-B+dp,k}
B_{\mathbf p,k}^d (x):=(\mrL_0^{-k}1)^d \in L^\infty,
\eeq 
defined on all of $\Omega$.  Recalling the averaging operator, \eqref{def-massfunc} we introduce
\begin{equation}
\label{def-Bdp-k-mass}
\oB _{\mathbf p, k}^d:=|\Omega|\left\langle B_{\mathbf p,k}^d \right\rangle_{L^2}. 
\end{equation} 
Here and below, we drop the $d$ superscript on the dressed function to simplify notation when no ambiguity arises. Introducing  $\eta_d:=\eta_1-\eta_2$, we define the first dressed correction $\phi_1$ to the pulse profile
\begin{equation}\label{def-phi1}
\begin{aligned}
\phi_1(\sigma)=\phi_{1}(z_{\mathbf p};\sigma)&:=\sigma B_{\mbp,2}+\frac{\eta_d}{2} \mathrm L_{0}^{-1}\left( z_\mbp\phi_{0}' \right),
\end{aligned}
\end{equation}
which depends upon the bulk density and meander parameters, $\sigma\in\mbR$ and $\mbp$. The bulk density parameter controls the value of $\phi_1$ outside of $\Gamma_\mbp^{2\ell}$, where the profile is constant. In the construction of the bilayer manifold $\sigma=\sigma(\mbp)$, adjusting the bulk density state to  make bilayer mass $|\Omega|\langle\Phi_\mbp-b_-\rangle_{L^2}$ independent of $\mbp.$ Viewed as a function on $\mathbb R$,  $\phi_1$ is smooth and is even with respect to $z$, while as a function on $\Omega$ it is smooth and even in $z_\mbp$ to leading order.   

The second order correction $\phi_2$ is composed of products of whisker independent dressed functions and the whisker dependent curvature $\kappa_\mbp=\kappa_\mbp(s_\mbp)$.  
As such $\phi_2$ is not strictly the dressing of a function of one variable, indeed for each fixed value of $s_\mbp$, we define it as the $L^2(\mbR)$ solution of
\begin{equation}\label{def-phi-2-0}
\begin{aligned}
\mathrm L_{0}^2\phi_2(z, s_{\mathbf p})&=g_2(z, s_\mbp):=-\mathrm L_{0}\left(z\kappa_{\mathbf p}^2 \phi_0' +\frac{W'''(\phi_0)}{2}\phi_1^2\right) -\bigg(\kappa_{\mathbf p}^2\phi_0''+( -\eta_1+W'''(\phi_0)\phi_1)\mathrm L_{0} \phi_1\\
&\qquad+\eta_d W''(\phi_0)\phi_1\bigg) -\kappa_{\mathbf p}\Big(2\mathrm L_{0}\phi_1'(\sigma_1^*) +(-\eta_1+2W'''(\phi_0)\phi_1(\sigma_1^*))\phi_0'\Big), 
\end{aligned}
\end{equation}
The constant $\sigma_1^*$ determined below to insure the right-hand side of \eqref{def-phi-2-0} is in the range of $\mrL_0$ on each whisker. Since each $s_\mbp$ dependent term decays exponentially to zero in $z_\mbp$, the resulting whisker-dependent function extends to a smooth dressing $\phi_2$ around $\Gamma_\mbp$.  We denote this extension by 
\begin{equation}\label{def-phi2}
\begin{aligned}
\phi_2
      &:=-\mathrm L_{0}^{-1} \left(z_{\mathbf p}\kappa_{\mathbf p}^2 \phi_0' +\frac{W'''(\phi_0)}{2}\phi_1^2\right) -\mathrm L_{0}^{-2}\bigg(\kappa_{\mathbf p}^2\phi_0''+( -\eta_1+W'''(\phi_0)\phi_1)\mathrm L_{0} \phi_1\\
&\qquad+\eta_d W''(\phi_0)\phi_1\bigg)  -\kappa_{\mathbf p}\mathrm L_{0}^{-2}\Big(2\mathrm L_{0}\phi_1'(\sigma_1^*) +(-\eta_1+2W'''(\phi_0)\phi_1(\sigma_1^*))\phi_0'\Big).
\end{aligned}
\end{equation}
To verify that the inverses of $\mrL_0$ are well defined we observe that the first two applications of $\mrL_0^{-1}$ in the right-hand side of \eqref{def-phi2} are to functions that are even in $z$, 
and hence orthogonal in $L^2(\mathbb R)$ to $\phi_0^\prime.$ 
 The third application is to a function that is odd in $z$, which we denote by $(g_2)^{\rm odd}$.
In $L^2(\mathbb R)$ we calculate the projection of $g_2^{\rm odd}$ onto the kernel of $\mrL_0$,
\begin{equation}\label{proj-phi2-1}
\int_{\mathbb R} g_2^{\rm odd}\,\phi_0' \, \mrd z_{}=  -\eta_1m_1^2+ 2\int_{\mathbb R} W'''(\phi_0) |\phi_0'|^2\phi_1(\sigma_1^*)  \mrd  z_{},
\end{equation}
where $m_1$ is defined as
\beq
\label{def-m1}
m_1:= \|\phi_0'\|_{L^2(\mbR)}.
\eeq
In light of Lemma \ref{lem-L0}, we have $\mrL_{0}^2(z_{}\phi_0')=-2\mathrm L_{0}(\phi_0'') = 2W'''(\phi_0)|\phi_0'|^2$. Using the definition of $\phi_1$ and integration by parts, we have
\begin{equation}\label{proj-phi2-2}
\begin{aligned}
2\int_{\mathbb R} W'''(\phi_0) |\phi_0'|^2\phi_1(\sigma_1^*) \dd z_{}=\int_{\mathbb R}\mathrm L_{0}^2 (z_{}\phi_0')\phi_1(\sigma_1^*)\dd z_{}
=-m_0\sigma_1^*+\frac{\eta_d}{2}m_1^2.
\end{aligned}
\end{equation}
For $\sigma_1^*$ given by
\beq\label{def-sigma1*}
\sigma_1^*:= -\frac{(\eta_1+\eta_2) m_1^2}{2m_0}\qquad \hbox{with} \quad m_0:=\int_{\mbR} (\phi_0 (z)-b_-) \dd z,
\eeq 
we see that the terms on the right-hand side of \eqref{proj-phi2-1} cancel, and we deduce the bounded invertibility of $\mrL_0$
in \eqref{def-phi2}. We are in position to introduce the  profile.

\begin{lemma}\label{lem-def-Phi-p} Let 
meander parameters $\mbp$ satisfy \eqref{A-00}. Then for 
$\phi_0$, $\phi_1,$ and $\phi_2$ defined in \eqref{def-phi0n2}, \eqref{def-phi1}, and \eqref{def-phi2} respectively, 
we define the bilayer \muckmuck
\begin{equation}
\label{def-Phi-p}
\Phi_{\mathbf p}(x; \sigma):=\phi_{0}(z_{\mathbf p})+\varep \phi_{1}( z_{\mathbf p}; \sigma)+\varep^2\phi_{ 2}(s_{\mathbf p}, z_{\mathbf p};\sigma,\sigma^*),
\end{equation}
which has the following residual 
\begin{equation}\label{exp-mF}
\begin{aligned}
\mathrm F(\Phi_{\mathbf p})= 
\varep\sigma+\varep^2\mathrm F_2+\varep^3\mathrm F_3+\varep^4\mathrm F_{\geq 4}.
\end{aligned}
\end{equation}
Here the expansion terms in the main residual $\mrF_m$ have the form
\begin{equation}\label{F-234}
\begin{aligned}
\mathrm F_2&=\kappa_{\mathbf p}(\sigma-\sigma_1^*) f_2(z_{\mathbf p});\qquad \qquad 
\mathrm F_3=-\phi_0'\Delta_{s_{\mathbf p}}\kappa_{\mathbf p} +f_3(z_{\mathbf p}, \bm \gamma_{\mbp}''), \\
&\mathrm F_{\geq 4}=f_{4,1}(z_{\mathbf p}, \bm \gamma_{\mbp}'') \Delta_g f_{4,2}(z_{\mathbf p}, \bm \gamma_{\mbp}'')+f_{4,2}(z_{\mathbf p}, \bm \gamma_{\mbp}''),
\end{aligned}
\end{equation}
where $f=f(z, \bm \gamma_\mbp'')$ with various subscripts  
are smooth functions which decay exponentially fast to a constant as $|z|\to \infty$. In particular, $f_2(z)$ is odd with respect to $z$ and decays to zero as $|z|\to \infty$.  In addition, $\mathrm F_2, \mathrm F_3$ satisfy the following projection properties:
\begin{equation}\label{est-proj-rF2,3}
\begin{aligned}
&\int_{\mathbb R_{2\ell}} \mathrm F_2\,\phi_0'\dd z_{\mathbf p}=m_0(\sigma_1^*-\sigma)\kappa_{\mathbf p}+O(e^{-\ell\nu/\varep}); 
\\
&\int_{\mathbb R_{2\ell}} \mathrm F_3 \,\phi_0'\, \mrd z_{\mathbf p}=m_1^2\left(-\Delta_{s_{\mathbf p}}\kappa_{\mathbf p} -\frac{\kappa_{\mathbf p}^3}{2} +\alpha \kappa_{\mathbf p}\right)+O(e^{-\ell\nu/\varep}).
\end{aligned}
\end{equation}
Here  $\alpha =\alpha (\sigma; \eta_1, \eta_2)$ is a smooth function of $\sigma.$ 
\end{lemma}
\begin{proof}
For brevity of notation, we drop the subscript $\mbp$ in the proof. 
the variational derivative  $\mathrm F(\Phi)$ can be written as
\begin{eqnarray}\label{eq-F-1}
\begin{aligned}
\mathrm F(\Phi)
=&\left[\p_z^2+\varep \mathrm H\p_z+\varep^2 \Delta_g-W''(\Phi)+\varep \eta_1\right]\left[\p_z^2 \Phi+\varep \mathrm H \p_z\Phi +\varep^2 \Delta_g\Phi-W'(\Phi)\right]\\
&+\varep\eta_dW'(\Phi).
\end{aligned}
\end{eqnarray}
The components of the profile $\Phi$ were chosen to make the residual $\Pi_0\mathrm F(\Phi)$ small to  $O(\varep^2)$.  We expand $\mathrm F(\Phi)$ in powers of $\varep$, 
and introduce $\phi_{\geq 1}:=\phi_1+\varep\phi_2$. 
Taylor expanding the $k$-th derivative of $W(\Phi)$ around 
$\phi_0$ for $k=1,2$ and keeping terms up to third order we find,
\begin{equation}
\begin{aligned}
W^{(k)}(\Phi) 
=&W^{(k)}(\phi_0)+\varep W^{(k+1)}(\phi_0)\phi_1+\varep^2 \left(W^{(k+1)}(\phi_0) \phi_2+\frac{W^{k+2}(h_k)}{2}\phi_1^2\right)+\\
&+\frac{\varep^3}{2}W^{(k+2)}(\phi_0)(2\phi_1+\varep\phi_{ 2})\phi_{ 2} +\varep^3 \frac{W^{(k+3)}(\phi_0)}{3!}\phi_{\geq 1}^3.
\end{aligned}
\end{equation}
Similarly the expansion of the extended curvature $\mathrm H$ to third order takes the form
\begin{equation*}
\mathrm H=-\frac{\kappa}{1-\varep z\kappa}=-\kappa -\varep z\frac{\kappa^2}{1-\varep z\kappa}= -\kappa- \varep z\kappa^2-\varep^2 z^2 \kappa^3-\varep^3 z^3\frac{\kappa^4}{1-\varep z\kappa}.
\end{equation*}
The whiskered coordinate expression \eqref{eq-F-1} of $\mathrm F(\Phi)$  admits the expansion
\begin{equation}\label{exp-bF-1}
\mathrm F(\Phi)=\varep \left(\mathrm L_0^2\phi_1+\eta_d W'(\phi_0)\right) +\varep^2\mathrm F_2+\varep^3\mathrm F_{3}+\varep^4\mathrm F_{\geq 4}.
\end{equation}
 Using the identities from Lemma \ref{lem-L0}, $\mathrm F_2$ and $\mathrm F_3$ reduce to
\begin{equation*}
\begin{aligned}
\mathrm F_2& =\mathrm L_0\left(\kappa\phi_1'+\mathrm L_0\phi_2 +z\kappa^2\phi_0'+\frac{W'''(\phi_0)}{2}\phi_1^2\right)+\big(\kappa \p_z -\eta_1+W'''(\phi_0)\phi_1\big)\\
&\qquad \times(\kappa\phi_0'+\mathrm L_0\phi_1) +\eta_d W''(\phi_0)\phi_1; \\
\mathrm F_3&=
\mathrm L_{0}\left(\kappa\p_z \phi_2+W'''(\phi_0) \phi_1\phi_2+z^2\kappa^3\phi_0'+z\kappa^2 \phi_1' +\frac{W^{(4)}(\phi_0)}{3!}\phi_1^3\right)\\
&\quad +(\kappa \p_{z}-\eta_1+W'''(\phi_0)\phi_1)  \bigg(\mathrm L_0\phi_2+\kappa\phi_1'+\kappa^2 z\phi_0'  +\frac{W'''(\phi_0)}{2}\phi_1^2\bigg) \\
&\quad-\Delta_{s}\kappa\phi_0' +\left(\frac{W^{(4)}(\phi_0)}{2}\phi_1^2+W'''(\phi_0)\phi_2\right)(\kappa\phi_0' +\mathrm L_0\phi_1)\\
&\quad +\kappa^3 z\phi_0'' +z\kappa^2 \p_ z\mathrm L_0\phi_1+\eta_d\left(W''(\phi_0)\phi_2+\frac{W'''(\phi_0)}{2}\phi_1^2\right).
\end{aligned}
\end{equation*}

Within $\Gamma_\mbp^\ell$ using the expressions for $\phi_1, \phi_2$ in \eqref{def-phi1} and \eqref{def-phi2} we see that the $O(\varep)$ term in \eqref{exp-bF-1} reduces to the constant $\sigma$.
Using the definition of $\phi_2$ given in \eqref{def-phi2}, the term  $\mathrm F_2$ further reduces to
\begin{equation*}
\mathrm F_2=\kappa\mathrm L_0(\phi_1-\phi_1(\sigma_1^*))'+\kappa\p_z\mathrm L_0(\phi_1-\phi_1(\sigma_1^*))+W'''(\phi_0)(\phi_1-\phi_1(\sigma_1^*))\kappa\phi_0',
\end{equation*}
and the final expression for $\mrF_2$ in \eqref{F-234} follows from \eqref{def-phi1} with the reductions for $\mathrm F_3$ and $\mathrm F_4$ obtained from similar calculations.
%
In particular ${\mathrm F_{\geq 4}}$ takes the exact form:
\begin{equation*}
\begin{aligned}
\mrF_{\geq 4}=&-(\p_z^2+\varep \mathrm H\p_z+\varep^2\Delta_g-W''(\Phi)+\varep\eta_1)\bigg(\frac{W'''(\phi_0)}{2}\phi_{ 2}^2+\Delta_g\phi_{2}+\frac{W^{(4)}(h)}{3!}(3\phi_1^2\phi_{ 2}  +3\varep\phi_1\phi_{ 2}^2+\varep^2\phi_{2}^3)\\[3pt]
&+\frac{(z^2\kappa^3\phi_0'+z\kappa^2\phi_1'+\kappa\p_z \phi_2)z\kappa}{1-\varep z\kappa}\bigg)+\left(\frac{z^2\kappa^3}{1-\varep z\kappa} \p_z+\frac{W^{(4)}(h) }{2}(2\phi_1+\varep\phi_{2}) \phi_{ 2}\right)(\kappa\phi_0'+\mathrm L_0\phi_1)
\\
&-(\mathrm H\p_z+\varep\Delta_g-W'''(h)\phi_{\geq 1} +\eta_1) \left(W'''(\phi_0)\phi_1\phi_2 +z^2\kappa^3\phi_0'+z\kappa^2\phi_1'+\kappa\p_z\phi_2+\frac{W^{(4)}(\phi_0)}{3!}\phi_1^3 \right)\\
&-\left(\Delta_g-z\kappa^2\p_z-\frac{W^{(4)}(h)}{2}\phi_{\geq 1}^2-W'''(\phi_0)\phi_{ 2}\right)\Big(\kappa\phi_1'+z\kappa^2\phi_0'+\mathrm L_0\phi_2 +\frac{W'''(\phi_0)}{2}\phi_1^2\Big),
\end{aligned}
\end{equation*}
where the $h$ terms denote remainders from Taylor expansion. The highest derivative with respect to $s$ arises from $\Delta_g\phi_2$ where $\phi_2=\phi_2(s,z)$ through its definition in \eqref{def-phi2}. 

The projection of $\mathrm F_2$ onto $\phi'_0$ is similar to the calculation of \eqref{proj-phi2-1} and \eqref{proj-phi2-2} and omitted.
 $L^2(\mathbb R_{\ell})$ to $z\phi_0'$ is zero.
To estimate the projection of  $\mathrm F_3$ in $\mbL^2(\mathbb R_{\ell})$ to the function $\phi'_0=\phi_0'(z)$, it suffices to consider the odd part of $\mathrm F_{3}$. 
Indeed, since $\phi_0, \phi_1$ are all even functions with respect to $z$, we have
\begin{equation*}
\begin{aligned}
\mathrm F_{3}^{\rodd}=& \mathrm L_0\left( \kappa\p_z\phi_2^{\reven}+z^2\kappa^3\phi_0'+W'''(\phi_0)\phi_1\phi_2^{\rodd}\right)-\Delta_s \kappa \phi_0' +\left( -\eta_1+W'''(\phi_0)\phi_1\right)\\
&\times(\mathrm L_0\phi_2^{\rodd}+ \kappa \phi_1')+\kappa\p_z\left(\mathrm L_0\phi_2^{\reven}++z\kappa^2\phi_0'+\frac{W'''(\phi_0)}{2}\phi_1^2\right)\\
&+\kappa\frac{W^{(4)}(\phi_0) }{2}\phi_1^2\phi_0'+\kappa W'''(\phi_0)\phi_2^{\reven} \phi_0'+\kappa^3 z\phi_0''+\eta_d W''(\phi_0)\phi_2^{\rodd}.
\end{aligned}
\end{equation*}
 Integrating by parts, using properties of $\mathrm L_0$ from Lemma \ref{lem-L0}  and re-organizing, we obtain 
\begin{align*}
\int_{\mathbb R_{\ell}} \mathrm F_{3} \phi_0'\dd z=&-\Delta_s\kappa m_1^2-\frac{\eta_1\kappa}{2}\int_{\mathbb R_{\ell}} \mathrm L_0\phi_1 \phi_0'z\dd z +\mathcal I_1+\mathcal I_2+\mathcal I_3+O(e^{-\frac{\ell\nu}{\varep}})
\end{align*}
where
\begin{align*}
\mathcal I_1:=\kappa\int_{\mathbb R_{\ell}}\mathrm L_0^2 \phi_2 \phi_0'z\dd z; \qquad \mcI_2:=\int_{\mathbb R_{\ell}}W'''(\phi_0)\phi_0' \phi_1 \mathrm L_0\phi_2^{odd}\dd z; \\
\mcI_3:=\eta_d \int_{\mathbb R_{\ell}} W''(\phi_0)\phi_0'\phi_2^{odd}\dd z.\h{60pt}
\end{align*}
For $\phi_1=\phi_1(\sigma)$ and some smooth function $\alpha=\alpha(\sigma)$,  the projection of $\mathrm F_3$ in \eqref{est-proj-rF2,3}  follows from the  identities:
\begin{equation*}
\begin{aligned}
\mathcal I_1=&-\frac{\kappa^3}{2}m_1^2+\eta_1\kappa\int_{\mathbb R_{\ell}} \mathrm L_0\phi_1\phi_0' z\dd z-\kappa\int_{\mathbb R_{\ell}}  W'''(\phi_0)\phi_1\mathrm L_0\phi_1\phi_0' z\dd z\\
&-\eta_d\kappa \int_{\mathbb R_{\ell}} W''(\phi_0)\phi_1\phi_0' z \dd z+\kappa\int_{\mathbb R_\ell} W'''(\phi_0)\phi_1^2\phi_0''\dd z;
\end{aligned}
\end{equation*}
\begin{equation*}
\begin{aligned}
\mcI_2=&-\kappa\int_{\mathbb R_\ell} W'''(\phi_0)\phi_0'\phi_1 \mathrm L_0^{-1}((-\eta_1+2W'''(\phi_0)\phi_{1}(\sigma_1^*))\phi_0')\\
&
-2\kappa\int_{\mathbb R_\ell} W'''(\phi_0)\phi_0' \phi_1\phi_1'(\sigma_1^*)\dd z;
\end{aligned}
\end{equation*}
\begin{equation*}
\mcI_3=-\eta_d\kappa\int_{\mathbb R_\ell}W''(\phi_0)\phi_0' \mathrm L_0^{-2}\left(2\mathrm L_0\phi_1'(\sigma_1^*)+(-\eta_1+2W'''(\phi_0)\phi_1(\sigma_1^*))\phi_0'\right)\dd z.
\end{equation*}

\end{proof}

Outside of $\Gamma_\mbp^{2\ell}$, the profile $\Phi_\mbp$ reduces to a constant value that admits the expansion
\beq
\Phi=b_-+\varep \phi_1^\infty +\varep^2\phi_2^\infty, 
\eeq  
where the leading order correction relates to the bulk density parameter
$$\phi_1^\infty = B_2^\infty \sigma.$$
The flow \eqref{eq-FCH-L2} conserves the system mass,
making it a key parameter that is fixed by the initial data.  As we study bilayers of length $O(1)$ it is natural to scale the mass
\beq\label{def-M0}
\int_\Omega (u-b_-) \dd x= \varep M_0.
\eeq
We adjust the bulk density parameter so that $\Phi_\mbp$ has mass $\varep M_0$, and a solution $u$ of \eqref{eq-FCH-L2} satisfies
\beq 
\label{e:ConsMass}
0=\langle u(t)-\Phi_{\mathbf p} \rangle_{L^2}=
\frac{\varep M_0}{|\Omega|}-\langle \Phi_{\mathbf p}-b_- \rangle_{L^2}. 
\eeq
The exact relation of $\sigma$ required to guarantee \eqref{e:ConsMass} is determined from the expansion \eqref{def-Phi-p} of $\Phi_{\mathbf p}$ with $\phi_1=\phi_1(\sigma)$ given by \eqref{def-phi1},
\begin{equation}\label{def-hatla}
\begin{aligned}
\sigma(\mathbf p)= \frac{1}{\oB_{\mathbf p,2} } \Bigg\{ M_0- \int_\Omega \bigg[&\frac{1}{\varep}\Big(  \phi_0(z_{\mathbf p}) -b_-  +\varep^2\phi_{2} (s_{\mathbf p}, z_{\mathbf p})\Big) +\frac{\eta_d}{2} \mathrm L_{\mathbf p, 0}^{-1}(z_{\mathbf p}\phi_0') \bigg]\, \mrd   x\Bigg\}.
\end{aligned}
\end{equation}

The bilayer  manifold of perturbations from $\Phi_0$ is constructed as the graph of $\Phi_\mbp$ over the domain $\cD_\delta$ subject to the mass condition $\langle \Phi_\mbp-b_-\rangle_{L^2}=\varep M_0/|\Omega|.$
\begin{defn}[Bilayer  Manifold]
\label{def-bM0}
Fix $K_0,\ell_0,\delta>0$.  Given a base point interface $\Gamma_0\in\cG_{K_0,2\ell_0}^4$ and system mass $M_0$, we define the bilayer  manifold $\cMb(\Gamma_0,M_0)$  to be the graph of the map $\mbp\mapsto \Phi_\mbp(\sigma)$ over the domain $\cD_\delta$ with $\sigma=\sigma(\mbp)$ given by \eqref{def-hatla}.
\end{defn}

From Lemma\,\ref{lem-Gamma-p} for each fixed $K_0,\ell_0$ there exists $K,\ell>0$ such that for all $\mbp\in\cD_\delta$ the interfaces $\Gamma_\mbp\in\cG_{K,2\ell}^2$ are $2\ell$ far from self-intersection and each bilayer \muckmuck $\Phi_\mbp$ has the mass $(b_-|\Omega| +\varep M_0)$.

\begin{lemma}\label{lem-sigma}  For a given bilayer  manifold, the map $\sigma=\sigma(\mbp)$ over $\cD_\delta$ can be approximated by 
\beqs
\sigma(\mbp) = 
\frac{M_0-m_0|\Gamma_0|}{B_2^\infty|\Omega|} - \frac{m_0|\Gamma_0| }{B_2^\infty|\Omega|}\mrp_0 +O(\varep),
 \eeqs
 where $B_2^\infty$ is the nonzero far field of $B_2$ introduced in \eqref{def-B+dp,k} and $\mrp_0$ is the first component of $\mbp$ that scales the length of $\Gamma_\mbp$.
\end{lemma}
 \begin{proof}
 At leading order, the mass per unit length of interface associated to $\Phi_\mbp$ is independent of $\mbp$ and given by $m_0$, defined in \eqref{def-sigma1*}. 
The mass of $\Phi_\mbp$ satisfies
\beq\label{mass-Phi-p}
M_0= \frac{|\Omega|\langle\Phi_\mbp-b_-\rangle_{L^2}}{\varep} =  m_0 |\Gamma_\mbp| + B_2^\infty |\Omega |\sigma+O(\varep),
\eeq
Combining this with \eqref{Gamma-p-length}   yields the result.
 \end{proof}

\begin{remark}
In the companion paper \cite{CP-nonlinear} we present a refinement of $\Phi_\mbp$ which reduces to an equilibrium of the system for $\mbp=(\mrp_0,\mrp_1,\mrp_2, 0, \ldots, 0)$, e.g., when $\hat{\mbp}=\bm 0$, and
 $\Gamma_0$ is a circle.
\end{remark}

 At leading order the residual of $\Phi_\mbp$ is controlled by the deviation of the bulk parameter $\sigma$ from $\sigma_1^*.$ 
 \begin{lemma}\label{lem-est-residual} Under assumption \eqref{A-00}, the residual satisfies
 \beqs
 \begin{aligned}
 \|\Pi_0\mrF(\Phi_\mbp)\|_{L^2}\lesssim \varep^{5/2}|\sigma-\sigma_1^*| 
 +\varep^{7/2}(1+\|\hat\mbp\|_{\mbV_4^2}).
 \end{aligned}
 \eeqs
 \end{lemma}
 \begin{proof}
 The second estimate  results directly from the form of $\mrF_2$ and $\mathrm F_3$ 
 in \eqref{F-234} and the use of the estimates 
 \beqs
 \|\kappa_\mbp\|_{L^\infty}\lesssim 1+ \|\hat\mbp\|_{\mbV_2}\lesssim 1,\qquad  \|\Delta_{s_\mbp} \kappa_\mbp\|_{L^2(\mathscr I_\mbp)}\lesssim 1+ \|\hat \mbp\|_{\mbV_4^2}.
 \eeqs
 
 \end{proof}

\section{Fast and Slow Spaces and Coercivity}\label{sec-linear}

The nonlinear stability of the bilayer manifold hinges upon the properties  of the linearization of the flow \eqref{eq-FCH-L2} about each fixed quasi-steady bilayer \muckmuck $\Phi_\mbp$ constructed in Lemma \ref{lem-def-Phi-p}.  
In this section we establish the coercivity properties of the linearized operators that allows the nonlinear control established in Section\,\ref{s:nonstab-BLM}.

We fix $K_0,\ell_0$ and a base point interface 
$\Gamma_0\in\cG_{K_0,2\ell_0}^4$ and choose $K,\ell>0$ such that $\Gamma_\mbp\in\cG^2_{K,2\ell}$ for all $\mbp\in\cD_\delta.$
The linearization of \eqref{eq-FCH-L2} about $\Phi_\mbp$ takes the form $\Pi_0\mbL$ where
\begin{equation}\label{def-bLp}
\begin{aligned}
\mbL:=\frac{\delta^2\mathcal F}{\delta u^2}\Big|_{u=\Phi_{\mathbf p} }=&(\varep^2\Delta-W''(\Phi_{\mathbf p} )+\varep \eta_1)(\varep^2\Delta -W''(\Phi_{\mathbf p} )) \\
&- (\varep^2 \Delta \Phi_{\mathbf p} - W'(\Phi_{\mathbf p} ))W'''(\Phi_{\mathbf p} )+\varep\eta_dW''(\Phi_{\mathbf p} ),
\end{aligned}
\end{equation}
denotes the second variational derivative of $\mathcal F$ at $\Phi_{\mathbf p} $ and recall that $\eta_d=\eta_1-\eta_2$. 
When restricted to functions with support within the reach $\Gamma_\mbp^{2\ell}$, the Cartesian Laplacian admits the local coordinate expression
 \eqref{eq-Lap-induced} in terms of  $(s_{\mathbf p},z_{\mathbf p})$ which induces the expansion
\begin{equation}\label{exp-bL+d}
\mbL =\mbL_0 +\varep \mbL_{ 1} +\varep^2 \mbL_{\geq 2}.\quad 
\end{equation}
The leading order operator takes the form
\begin{equation}
\begin{aligned}
\label{def-cLp0}
\mbL_{ 0} := \left(\mrL_{0}-\varep^2\Delta_{s_{\mbp}}\right)^2=\mathcal L^2. 
\end{aligned}
\end{equation}
where we have introduced $\mathcal L := \mathrm L_{0} -\varep^2 \Delta_{s_{\mbp}}$. Much of the structure of the FCH flow stems
from $\mbL_{0}$, and its balancing of the $\Gamma_\mbp$ dressed operator $\mrL_{0}$, defined in \eqref{def-rL-p0}, against the Laplace-Beltrami operator associated to $\Gamma_\mbp$.  The next correction to $\mbL$ takes the form
\begin{equation}\label{def-bLp,1}
\begin{aligned}
\mbL_{1} = &(\kappa_{\mathbf p} \p_{z_{\mathbf p}} +W'''(\phi_0)\phi_1-z_{\mathbf p}\varep^2D_{{s_{\mathbf p}},2}-\eta_1)\mathcal L +\mathcal L (\kappa_{\mathbf p}\p_{z_{\mathbf p}} +W'''(\phi_0)\phi_1 \\
&-z_{\mathbf p}\varep^2 D_{{s_{\mathbf p}},2})+W'''(\phi_{ 0}) \left(\kappa_{\mathbf p} \left(\phi_{ 0}'\right) +\mathrm L_{0} \phi_1\right)+\eta_d W''(\phi_0).
\end{aligned}
\end{equation}
The second and higher order correction term, $ \mbL_{\geq 2}$, is relatively compact with respect to $\mbL_{0}$ and its precise form is not material.
We use the expansion \eqref{exp-bL+d} to construct approximate slow spaces, the meander and pearling spaces, that characterize the small spectrum of $\mbL$ in the sense that the operator is uniformly coercive on their compliment.  We tune the spectral cut-off parameter $\rho$ that controls the size of the pearling and meander spaces and to preserve the asymptotically large gap in their in-plane wave numbers while obtaining optimal coercivity.   
\subsection{Approximate slow spaces} Up to exponentially small terms, the approximate slow space $\mathcal Z$ is a product of functions of $z_\mbp$ and $s_\mbp$ that exploit the balance of the
operator $\mbL_{0}$ viewed as acting on the tensor product space $L^2(\mathbb R) \times L^2(\msI_\mbp)$.  As both $\mbL_{0}$ and $\mathcal L$
are self-adjoint, it is sufficient to establish coercivity of $\mathcal L.$ The spectrum of $\mathrm L_{0}$, in particular its first two eigenmodes $\{\psi_k\}_{k=0,1}$ of $\mrL_0$ are introduced in Lemma\,\ref{lem-L0}. The Laplace-Beltrami eigenvalues $\{\beta_{\mbp,j}^2\}_{j\geq0}$ of $\Delta_{s_\mbp}=\p_{\tilde s_\mbp}^2$ are discussed in \eqref{tTheta''}. 
\begin{defn}\label{def-slow-space}

For $k=0, 1$, we introduce the disjoint index sets:
\begin{equation}\label{def-Sigma}
\Sigma_k=\Sigma_k(\mbp, \rho)= \left\{j \bigl| \,\La_{kj}^2:=(\la_k+\varep^2\beta_{\mathbf p,j}^2)^2 \leq \rho\right\}, 
\end{equation}
their union, $\Sigma:=\Sigma_0\cup \Sigma_1,$ and the index function $I:\Sigma\mapsto\{0,1\}$ which takes the value $I(j)=k$ if $j\in\Sigma_k.$

The preliminary pearling and meander spaces, denoted  by $\mcZ^0(\mbp, \rho)$ and $\mcZ^1(\mbp, \rho)$ respectively, are defined in terms of their basis functions
\beq
Z_{\mathbf p}^{ I(j)j}:=\tilde\psi_{I(j)} (z_{\mathbf p}(  x)) \tilde \Theta_j(\tilde s_{\mathbf p}(  x)),\quad j\in\Sigma,
\eeq
where the dressed and scaled versions of the eigenmodes of $\mathrm L_{0}$ are defined by
\begin{equation*}
\tilde \psi_{ k} (z_{\mathbf p}):=\varep^{-1/2}\psi _{ k}(z_{\mathbf p})\quad k=0,1.
\end{equation*}
In particular,
$$\mcZ^k(\mbp,\rho)=\mathrm{span}\left\{ \mcZ^{kj}_\mbp\,\bigl|\, j\in\Sigma_k\right\},$$
and the preliminary slow space $\mcZ(\mbp, \rho):= \mcZ^0(\mbp, \rho)\cup\mcZ^1(\mbp, \rho) \subset L^2(\Omega)$, is their union. For simplicity of notation, we use $\mcZ^k, \mcZ$ to denote $\mcZ^k(\mbp, \rho)$ and $\mcZ(\mbp, \rho)$, respectively when there is no ambiguity. 
\end{defn}

The exponential decay of $\tilde\psi_k$ to zero away from the interface implies that the corrections arising from dressing are exponentially small, in particular there exist $\nu>0$ such that
\beq\label{bL0-Z}
\mathbb L_{0} Z_{\mbp}^{I(i)i}=\La_{I(i)i}^2 Z_{\mbp}^{I(i)i} +O(e^{-\ell\nu/\varep}),
\eeq
for all $i\in\Sigma$.
Since the set $\{\La_{I(i)i}^2\}_{i\in\Sigma}$ lies in the interval $(0,\rho)$ the 
functions in the slow space are compressed by a factor of $\rho$ under the action of $\mbL_{ 0}.$

To estimate the sizes $N_0$ and $N_1$ of $\Sigma_0$ and $\Sigma_1$, we remark from \eqref{tTheta''} and \eqref{def-Theta-beta} that
$\beta_{\mathbf p,j}^2\sim Cj^2.$
The ground-state eigenvalue $\la_0<0$, hence $k$ lies in $\Sigma_0(\rho)$ if and only if
\begin{equation}\label{est-N0}
\varep^{-1}\sqrt{-\la_0-\rho^{1/2}}\lesssim j\lesssim \varep^{-1} \sqrt{-\la_0+\rho^{1/2}},\quad \implies 
\quad N_0:=|\Sigma_0(\rho)|\sim \varep^{-1}\rho^{1/2}.
\end{equation}
On the other hand $\lambda_1=0$, so $j$ lies in $ \Sigma_1(\rho)$ if and only if
\begin{equation}\label{est-N1}
0\leq j\lesssim \varep^{-1} \rho^{1/4},\quad \implies 
\quad  N_1:=|\Sigma_1(\rho)|\sim \varep^{-1}\rho^{1/4}.
\end{equation}
The lower bound of elements in $\Sigma_0$, $\varep^{-1}\sqrt{-\la_0-\rho^{1/2}}$, is of order $\varep^{-1}$ while  the upper bound of $\Sigma_1$ is of order $\varep^{-1}\rho^{1/4}$. We deduce
that  $\Sigma_0$ and $\Sigma_1$ are  disjoint for $\rho$  suitably small. Indeed there exists $\rho_0, c>0$ such that for $\rho<\rho_0$ the associated in-plane wave numbers $\{\beta_i\}$ satisfy
\beq\label{eq-wavenumgap} 
 |\beta_i-\beta_j|\geq c\varep^{-1}, \quad \forall
i\in\Sigma_0, j\in \Sigma_1.
\eeq
The slow space $\mcZ$ has dimension 
\begin{equation*}
N:=|\Sigma(\rho)|=|\Sigma_0(\rho)|+|\Sigma_1(\rho)|\sim \varep^{-1}\rho^{1/4}.
\end{equation*}

\begin{remark}
The wave-number gap \eqref{eq-wavenumgap} plays an important role in bounding interactions between meander and pearling modes. In particular it yields the factor of $\varep$ in the upper bound of \eqref{est-Theta-2}. This is used in Proposition\,\ref{prop-M*} and is required to close the nonlinear estimates in the follow-on paper \cite{CP-nonlinear}.
\end{remark}

Using the formalism of Notation\,\ref{Notation-h} we have the following estimates.
 \begin{lemma}\label{lem-Theta}
Assume $\mbp\in \cD_\delta$ with $\cD_\delta$ introduced  in \eqref{A-00}, $\rho$ suitably small and $h=h(\bm \gamma_\mbp^{(k)})$ is a function satisfying  Notation \ref{Notation-h}. Then there exists a matrix $\mbE=(\mbE_{ij})$ which is
bounded in $l^2_*$ as a map from $l^2(\mbR^{N})$ to $ l^2(\mbR^N)$ such that
\begin{equation}\label{est-Theta}
\left|\int_{\msI_\mbp} h(\bm \gamma^{(k)}_{\mathbf p})\tilde \Theta_i \tilde \Theta_j\dd \tilde s_\mbp \right| \lesssim  \mbE_{ij}
\end{equation}
hold for  $i, j\in \Sigma=\Sigma_0\cup \Sigma_1$,  and  $k=1,2.$ Moreover for all $i,j$ such that $I(i)\neq I(j)$ we have
\begin{equation}\label{est-Theta-2}
\left|\int_{\msI_\mbp} h(\bm \gamma^{(k)}_{\mathbf p})\tilde \Theta_i \tilde \Theta_j\dd \tilde s_\mbp \right| \lesssim \varep(1+ \|\hat{\mathbf p}\|_{\mbV_{4}^2} )\mbE_{ij}.
\end{equation}
\end{lemma}
\begin{proof}  With $h$ satisfying Notation \ref{Notation-h},  we bound the $L^\infty$-norm from Lemma \ref{lem-e_ij} as
\beq\label{L-INF-diff}
|h(\bm \gamma_{\mbp}^{(k)})|\lesssim 1, \qquad k=1,2
\eeq
and deduce from Lemma \ref{Notation-e_i,j} that these terms are $O(1)\mbE_{ij}$ for a matrix $\mbE$ as above; the estimate \eqref{est-Theta} follows.  
For \eqref{est-Theta-2}, when $I(i)\neq I(j)$ we have $\beta_i\neq \beta_j$.    Integrating by parts twice we transfer the highest derivative of $\tilde \Theta_i$ to $\tilde \Theta_j$, which generates lower derivative terms from the product rule with $h$. Noting $\p_{\tilde s_\mbp}=\nabla_{s_\mbp}$, we write the result in the form
\begin{equation}\label{est-h-Theta-i,j-0}
\begin{aligned}
\int_{\msI_\mbp} h\left(\bm \gamma_\mbp'' \right)\tilde \Theta_i''\tilde \Theta_j \dd  \tilde s_{\mbp}&=-\int_{\msI_\mbp} h\left(\bm \gamma_\mbp''\right) \tilde \Theta_i' \tilde \Theta_j' \dd \tilde s_{\mbp}-\int_{\msI_\mbp} \nabla_{s_\mbp}h \tilde \Theta_i'\tilde \Theta_j   \dd \tilde s_{\mbp}\\
&=\int_{\msI_\mbp}h\left(\bm \gamma_\mbp''\right)\tilde \Theta_i\tilde \Theta_j''   |\bm \gamma_\mbp'| \dd s_{\mbp}+\int_{\msI_\mbp} \nabla_{s_\mbp}h \left(\tilde \Theta_i\tilde \Theta_j'-\tilde \Theta_i'\tilde \Theta_j  \right) \dd \tilde s_{\mbp}.
\end{aligned}
\end{equation}
Applying identity \eqref{tTheta''}  and after some algebraic rearrangement we obtain 
\begin{equation}\label{est-h-Theta-i,j-1}
\begin{aligned}
(\beta_{\mbp, j}^2-\beta_{\mbp, i}^2)\int_{\msI_\mbp}h\left(\bm \gamma_\mbp''\right)\tilde \Theta_i\tilde \Theta_j   \dd \tilde s_{\mbp}=\int_{\msI_\mbp}\nabla_{s_\mbp}h\left(\tilde \Theta_i\tilde \Theta_j'- \tilde \Theta_i'\tilde \Theta_j\right) \dd \tilde s_{\mbp}.
\end{aligned}
\end{equation}
By the relation \eqref{Theta'}, the right hand side can be rewritten as
\beqs
\int_{\msI_\mbp}\nabla_{s_\mbp}h\left(\tilde \Theta_i\tilde \Theta_j'- \tilde \Theta_i'\tilde \Theta_j\right) \dd \tilde s_{\mbp}= \beta_{\mbp, j}\int_{\msI_\mbp} \nabla_{s_\mbp} h \tilde \Theta_i  \tilde  \Theta_j'/\beta_{\mbp, j} \dd\tilde s_\mbp - \beta_{\mbp, i}\int_{\msI_\mbp} \nabla_{s_\mbp} h \tilde \Theta_i'/\beta_{\mbp, i}  \tilde \Theta_j \dd\tilde s_\mbp. 
\eeqs 
Note that from \eqref{Theta'} that $\{\pm \tilde \Theta_j'/\beta_{\mbp,j}, j\in \Sigma\}$ is equivalent to the set $\{\pm\tilde \Theta_j, j\in \Sigma\}$. Hence Lemma \ref{Notation-e_i,j} applies and there exists a matrix $\mbE=(\mbE_{ij})$ with $l_2^*$ norm one such that 
\beqs
\left|\int_{\msI_\mbp}\nabla_{s_\mbp}h\left(\tilde \Theta_i\tilde \Theta_j'- \tilde \Theta_i'\tilde \Theta_j\right) \dd \tilde s_{\mbp}\right|\lesssim (\beta_{\mbp,i}+\beta_{\mbp, j})(1+\|\hat\mbp\|_{\mbV_3})\mbE_{ij}. 
\eeqs
Here we also used Lemma \ref{lem-e_ij} to bound the $L^\infty$-norm of $\nabla_{s_\mbp} h$, or $\nabla_{s_\mbp}h$. We divide both sides of the equality \eqref{est-h-Theta-i,j-1}  by the quantity $\beta_{\mbp, j}^2-\beta_{\mbp, i}^2$  to obtain
\begin{equation}\label{est-h-Theta-i,j-2}
\begin{aligned}
 \int_{\msI_\mbp}h\left(\bm \gamma_\mbp''\right)\tilde \Theta_i\tilde \Theta_j   \dd \tilde s_{\mbp}\lesssim  
 \frac{1+\|\hat\mbp\|_{\mbV_3}}{|\beta_{\mbp,i}-\beta_{\mbp,j}|} \mbE_{ij}\lesssim  \frac{1+\|\hat\mbp\|_{\mbV_4^2}}{|\beta_{\mbp,i}-\beta_{\mbp,j}|}\mbE_{ij}.
 \end{aligned}
\end{equation}
Here we used embedding Lemma \ref{lem-est-V}. Moreover,  $\beta_{\mbp,k}= \frac{\beta_k}{(1+\mrp_0)R_0}$ by \eqref{ortho-tTheta}, the bound \eqref{est-Theta-2} follows from  \eqref{eq-wavenumgap} and \eqref{est-h-Theta-i,j-2}.
\end{proof}

The estimates of Lemma  \ref{lem-Theta} are  central to controlling the action of the operator $\mbL $ when restricted to the asymptotically large slow space $\mathcal Z$.
 A benefit of conducting our analysis in $\mathbb R^2$ is that single derivatives of the Laplace-Beltrami eigenmodes behave well. Indeed, from \eqref{def-Theta-beta} we have 
\beq\label{Theta'}
\tilde \Theta_i'= \left\{ \begin{array}{lc} -\beta_{\mbp,i}\tilde \Theta_{i+1}, & i \quad \hbox{odd},\\
                                                        \hspace{0.1in} \beta_{\mbp,i}\tilde \Theta_{i-1}, & i\quad \hbox{even},
                                                        \end{array}\right.
\eeq
which furthermore implies
\begin{equation}\label{Theta+(k)}
\tilde \Theta_i^{(k)}\in \mathrm {span}\{\tilde \Theta_i, \tilde \Theta_i'\}.
\end{equation}
The following lemma provides the asymptotic form of the restriction of $\mbL$. It uses shape parameter
\begin{equation}\label{def-S1}
\begin{aligned}
&S_1:=\int_{\mathbb R} W'''(\phi_0(z)) B_{1}(z)  \, |\psi_0 (z)|^2\, \mrd z,
\end{aligned}
\end{equation}
where $B_{\mbp,1}$ and $\phi_1=\phi_1(z_\mbp; \sigma)$ are introduced in \eqref{def-B+dp,k} and \eqref{def-phi1} respectivly. This parameter is independent of choice of $\mbp\in\cD_\delta.$
In addition, for $k=0,1$, we have the $\sigma$ dependent parameters
\beq\label{def-S2}
S_{2,k}(\sigma)= 2\int_{\mbR} W'''(\phi_0(z)) \phi_1(z; \sigma) |\psi_k(z)|^2 \dd z -\eta_1. 
\eeq

\begin{lemma}\label{lem-prop-bL_p}
Let $\mbp\in\cD_\delta$ with $\cD_\delta$ defined in \eqref{A-00}  and $\rho$ suitably small. The basis functions of the slow space $\left\{Z_{\mathbf p}^{I(k)k}, k\in \Sigma\right\}$  are  approximately orthonormal in $L^2$. More precisely
there exists a matrix $\mbE$ with $l^2_*$-norm one for which 
\begin{equation}\label{ortho-Z}
\begin{aligned}
\left<Z_{\mathbf p}^{I(i)i}, Z_{\mathbf p}^{I(j)j}\right>_{L^2}=
\left\{\begin{aligned} &\left(1+ \mrp_0 \right) \delta_{ij},  & I(i)=I(j);\\
& O\left( \varep^2,\varep^2 \|\hat\mbp\|_{\mbV_4^2}\right)\mbE_{ij}, \qquad  &I(i)\neq I(j).
\end{aligned}\right.
\end{aligned}
\end{equation}
Moreover,   the action of the linear operator $\mbL  $  restricted to the preliminary slow space $ \mathcal Z(\mbp, \rho)$ is given by  
$$\mathbb M_{ij}:=\left<\mbL  Z_{\mathbf p}^{I(i)i}, Z_{\mathbf p}^{I(j)j}\right>_{L^2},$$
whose entries have the leading order approximations
\begin{equation*}
\mathbb M_{ij}=\left\{\begin{aligned}
&(1+\mrp_0) \Big(\La_{0 i}^2 +\varep (\sigma S_1+\eta_d \la_0) +\varep S_{2,0}(\sigma) \La_{0i} \Big) +O(\varep^2)&\hbox{ $ i=j$ and $ I(i)=0$;}\\[3pt]
&\qquad  (1+\mrp_0 )\Big(\La_{1 i}^2 +\varep S_{2,1} (\sigma)\La_{1i} \Big)+O(\varep^2)\qquad   &\hbox{if $ i=j $ and $I(i)=1$;}\\
&\qquad\qquad    O\!\left(\varep^2\right)\mbE_{ij} 
&\hbox{ $i\neq j$ and $I(i)=I(j)$;}\\
&\qquad    O\!\left(\varep^2, \varep^2\|\hat\mbp\|_{\mbV_4^2}\right)\mbE_{ij} 
&\hbox{  $I(i)\neq I(j)$.} 
\end{aligned}\right.
\end{equation*}
\end{lemma}
\begin{remark}
In the absence of the asymptotic gap between $\Sigma_0$ and $\Sigma_1$, then the leading term in $\mathbb M_{ij}$ for $I(i)\neq I(j)$ generically increases to  $O\!\left(\eps\right)$.  
\end{remark}
\begin{proof} Using the localization of the basis functions, we establish the approximate orthonormality \eqref{ortho-Z}  by  integrating over $\Gamma_{\mathbf p}^{2\ell}$. Recalling that $\dd x=\tilde \mrJ \dd \tilde s_\mbp \dd z_\mbp$ with $\tilde{\mrJ}=\varep(1-\varep z_{\mathbf p}\kappa_{\mathbf p})$ in local coordinates, we write
\begin{equation}\label{est-ortho-Z-1}
\begin{aligned}
\left<Z_{\mathbf p}^{I(i)i}, Z_{\mathbf p}^{I(j)j}\right>_{L^2}&=\int_{\mathbb R_{2\ell}}\int_{\msI_\mbp}  \psi_{I(i)} \psi_{I(j)}  \tilde \Theta_i \tilde \Theta_j  \dd  \tilde s_{\mathbf p}\mrd z_{\mathbf p}-\varep\int_{\msI_\mbp} \kappa_{\mathbf p}\tilde \Theta_i\tilde  \Theta_j\dd \tilde s_\mbp \int_{\mathbb R_{2\ell}}\psi_{I(i)} \psi_{I(j)} z_{\mathbf p} \mrd z_{\mathbf p}.
\end{aligned}
\end{equation}
The orthogonality of $\{\tilde \Theta_i\}$ given in \eqref{ortho-tTheta} shows that the first term  on the right-hand side contributes the main $\delta_{ij}$ term in \eqref{ortho-Z}.
We claim second term on the right hand side can be bounded by
\beq\label{est-ortho-Z-2}
\varep\int_{\msI_\mbp} \kappa_{\mathbf p}\tilde \Theta_i\tilde  \Theta_j\dd \tilde s_\mbp\int_{\mathbb R_{2\ell}}\psi_{I(i)} \psi_{I(j)} z_{\mathbf p} \mrd z_{\mathbf p}= \left\{\begin{aligned}& 0, \quad& \hbox{if $I(i)=I(j)$}\\
& O( \varep^2 , \varep^2\|\hat\mbp\|_{\mbV_4^2} ) \mbE_{ij}, \qquad& \hbox{if $I(i)\neq I(j)$}.
\end{aligned}\right.
\eeq
 Indeed, if $I(i)=I(j)$ \eqref{est-ortho-Z-2} holds by parity  since  $|\psi_{I(i)} |^2z_{\mathbf p}$ is odd. On the other hand,  if $I(i)\neq I(j)$ we use estimate \eqref{est-Theta-2} from Lemma \ref{lem-Theta} to 
 bound the projection of $\kappa_\mbp$ to $\tilde \Theta_i\tilde \Theta_j$ in $L^2(\msI_\mbp)$, and \eqref{est-ortho-Z-2} follows.  
Returning back to \eqref{est-ortho-Z-1} implies the approximate orthogonality \eqref{ortho-Z}.
 
To establish the estimates of   $\mbL_\mbp$ on $\mathcal Z$ we apply the expansion  \eqref{exp-bL+d} of $\mbL $ to the inner product:
\begin{equation}\label{exp-bM-0}
\begin{aligned}
\left<\mbL  Z_{\mathbf p}^{I(i)i}, Z_{\mathbf p}^{I(j)j}\right>_{L^2}=&\left<\mbL_{0} Z_{\mathbf p}^{I(i)i},  Z_{\mathbf p}^{I(j)j}\right>_{L^2}+\varep\left<\mbL_{ 1} Z_{\mathbf p}^{I(i)i},  Z_{\mathbf p}^{I(j)j}\right>_{L^2}\\
&+\varep^2\left<\mbL_{\geq 2} Z_{\mathbf p}^{I(i)i},  Z_{\mathbf p}^{I(j)j}\right>_{L^2}.
\end{aligned}
\end{equation}
Recalling \eqref{bL0-Z} and employing the approximate orthogonality identity  \eqref{ortho-Z}, we obtain the leading order approximation
\begin{equation}\label{est-bM-0}
\left<\mbL_{0}Z_{\mathbf p}^{I(i)i}, Z_{\mathbf p}^{I(j)j}\right>_{L^2} = \left\{\begin{aligned}
&\left(1+\mrp_0\right)\La_{I(j)j}^2 \delta_{ij}+O\left( e^{-\ell\nu/\varep}\right)\mbE_{ij},  & I(i)=I(j); \\
&\qquad O\left( \varep^2, \varep^2 \|\hat\mbp\|_{\mbV_4^2}\right), & I(i)\neq I(j).
\end{aligned}\right.
\end{equation}
Estimates on $\mbL_{1}$ restricted to $\mathcal Z$ are more complicated. Recalling \eqref{def-bLp,1}, direct calculations establish
\beq\label{eq-bL{p,1}-1}
\begin{aligned}
&\mbL_{1}(\psi_{I(i)}\tilde \Theta_i)=\La_{I(i)i}\Big(\kappa_{\mathbf p} \psi_{I(i)}' \tilde \Theta_i +W'''(\phi_0)\phi_1 \psi_{I(i)}\tilde \Theta_i-z_{\mathbf p} \varep^2 D_{s_{\mathbf p, 2}}\tilde  \Theta_{i} \psi_{I(i)} \\
&\quad-\eta_1\psi_{I(i)}\tilde \Theta_i\Big) +\mathcal L \Big(\kappa_{\mathbf p} \psi_{I(i)}' \tilde \Theta_i+W'''(\phi_0)\phi_1 \psi_{I(i)} \tilde \Theta_i
 -z_{\mathbf p}\varep^2D_{s_{\mathbf p},2}\tilde  \Theta_{i} \psi_{I(i)} \Big)\\
 &\quad +W'''(\phi_0 )(\kappa_{\mathbf p} \phi_0'+\mathrm L_{0} \phi_1 )\psi_{I(i)}\tilde \Theta_i+\eta_d W''(\phi_0)\psi_{I(i)} \tilde \Theta_i.
 \end{aligned}
\eeq
Since the operators $D_{s_\mbp, 2}$ and $\mcL$ incorporate derivatives with respect to $\tilde s_\mbp$ scaled with $\varep$, we apply \eqref{Theta'}-\eqref{Theta+(k)} and
we separate into cases for $\tilde \Theta_i$ and $\tilde \Theta_i'$.  We also exploit the even and odd parity of functions with respect to $z_\mbp.$
We define functions $h_{1}(z_{\mathbf p}, \bm \gamma''_{\mathbf p})$ and $h_2(z_{\mathbf p}, \bm \gamma''_{\mathbf p})$, denoting higher order terms,  that enjoy the properties of Notation \ref{Notation-h}. 
With these steps the identity \eqref{eq-bL{p,1}-1} is rewritten as
\beq\label{eq-bL{p,1}-G}
\begin{aligned}
\mbL_{1}(\psi_{I(i)}\tilde \Theta_i)=&\left(g_1^\bot(z_{\mathbf p}, \bm \gamma''_{\mathbf p})+\varep h_1 (z_{\mathbf p}, \bm \gamma''_{\mathbf p}) \right)\tilde \Theta_i+ \left(g^\bot_2( \bm \gamma_{\mbp}'', z_{\mathbf p})+\varep h_2(\bm \gamma''_{\mbp}, z_{\mbp})\right)\varep\tilde \Theta_i'\\
&+g^*(z_{\mathbf p})\tilde \Theta_i
\end{aligned}
\eeq
where the functions $g_k^\bot=g_k^\bot(z_{\mbp}, \bm \gamma''_{\mbp})$  have opposite $z_\mbp$ parity of $\psi_{I(i)}$.  Hence they satisfy
\begin{equation}\label{eq-bL{p,1}-G-1}
\int_{\mathbb R_{2\ell}} g_k^\bot(z_\mbp, \bm \gamma_{\mbp}'')\psi_{I(i)}\dd z_{\mbp}=0, \qquad k=1,2.
\end{equation}
The $z_\mbp$ dependent only function  $g^*=g^*(z_{\mbp})$ is given explicitly by
\begin{equation}\label{def-g+*}
\begin{aligned}
g^*(z_{\mbp}):=&\La_{I(i)i}\left(W'''(\phi_0)\phi_1\psi_{I(i)}-\eta_1\psi_{I(i)}\right)+\left(\mrL_{0}+\varep^2 \beta_{\mbp,i}^2\right)\left(W'''(\phi_0)\phi_1\psi_{I(i)}\right)\\
&\qquad \qquad+W'''(\phi_0 )\mathrm L_{0} \phi_1 \psi_{I(i)} +\eta_d W''(\phi_0)\psi_{I(i)}.
\end{aligned}
\end{equation}
\vskip -0.1in
From \eqref{eq-bL{p,1}-G} we decompose the $(i,j)$-th component of the bilinear form of $\mbL_{  1}$ restricted to $\mathcal Z$ as
\begin{equation}\label{decomp-L{p,1}-Z}
\begin{aligned}
\left<\mbL_{1}Z_{\mathbf p}^{I(i)i}, Z_{\mathbf p}^{I(j)j}\right>_{L^2}=\mathcal I_0+ \mathcal I_1+\mathcal I_2+\mathcal I_3,
\end{aligned}
\end{equation}
where we have defined
\beq\label{def-I-0123}
\begin{aligned}
\mathcal I_0&:=\int_{\mathbb R_{2\ell}}g^*(z_\mbp)\psi_{I(j)}\dd z_\mbp \int_{\msI_\mbp}\tilde \Theta_i \tilde \Theta_j\dd \tilde s_\mbp,\\
\mathcal I_1&:=\int_{\mathbb R_{2\ell}}\int_{\msI_\mbp} \left(g_1^\bot(z_{\mathbf p}, \bm \gamma''_{\mathbf p})+\varep h_{1}(\bm \gamma''_{\mbp}, z_{\mbp}) \right)\psi_{I(j)}\tilde \Theta_i\tilde  \Theta_j\dd \tilde s_{\mbp}\dd z_\mbp,\\
\mathcal I_2&:=\int_{\mathbb R_{2\ell}}\int_{\msI_\mbp} \left(g_2^\bot(z_{\mathbf p}, \bm \gamma''_{\mathbf p})+\varep h_{2}(\bm \gamma''_{\mbp}, z_{\mbp}) \right)\psi_{I(j)} \varep\tilde\Theta_i' \tilde\Theta_j\dd \tilde s_{\mbp}\dd z_\mbp,\\
\mathcal I_3&:=-\varep \int_{\mathbb R_{2\ell}}\int_{\msI_\mbp}  \mbL_{1} (\psi_{I(i)}\tilde \Theta_i) \psi_{I(j)}\tilde \Theta_j z_{\mbp}\kappa_\mbp\dd \tilde s_{\mbp}\mrd z_{\mathbf p}.
\end{aligned}
\eeq 
In light of orthogonality \eqref{eq-bL{p,1}-G-1}, we see that $\mathcal I_1, \mcI_2, \mcI_3$ are higher order terms. Indeed, with the aids of Lemma \ref{lem-Theta}, \eqref{Theta'}-\eqref{Theta+(k)}, and
uniform bounds on $\varep \beta_{\mbp, i}$, a direct calculation establishes
\beq\label{est-I-123}
\mcI_1+\mcI_2+\mcI_3=O(\varep) \mbE_{ij}.
\eeq
From the orthogonality of $\{\tilde \Theta_i\}$ given in \eqref{ortho-tTheta}, the term $\mathcal I_0$, is zero unless $i=j$. As  $g^*=g^*(z_{\mbp})$ defined in \eqref{def-g+*} 
decays exponentially in $z_\mbp$, we  may decompose
\begin{equation}\label{decomp-g+*}
\int_{\mathbb R_{2\ell}} g^*(z_{\mbp})\psi_{I(i)}\dd z_{\mbp}=\mathcal I_{01}+\mathcal I_{02}+\mathcal I_{03} +Ce^{-\ell\nu/\varep},
\end{equation}
where we have introduced the sub-terms
\beqs
\begin{aligned}
\mathcal I_{01}&:=\La_{I(i)i}\int_{\mathbb R}\left(W'''(\phi_0)\phi_1\psi_{I(i)}-\eta_1\psi_{I(i)}\right)\psi_{I(i)}\dd z\\
&\quad +\int_{\mathbb R}\left(\mrL_{0}+\varep^2 \beta_{\mbp,i}^2\right)\left(W'''(\phi_0)\phi_1\psi_{I(i)}\right)\psi_{I(i)}\dd z,\\
\mathcal I_{02}&:=\int_{\mathbb R}W'''(\phi_0 )\mathrm L_{0} \phi_1 \psi_{I(i)}\psi_{I(i)} \dd z,\\ 
\mathcal I_{03}&:=\eta_d\int_{\mathbb R}W''(\phi_0)\psi_{I(i)}\psi_{I(i)}\dd z.
\end{aligned}
\eeqs
Proceeding term by term, we integrate by parts in the second integral of $\mathcal I_{01}$, rewriting it as
\begin{equation}\label{est-I-01}
\mathcal I_{01}=S_{2,I(i)}\La_{I(i)i}
\end{equation}
where $S_{2, I(i)}=S_{2,I(i)}(\sigma)$ depending on $\sigma$ is introduced in \eqref{def-S2}.  Recalling the definition \eqref{def-phi1} of $\phi_1$,  we separate $\mathcal I_{02}$,
\begin{equation}\label{est-I-02-0}
\mathcal I_{02}=\mathcal I_{02,1}+\mathcal I_{02,2}:=\sigma\int_{\mathbb R}W'''(\phi_0) B_{1} |\psi_{I(i)}|^2\dd z+\frac{\eta_d}{2}\int_{\mathbb R}W'''(\phi_0)z\phi_0'|\psi_{I(i)}|^2\dd z.
\end{equation}
From the definition of $\mathrm L_{0}$ we observe that
\begin{equation*}
W'''(\phi_{0})\phi_{0}'\psi_{k}=\la_k\psi_{k}'-\mathrm L_{0} \psi_{k}', \quad \hbox{and} \quad \mathrm L_{0}(z\psi_k)=z\la_k\psi_k -2\psi_k',
\end{equation*}
which together with the self-adjointness of  $\mrL_0$ on $L^2(\mbR)$ yield
\begin{equation}\label{est-I-02-1}
\begin{aligned}
\mathcal I_{02,2}&= \frac{\eta_d}{2}\int_{\mathbb R } \left(\la_{I(i)} \psi_{I(i)}'\psi_{I(i)}z-\psi_{I(i)}'\mathrm L_{0}(z\psi_{I(i)}  )\right)\, \mrd z\\
&=\eta_d\|\psi_{I(i)}'\|_{L^2(\mathbb R)}^2.
\end{aligned}
\end{equation}
When $I(i)=1$ we have $\psi_{I(i)}=\phi_0'/m_1$. Recalling the identify $W'''(\phi_0)|\phi_0'|^2=-\mathrm L_{0}\phi_0''$ from Lemma \ref{lem-L0} yields
\begin{equation}\label{est-I-02-2}
\int_{\mathbb R}W'''(\phi_0)B_{1}|\psi_1|^2\dd z=-\frac{1}{m_1^2}\int_{\mathbb R}B_{1}\mrL_{0}\phi_0''\dd z=-\frac{1}{m_1^2}\int_{\mathbb R} \phi_0''\dd z=0,
\end{equation}
and hence $\mathcal I_{02,1}=0$ when $I(i)=1$. Combining  the identity \eqref{est-I-02-1}  with \eqref{est-I-02-0}  we obtain
\begin{equation}\label{est-I-02-3}
\mathcal I_{02}=\sigma S_1\delta_{I(i)0}+\eta_d\|\psi_{I(i)}'\|_{L^2(\mathbb R)}^2,
\end{equation}
where  $S_1$ was introduced in \eqref{def-S1}. Finally, from the definitions of $\mathrm L_{0}$ and $\psi_{I(i)}$, $\mathcal I_{03}$  reduces to 
\beq
\label{est-I-03}
\mathcal I_{03}=\eta_d\int_{\mathbb R} (\mrL_{0}+\p_{z}^2)\psi_{I(i)}\psi_{I(i)}\dd z=\eta_d\la_{I(i)}-\eta_d \|\psi_{I(i)}'\|_{L^2(\mathbb R)}^2,
\eeq 
where $\psi_{I(i)}$ has been normalized in $L^2(\mathbb R)$. Combining estimates \eqref{est-I-01}, \eqref{est-I-02-3} and \eqref{est-I-03} with \eqref{decomp-g+*} yields for some bounded $\mbp$-independent  constant $C$, 
\beq\label{eq-bL{p,1}-G-2}
\int_{\mathbb R_{2\ell}} g^*(z_{\mbp})\psi_{I(i)}(z_{\mbp})\dd z_{\mbp}=(\sigma S_1+\eta_d\la_0)\delta_{I(i)0}+S_{2,I(i)}\La_{I(i)i}+Ce^{-\ell\nu/\varep },
\eeq
which  combined with the orthogonality \eqref{ortho-tTheta}  and $\mathcal I_0$ defined in \eqref{def-I-0123} furthermore implies
\beq\label{S-est-I-0}
\mathcal I_0=\left(1+\mrp_0\right)\left[(\sigma S_1+\eta_d\la_0)\delta_{I(i)0} +S_{2,I(i)} \La_{I(i)i}+O(e^{-\ell\nu/\varep})\right]\delta_{ij}.
\eeq
Combining estimates \eqref{S-est-I-0} and \eqref{est-I-123}  with \eqref{decomp-L{p,1}-Z} imply
\begin{equation}\label{est-L{p,1}-Z}
\left<\mbL_{1}Z_{\mathbf p}^{I(i)i}, Z_{\mathbf p}^{I(j)j}\right>_{L^2}=\left\{
\begin{aligned}
&\left(1+\mrp_0\right)\left[\left(\sigma S_1+\eta_d\la_0\right) \delta_{I(i)0}+S_{2,I(i)}\La_{I(i)} \right],  &\hbox{$i=j$;}\\
&\qquad  O(\varep)\mbE_{ij}, &  \hbox{$i\neq j$;}
\end{aligned} 
\right.
\end{equation}
To address the  bilinear form induced by $\mbL_{ \geq 2}$ we employ \eqref{Theta+(k)} to arrive at the general form
\beqs
\mbL_{  \geq 2}Z_{\mbp}^{I(i)i}=\varep^{-1/2}\left(h_1(z_{\mbp}, \bm \gamma_{\mbp}'')\tilde \Theta_i+h_2(z_{\mbp}, \bm \gamma_{\mbp}'')\varep\tilde \Theta_i'\right),
\eeqs
where the functions $h_1$ and $h_2$ enjoy the properties of Notation\,\ref{Notation-h} and are localized near $\Gamma_\mbp$. Integrating out $z_\mbp$, and employing Lemma \ref{lem-Theta}, \eqref{Theta'} and the uniform bounds on $\varep \beta_i$ we deduce 
\begin{equation}\label{est-L{p,2}-Z}
\left|\left<\mbL_{ \geq  2}Z_{\mbp}^{I(i)i}, Z_{\mbp}^{I(j)j}\right>\right|\lesssim \mbE_{ij}.
\end{equation}
The conclusion follows from \eqref{exp-bM-0}, the  estimates \eqref{est-bM-0} and  \eqref{est-L{p,1}-Z}-\eqref{est-L{p,2}-Z}.
\end{proof}


\subsection{Modified approximate slow spaces}

The modified spaces are corrections to the preliminary spaces that make them closer to being an invariant subspace of $\mbL$. This provides the better control required to close the nonlinear estimates of Section\,\ref{s:nonstab-BLM}.
\begin{lemma}\label{lem-def-Z*}
For $i\in \Sigma$, there exist functions $\varphi_{k,i}=\varphi_{k,i}(z_{\mathbf p}, \bm \gamma''_{\mathbf p})(k=1,2)$  localized near $\Gamma_\mbp$ that enjoy the properties of Notation\,\ref{Notation-h} for which
\begin{equation}\label{ortho-varphi}
\int_{\mathbb R_{2\ell }}\varphi_{k,i}(z_\mbp,\bm \gamma_{\mbp}'') \psi_{I(i)}(z_\mbp)\dd z_\mbp=0,  \qquad k=1,2.
\end{equation}
The modified basis functions
\begin{equation}\label{eq-def-Z*}
\begin{aligned}
Z_{\mathbf p,*}^{I(i)i}&:=\left(\tilde \psi_{I(i)} +\varep \tilde \varphi_{1,i} \right)\tilde \Theta_i+\varep\tilde\varphi_{2,i}  \varep \tilde \Theta_i'=\varep^{-1/2}\left[\left(\psi_{I(i)} +\varep \varphi_{1,i} \right)\tilde \Theta_i+\varep \varphi_{2,i} \varep\tilde \Theta_i'\right],
\end{aligned}
\end{equation}
 are  $\mbL   $ is invariant up to order $\varep^2$ in $L^2(\Omega)$, satisfying
\begin{equation}\label{bL-pZ-*}
\begin{aligned}
\mbL  Z_{\mathbf p, *}^{I(i)i} =&\Big(\La_{I(i)i}^2 +\varep\delta_{I(i)0} (\sigma S_1+\eta_d \la_0)+S_{2,I(i)}(\sigma)\La_{I(i)i} \Big)Z_{\mathbf p, *}^{I(i)i}+\varep^{3/2}\left(h_1\tilde \Theta_i+h_2\varep\tilde \Theta_i'\right)\\
&+\varep^{3/2}\sum_{k=1}^4\left(\varep^{k-1}\p_{s_\mbp}^k h_{3,k}\tilde \Theta_i+\varep^{k-1}\p_{s_{\mbp}}^k h_{4,k}\varep\tilde \Theta_i'\right).
\end{aligned}
\end{equation}
Here the functions $h=h(z_{\mbp}, \bm \gamma_{\mbp}'')$  are localized  near $\Gamma_\mbp$, enjoy the properties of Notation\,\ref{Notation-h}, and have $L^2(\Omega)$ norm of $O(\sqrt{\varep}).$
\end{lemma}
\begin{proof}
To establish the Lemma it suffices to construct $\varphi_{k,i}$  in the interior region as the dressing process incorporates only exponentially small errors. 
Using the expansion  \eqref{exp-bL+d} of $\mbL  $, we compute
\begin{equation}\label{eq-bLp-Z*}
\begin{aligned}
\mbL  Z_{\mbp, *}^{I(i)i}=&\mbL _{0}Z_{\mbp}^{I(i)i}+\varep \cdot \varep^{-1/2}\Big(\mbL_{1} (\psi_{I(i)}\tilde \Theta_i)+\mbL _{0} (\varphi_{1,i} \tilde \Theta_i)+\mbL_{ 0}(\varphi_{2,i}\varep \tilde \Theta_i')\Big)\\
&+\varep^{2}\cdot\varep^{-1/2}\left(\mbL_{1}(\varphi_{1,i}\tilde \Theta_i)+\mbL_{ 1}(\varphi_{2,i}\varep \tilde \Theta_i')+\varep^{1/2}\mbL_{\geq 2} Z_{\mathbf p, *}^{I(i)i}\right).
\end{aligned}
\end{equation}
The first term is calculated as in \eqref{bL0-Z}. Since $\mbL_{ 0}=\mathcal L ^2$, from \eqref{tTheta''} we see that
\begin{equation}\label{est-bLp-Z*-1}
\mbL_{ 0}(\varphi_{k,i}\tilde \Theta_i)=\left(\mrL_{0}+\varep^2 \beta_{\mbp,i}^2\right)^2\varphi_{k,i} \tilde \Theta_i+\left(\mathcal L^2-\left(\mrL_{0}+\varep^2 \beta_{\mbp,i}^2\right)^2 \right)\left(\varphi_{k,i}\tilde\Theta_i\right).
\end{equation}
We show that $\varphi_{k,i}=\varphi_{k,i}(z_{\mbp}, \bm \gamma_{\mbp}'')$ in the sense of Notation\,\ref{Notation-h}, and consequently, in \eqref{error-Lp}, bound the second term of \eqref{est-bLp-Z*-1}. 

It remains to determine $\varphi_{k, i}$ for which the $\varep$-order term in \eqref{eq-bLp-Z*} equals $\La_{I(i)i}^2(\varphi_{1,i}\Theta_i+\varphi_{2,i} \varep\Theta_i')$ to leading order.
 From \eqref{eq-bL{p,1}-G},  we define $\varphi_{k,i}(\cdot, \bm \gamma_\mbp'')$   as the $L^2(\mathbb R)$ solutions to
\begin{equation}\label{def-varphi}
\begin{aligned}
\Bigl( \left(\mrL_{0}+\varep^2 \beta_{\mbp,i}^2\right)^2-\La_{I(i)i}^2\Bigr)\varphi_{k,i}=-g_{k}^\bot(z, \bm \gamma_{\mbp}'')+\delta_{1I(i)} \Big(\left( \delta_{I(i)0}(\sigma S_1+\eta_d\la_0)+S_{2,I(i)}\La_{I(i)i}\right)\psi_{I(i)}-g^*(z)\Big).
\end{aligned}
\end{equation}
in the subspace perpendicular to $\psi_{I(i)}$.  The definition is well posed since \eqref{eq-bL{p,1}-G-1} and \eqref{eq-bL{p,1}-G-2} imply that the right-hand side of the identity is orthogonal to $\psi_{I(i)}$ in $L^2(\mathbb R)$.  Dressing these functions on $\Gamma_\mbp$, we extend $\varphi_{k,i}$ to $\Omega$. 
Applying \eqref{est-bLp-Z*-1}, identity \eqref{eq-bL{p,1}-G}  and \eqref{tTheta''} implies
\begin{equation*}
\begin{aligned}
&\mbL_{1}(\psi_{I(i)}\tilde \Theta_i)+\mbL _{0} (\varphi_{1,i}\tilde \Theta_i)+\mbL _{0} (\varphi_{2,i}\tilde\Theta_i)\\
&=\Big(\delta_{I(i)0}(\sigma S_1+\eta_d\la_0)+S_{2,I(i)}\La_{I(i)i}\Big)\psi_{I(i)}\tilde \Theta_i +\La_{I(i)i}^2 (\varphi_{1,i}\tilde\Theta_i+\varphi_{2,i}\tilde \varep\Theta_i')\\
&\quad
+\left(\mathcal L ^2-\left(\mrL_{0}+\varep^2 \beta_{\mbp,i}^2\right)^2\right)(\varphi_{1,i}\tilde \Theta_i+\varphi_{2,i}\varep\tilde\Theta_i') \\
&\quad +\varep \left(h_1(z_{\mbp}, \bm \gamma_{\mbp}'')\tilde \Theta_i+h_2(z_{\mbp},\bm \gamma_{\mbp}'') \varep\tilde\Theta_i'\right).
\end{aligned}
\end{equation*}
Returning this expansion to \eqref{eq-bLp-Z*}, we obtain 
\begin{equation}\label{bL-pZ-*-1}
\begin{aligned}
\mbL  Z_{\mathbf p, *}^{I(i)i} =&\Big(\La_{I(i)i}^2 +\varep\delta_{I(i)0} (\sigma S_1+\eta_d \la_0)+S_{2,I(i)}\La_{I(i)i} \Big)Z_{\mathbf p, *}^{I(i)i}+\varep^{2}\mbL_{1}\Big(\tilde\varphi_{1,i}\tilde \Theta_i\\
&+\tilde\varphi_{2,i}\varep \tilde \Theta_i'\Big)+\varep^2\mbL_{\geq 2}Z_{\mathbf p, *}^{I(i)i} +\varep^2\cdot \varep^{-1/2} \left(h_1(z_{\mbp}, \bm \gamma_{\mbp}'')\tilde \Theta_i+h_2(z_{\mbp},\bm \gamma_{\mbp}'') \varep\tilde \Theta_i'\right)\\
&+\varep\left(\mathcal L ^2 -\left(\mrL_{0}+\varep^2 \beta_{\mbp,i}^2\right)^2\right)(\tilde\varphi_{1,i}\tilde \Theta_i+\tilde\varphi_{2,i}\varep\tilde\Theta_i').
\end{aligned}
\end{equation}
Expanding the operators $\mbL_{1}$ and $\mbL_{\mathbf p,\geq2}$, and using  \eqref{Theta'},  we write the second and third terms as
\begin{equation*}
\mbL_{1}\left(\tilde\varphi_{1,i}\tilde \Theta_i+\tilde\varphi_{2,i}\varep\tilde \Theta_i'\right)+\mbL_{  \geq 2} Z_{\mbp, *}^{I(i)i}=\varep^{-1/2}\left(h_1(z_{\mbp}, \bm \gamma_{\mbp}'')\tilde \Theta_i+h_2(z_{\mbp},\bm \gamma_{\mbp}'')\varep\tilde \Theta_i'\right),
\end{equation*}
where $h_1$ and $h_2$ are new general functions. The conclusion follows from this identity, \eqref{bL-pZ-*-1},  and the relation 
\beq\label{error-Lp}
\begin{aligned}
&\Big(\mathcal L ^2 -\left(\mrL_{0}+\varep^2 \beta_{\mbp,i}^2\right)^2\Big)(\tilde\varphi_{1,i}\Theta_i+\tilde\varphi_{2,i}\varep\Theta_i')\\
=&\sum_{k=1}^4\left(\varep^k\p_{s_\mbp}^k h_{3,k}(z_{\mbp}, \bm \gamma_{\mbp}'')\Theta_i+\varep^k\p_{s_{\mbp}}^k h_{4,k}(z_{\mbp}, \bm \gamma_{\mbp}'')\varep\Theta_i'\right).
\end{aligned}
\eeq
Here we note the dependence of $\varphi_{k,i}$  on $\tilde s_\mbp$ are  uniform on $i$ by its definition from \eqref{def-varphi} and hence we omit the dependence of $h$s on $i$ by abusing notation.
\end{proof}
Note the dependence of $\varphi_{k,j}$ on $s_\mbp$ are uniform in $j$, we may use this fact without further  mention. 

The modified approximate slow spaces are defined as the spans of the modified basis functions of \eqref{eq-def-Z*}: 
\begin{equation}\label{def-Z*}
\mathcal Z_{*}(\mbp, \rho):=\mathcal Z_{*}^0(\mbp, \rho)\cup \mathcal Z_{*}^1(\mbp, \rho) \qquad  \hbox{with}\qquad \mathcal Z_{*}^k(\mbp, \rho)=\mathrm{span}\left\{Z_{\mbp, *}^{I(i)i}, i\in \Sigma_k\right\}.
\end{equation} 
Similarly as we used for the leading order slow spaces, we utilize $\mcZ_*^k, \mcZ_*$ to simplify the notation when there is no ambiguity.  
When restricted to $\mathcal Z_*$ the bilinear form of the full linearized operator $\Pi_0\mbL  \big|_{\mathcal Z_*}$, induces an $N\times N$ matrix $\mathbb M^*$ with entries
 \begin{equation}\label{def-M+*}
 \mathbb M_{ij}^*=\left<\Pi_0 \mbL   Z_{\mathbf p, *}^{I(i)i}, Z_{\mathbf p, *}^{I(j)j}\right>_{L^2}.
 \end{equation}
By construction, $\varphi_{I(i), i}$ are perpendicular to $\psi_{I(i)}$, see \eqref{ortho-varphi}. Following the arguments that establish \eqref{ortho-Z}  it is easy to verify  that under assumption \eqref{A-00}
\begin{equation}\label{ortho-Z*}
\begin{aligned}
&\left<Z_{\mathbf p, *}^{I(i)i}, Z_{\mathbf p, *}^{I(j)j}\right>_{L^2}= \left\{ \displaystyle\begin{array}{cc}\left(1+\mrp_0\right)\delta_{ij}+O\left(\varep^2, \varep^{2}\|\hat{\mathbf p}\|_{\mbV_2^2}\right)\mbE_{ij},  &I(i)=I(j);\\[2pt]
 \qquad O\left(\varep^2,  \varep^2\|\hat{\mathbf p}\|_{\mbV_4^2}\right)\mbE_{ij}, \qquad\qquad \qquad  &I(i)\neq I(j).
  \end{array}\right.
  \end{aligned}
\end{equation}  
From the definition of the zero-mass projection $\Pi_0$, the identity \eqref{def-M+*} can be written as
\begin{equation}\label{def-M*ij-1}
\mathbb M_{ij}^*=\left<\mbL   Z_{\mathbf p, *}^{I(i)i}, Z_{\mathbf p, *}^{I(j)j}\right>_{L^2}-\frac{1}{|\Omega|}\int_\Omega \mbL   Z_{\mathbf p, *}^{I(i)i}\dd   x\int_\Omega Z_{\mathbf p, *}^{I(j)j}\dd   x.
\end{equation}
To estimate $\mathbb M^*_{ij}$, use the following Corollary to control the mass of $Z_{\mbp, *}^{I(j)j}$ and its image under $\mbL$.
\begin{cor} Under assumption \eqref{A-00} there exists a unit vector $\bm e=(e_j)_{j\in \Sigma}$ such that for $j\in \Sigma$,
\begin{equation}\label{est-Z*-mass}
\int_\Omega Z_{\mathbf p, *}^{I(j)j} \dd   x
=O(\varep^{3/2})\, e_j.
\end{equation}
Furthermore, 
\beq\label{mass-bL-pZ-*}
\int_\Omega \mbL   Z_{\mbp,*}^{I(j)j}\dd x=O\left(\varep^{3/2}(1+\|\hat{\mbp}\|_{\mbV_3^2})\right)e_j.
\eeq
\end{cor}
\begin{proof}
With $Z_{\mbp, *}^{I(j)j }$ introduced in \eqref{eq-def-Z*}, we have
\begin{equation*}
\begin{aligned}
\int_\Omega Z_{\mbp, *}^{0j}\dd x=&\varep^{1/2}\int_{\mathbb R_{2\ell}} \psi_{I(j)}(z_\mbp) \dd z_\mbp \int_{\msI_\mbp} \tilde \Theta_j\dd  \tilde s_{\mbp}\\
&+  \varep^{3/2}\int_{\mathbb R_{2\ell}}\int_{\msI_\mbp} \left(\varphi_{1,j}\tilde \Theta_j+ \varphi_{2,j} \varep\tilde \Theta_j'\right)(1-\varep z_{\mbp}\kappa_{\mbp})\dd \tilde s_\mbp \mrd z_{\mbp}.
\end{aligned}
\end{equation*}
The integration of $\tilde \Theta_j$ with respect to $\tilde s_\mbp$ is zero for $j\in \Sigma\setminus\{0\}$ while  for $j=0$, $\psi_{I(j)} =\psi_1 $ has odd parity in $z_\mbp$. We deduce that the first term on the right-hand side is zero.   After integrating with respect to $z_\mbp$,  the second integral takes the form
\begin{equation}
\begin{aligned}
\varep^{3/2}\left(\int_{\msI_\mbp} h_1(\bm \gamma_{\mbp}'') \tilde \Theta_j\dd \tilde s_\mbp+ \int_{\msI_\mbp} h_2(\bm \gamma_{\mbp}'')\varep \tilde \Theta_j'\dd \tilde s_\mbp\right).
\end{aligned}
\end{equation}
The estimate \eqref{est-Z*-mass} follows from \eqref{def-e-i} in Lemma \ref{Notation-e_i,j} and \eqref{est-h-p-V-2+2}.   
   To derive \eqref{mass-bL-pZ-*} we employ Lemma \ref{lem-def-Z*} and the estimate  \eqref{est-Z*-mass}.
The error bound involves the $\mathbb V_3^2$-norm instead of  the $\mbV_2^2$-norm of $\hat\mbp$ because there is an additional higher derivative acting on $h=h(\bm \gamma_\mbp'')$ 
as shown in \eqref{bL-pZ-*}. 
\end{proof}

  Applying the orthogonality and mass estimates  \eqref{ortho-Z*} and \eqref{est-Z*-mass} to  \eqref{def-M*ij-1} yields the expansion of $\mathbb M_{ij}^*$. This principle result gives a sharp characterization of the behaviour of the linearized operator on the modified slow space, which we summarize below.

  \begin{prop} 
  \label{prop-M*} 
  For $i,j\in \Sigma$,  the $\mathbb M^*$ with components $\mathbb M_{ij}^*$ defined in \eqref{def-M+*} can be approximated by
 \begin{equation}\label{est-bM*ij}
\mathbb M_{ij}^*= \left\{\begin{aligned}
 &(1+\mrp_0)\left( \La_{0i}^2+\varep(\sigma S_1 +\eta_d\la_0) +\varep S_{2,0}\La_{0i}\right)+O(\varep^2)& \hbox{if $i=j, I(i)=0$};\\
 &\qquad\qquad  (1+\mrp_0)\left(\La_{1i}^2+\varep S_{2,1}\La_{1i}\right) +O(\varep^2)&\hbox{if $i=j, I(i)=1$};\\
 &\qquad\qquad\qquad    O\!\left(\varep^2\right)\mbE_{ij}
 &\hbox{if $i\neq j, I(i)=I(j)$};\\
 &\qquad  \qquad   O\!\left(\varep^2, \varep^2\|\hat\mbp\|_{\mbV_4^2}\right)\mbE_{ij}
 &\hbox{if $ I(i)\neq I(j)$},
  \end{aligned}\right. 
 \end{equation}
where the matrix $\mathbb E$ is norm-one as an operator from $l^2(\mbR^{N})$ to $l^2(\mbR^{N})$.
 \end{prop}
We decompose $\mathbb M^*$ into a block structure corresponding to the pearling and meandering spaces,
\begin{equation}
\mathbb M^*=\left(\begin{array}{cc}
\mathbb M^*(0,0) & \mathbb M^*(0,1)\\
\mathbb M^*(1,0) &\mathbb M^*(1,1)
\end{array}\right),\qquad \mathbb M_{ij}^*(k,l)=\mathbb M^*_{ij} \quad \hbox{for}\;\;\,i\in \Sigma_k, j\in \Sigma_l.
\end{equation}
Since matrix $\mathbb E$ is norm-one, the  $N_0\times N_0$ subblock matrix $\mathbb M^*(0,0)$ is diagonally  dominant. 
In particular, under the {\sl pearling stability condition} 
\beq\label{cond-P-stab}
(\mathbf{PSC}) \qquad  \sigma S_1+\eta_d \la_0>0,
\eeq
$\mbM^*(0,0)$ is positive definite.  This pearling-mode coercivity is formulated in the following Lemma.
\begin{lemma}\label{lem-eigen-bM}
Assume $\sigma(\mbp)$ given by \eqref{def-hatla} is uniformly bounded, independent of $\varep>0$, for all  $\mbp\in\cD_\delta$. Then there exists $\varep_0$ sufficiently small such that for all $\varep\in(0,\varep_0)$ and for all $\mbp\in\cD_\delta$ for which the pearling stability condition \eqref{cond-P-stab} holds, we have 
\begin{equation*}
\mbq^T \mbM^*(0,0) \mbq \geq  \frac{\varep}{4} (\sigma S_1+\eta_d \la_0) \|\mbq\|_{l^2}^2,\quad \forall \mbq \in l^2(\mbR^{N_0}).
\end{equation*}
\end{lemma}
\begin{proof}
 The constants $S_{2,0}$ and $S_{2,1}$ in \eqref{est-bM*ij} depend upon $\sigma$, but are uniformly bounded, independent of $\varep$ since $\sigma$ is bounded by assumption. In view of the expansion of $\mbM_{ij}^*$ from \eqref{est-bM*ij} for $i,j\in \Sigma_0$, the Lemma follows for $\varep<\varep_0$ sufficiently small, by completing the square in $\La_{I(i)i}$ in the diagonal terms and using  the pearling stability condition $(\mathbf{PSC})$, the uniform bounds on $S_{2,0}, S_{2,1}$, and the diagonal dominance of $\mathbb M^*(0,0).$
\end{proof}

We denote by the $L^2$ projections to the finite-dimensional slow spaces $\mcZ_*^0, \mcZ_*^1$ and $\mcZ_*$ by $\Pi_{\mcZ_*^0}, \Pi_{\mcZ_*^1}, \Pi_{\mcZ_*}$ respectively. Introducing the $H^2$ inner norm
\beq
\label{H2-inner}
\| u \|_{\Htwoin}:=\|u\|_{L^2}+ \varep^2 \|u\|_{H^2}. 
\eeq
then we have the following result with regards to these projections. We state it for $\mcZ_*^0$, similar statements hold for $\mcZ_*$ and $\mcZ_*^1.$
\begin{lemma}\label{lem-proj-Z0} Suppose $\mbp\in \cD_\delta$ satisfying $\varep^2\|\hat\mbp\|_{\mbV_4^2}\leq \delta$. If $\delta$ is sufficiently small then for any $u\in L^2$  there exists a unique $N_0$-vector $\mathbf q=(\mathrm q_j)_j\in l^2$ such that  $Q:=\Pi_{\mcZ_*^0}u$, can be expressed as 
\begin{equation}
\label{def-Q}
Q:=\sum_{j\in \Sigma_0}\mathrm q_j Z_{\mathbf p, *}^{0j}. 
\end{equation}
Moreover, there exists $\varep_0$ suitably small such that for all $\varep\in(0,\varep_0)$ we have the relations
  \beqs
 \|\mbq\|_{l^2} \lesssim \|u\|_{L^2}; \qquad 
 \|Q\|_{\Htwoin}\sim \|Q\|_{L^2}\sim \|\mbq\|_{l^2}. 
  \eeqs
The parameters $\delta,\varep_0$ depend only upon the domain,  the system parameters, and the choice of $K_0,\ell_0$.
 \end{lemma}
 \begin{proof} 
 For any $u\in L^2(\Omega)$, the $L^2$ linear projection  $Q:=\Pi_{\mcZ_*^0}u\in \mcZ_*^0$ is well-defined by the Projection theorem,  and hence there exists $\mbq=(\mrq_j)\in l^2$ satisfying \eqref{def-Q}. In particular,  the vector $\mathbf q=(\mathrm q_j)$ satisfies the linear algebraic system
 \beqs
 \sum_{j\in \Sigma_0} \mathrm q_j \left<Z_{\mbp, *}^{0j}, Z_{\mbp, *}^{0k}\right>_{L^2} = \left<u,  Z_{\mbp, *}^{0k}\right>_{L^2}, \qquad \forall k\in \Sigma_0.
 \eeqs
 Due to the approximate orthogonality afforded by \eqref{ortho-Z*} and the bound $\varep^2 \|\hat\mbp\|_{\mbV_4^2}\leq  \delta$ with $\delta$ suitably small, there exists a unique $\mbq$ solving the system and $\mbq$ can be bounded in terms of $L^2$-norm of $u$ as 
 \beqs
 \|\mbq\|_{l^2} \lesssim \|u\|_{L^2}. 
 \eeqs 
 It remains to show the norm equivalences among $Q$ and $\mbq$. First, the equivalence of the $L^2$-norm of $Q$ and $l^2$-norm of $\mbq$ follows directly from the orthogonality relation \eqref{ortho-Z*} which requires
the condition $\|\hat\mbp\|_{\mbV_2^2}\lesssim \|\hat\mbp\|_{\mbV_2}\lesssim 1$ and $\varep_0$ suitably small. The set  $\{\varep^2\Delta Z_{\mbp, *}^{0j}\}_{j\in N_0}$ is approximately $L^2(\Omega)$  orthogonal due to the local coordinates Laplacian expansion  \eqref{eq-Lap-induced}, the form of $Z_{\mbp,*}^{0j}$ and Lemma \ref{Notation-e_i,j}. Combining these implies
 \beqs
 \|\varep^2\Delta Q\|_{L^2} \sim \|\mbq\|_{l^2},
 \eeqs 
 and the Lemma follows. 
 \end{proof}
 We call  $Q=\Pi_{\mcZ_*^0}u$ and the associated vector $\mathbf q=(\mathrm q_j)_j\in l^2$ defined through \eqref{def-Q} the pearling mode component and pearling parameters of $u$, respectively. The relations \eqref{est-Z*-mass} and \eqref{mass-bL-pZ-*}  imply
\beq\label{est-Q-mass}
\int_\Omega Q\dd x =O\left(\varep^3\|\mbq\|_{l^2}\right), \qquad \int_\Omega \mbL_\mbp Q\dd x=O\left(\varep^3 (1+\|\hat\mbp\|_{\mbV_3^2})\|\mbq\|_{l^2}\right).
\eeq
We present our principle results on the linear coupling between the pearling-meander and slow-fast modes.

\begin{thm}\label{thm-coupling est}
Assume that  $\rho, \delta$ are suitably small depending on the domain $\Omega$, the system parameters, and the choice of $K_0, \ell_0$. Then the following results hold uniformly for all $\mbp\in\cD_\delta,$ defined in \eqref{A-00}.   
\begin{enumerate}
\item  All $Q$ in the pearling slow space $\mcZ_{*}^0$ take the form \eqref{def-Q} and satisfy $\|Q\|_{\Htwoin} \sim \|\mbq\|_{l^2}$. Moreover the pearling-meander coupling satisfies the bound
\begin{equation*}
 \|\Pi_{\mathcal Z_*^1}\Pi_0 \mbL  Q\|_{L^2}\lesssim \left( \varep^2 + \varep^2 \|\hat\mbp\|_{\mbV_{4}^2}\right)\|\mathbf q\|_{l^2}.
\end{equation*}
\item  For any  function $v\in H^2$, the slow-fast coupling satisfies the bound 
\beqs
 \|\Pi_{\mathcal Z_*}^\bot\mbL   \Pi_{\mathcal Z_*}v\|_{L^2} +\|\Pi_{\mathcal Z_*}\mbL   \Pi_{\mathcal Z_*}^\bot v\|_{L^2}\lesssim \left(\varep^2+ \varep^{2}\|\hat\mbp\|_{\mbV_4^2}\right)\|v\|_{L^2}.
\eeqs
\end{enumerate}
\end{thm}
\begin{proof}
We address the bounds in term.
\begin{enumerate}
\item The equivalence of the $\Htwoin$ and $l^2$ norms, $\|Q\|_{\Htwoin}\sim \|\mbq\|_{l^2}$ is established in Lemma \ref{lem-proj-Z0}.  For the pearling-translation coupling estimate we remark that
\begin{equation*}
\|\Pi_{\mathcal Z_*^1} \Pi_0\mbL   Q\|_{L^2}=\left(\sum_{i\in \Sigma_1}\left<\Pi_0\mbL   Q, Z_{\mathbf p, *}^{1i}\right>_{L^2}^2\right)^{1/2}=\|\mathbb M^*(0,1)\mbq\|_{l^2}.
\end{equation*}
Applying \eqref{est-bM*ij} for the case $I(i)\neq I(j)$ yields the first bound.

\item To establish the second bound it suffices to show for any $v\in \mcZ_{*}, w\in \mcZ_{*}^\bot$, we have
\beq\label{est-bLvw}
\left<\mbL   v, w\right>_{L^2}  \lesssim  \left(\varep^2+ \varep^{2}\|\hat\mbp\|_{\mbV_4^2}\right)\|v\|_{L^2}\|w\|_{L^2}.
\eeq
Writing $v\in \mcZ_*$ in the form $v= \sum_{i\in \Sigma} v_i Z_{\mbp, *}^{I(i)i}$ for  $\{v_i\} \in \mathbb R^N$, we obtain
\beq\label{est-bLvw-1}
\left<\mbL   v, w\right>_{L^2} = \sum_i v_i\left<\mbL_\mbp Z_{\mbp, *}^{I(i)i},  w\right>_{L^2}.
\eeq
We consider each component in the summation. Utilizing  Lemma \ref{lem-def-Z*} and the orthogonality of $w$ and $\mathcal Z_*$ implies
\begin{equation}\label{est-bLvw-2}
\begin{aligned}
\left<w, \mbL Z_{\mathbf p,*}^{I(j)j}\right>_{L^2}=\varep^2\left<w,  \varep^{-1/2}(h_1\tilde \Theta_j+h_2\varep\tilde \Theta_j')\right>_{L^2}\h{80pt}\\
+\varep^{3/2}\sum_{k=1}^4\left<w, \varep^{k-1}\p_{s_{\mbp}}^kh_{3,k}\tilde \Theta_i+ \varep^{k-1}\p_{s_{\mbp}}^k h_{4,k} \varep\tilde \Theta_i'\right>_{L^2},
\end{aligned}
\end{equation}
where the functions $h=h(z_{\mbp}, \bm \gamma_{\mbp}'')$ enjoy the properties of Notation\,\ref{Notation-h}, localized in $\Gamma_{\mbp}^{2\ell}$  and   can be bounded in two ways, 
\beq\label{up-bd-h}
\|\varep^k\p_{s_\mbp}^kh\|_{L^\infty}\lesssim 1+ \|\hat\mbp\|_{\mbV_2}\lesssim 1,\qquad 
\|\varep^{k-1}\p_{s_\mbp}^kh\|_{L^\infty}\lesssim 1+\|\hat\mbp\|_{\mbV_3}\lesssim 1+\|\hat\mbp\|_{\mbV_4^2}.
\eeq
Inserting \eqref{est-bLvw-1}-\eqref{up-bd-h} into \eqref{est-bLvw} and using Lemma \ref{Notation-e_i,j} completes the proof. 
\end{enumerate}
\end{proof}
We extend these results to the full linearization $\Pi_0 \mbL_\mbp$ of the mass preserving flow \eqref{eq-FCH-L2} at $\Phi_\mbp.$

\begin{cor}\label{cor-est-w-Z*}
Under the same assumptions as Theorem \ref{thm-coupling est}, if   $w\in \mcZ_{*}^\bot$ and $\mbq$ is such that $w+Q$ is mass free, with $Q$ as in \eqref{def-Q}, then
\begin{equation*}
 \begin{aligned}
 \|\Pi_{\mathcal Z_*}\Pi_0\mbL  w\|_{L^2} & \lesssim \left(\varep^2+ \varep^{2}\|\hat\mbp\|_{\mbV_4^2}\right)\|w\|_{L^2}+\varep^3\|\mathbf q\|_{l^2}.
 \end{aligned}
\end{equation*}
\end{cor}
\begin{proof}
As in the slow-fast coupling estimate of Theorem\,\ref{thm-coupling est},  we need show for any $v \in \mcZ_*$, $w\in \mcZ_{*}^\bot,$ and $\mbq$ as above, that 
\beq\label{est-PibLvw}
\left<\Pi_0\mbL   w,  v\right>_{L^2}  \lesssim \left[  \left(\varep^2+ \varep^{2}\|\hat\mbp\|_{\mbV_4^2}\right) \|w\|_{L^2}  + \varep^3\|\mathbf q\|_{l^2}\right]\|v\|_{L^2}.
\eeq
 Writing as $v=\sum_{i\in \Sigma} v_i Z_{\mbp, *}^{I(i)i}$, we have the equality 
 \beqs
 \left<\Pi_0\mbL   w,  v\right>_{L^2} = \sum_{i \in \Sigma} v_i \left<\Pi_0 \mbL   Z_{\mbp, *}^{I(i)i}, w\right>_{L^2}.
 \eeqs 
 We use the definition of $\Pi_0$
\begin{equation}\label{est-bLpw-Z*}
\left<\Pi_0 \mbL w, Z^{I(j)j}_{\mathbf p, *}\right>_{L^2}
=\left<\mbL  w, Z_{\mathbf p, *}^{I(j)j}\right>_{L^2}-\frac{1}{|\Omega|}\int_\Omega Z_{\mathbf p, *}^{I(j)j}\dd   x\int_\Omega \mbL w\dd   x,
\end{equation}
and apply the estimate \eqref{est-Z*-mass} and identity \eqref{est-bLpw-mass}  from Lemma \ref{lem-coer} below to deduce
\begin{equation}\label{est-bLpw-Z*-1}
\left\|\Pi_{\mathcal Z_*}\Pi_0 \mbL w\right\|\lesssim \|\Pi_{\mathcal Z_*}\mbL   w\|_{L^2}+\varep^2(\|w\|_{L^2}+\varep \|\mathbf q\|_{l^2}).
\end{equation}
The corollary follows from Theorem \ref{thm-coupling est} by noting $w\in \mcZ_*^\bot$.
\end{proof}

\subsection{Coercivity}

The coercivity estimates on the operator $\mbL$ restricted to the orthogonal complement of the modified slow space, $\mcZ_*^\bot$, are essential to the orbital stability of the underlying manifold. Coercivity estimates for the constrained bilinear form $\mbL \big|_{\mathcal Z}$ for the preliminary slow space were derived in  \cite{DHPW-14, NP-17} and Theorem $2.5$ of \cite{HP-15}, for the weak functionalization under the restriction $\rho\sim \sqrt{\varep}$. However, these results lead to an $\varep$ dependent coercivity estimate. Our main coercivity result, requires only $\rho=o(1)$, independent of $\varep$, and exploits the improved orthogonality of the modified slow spaces. In this subsection we establish this enhanced coercivity of the linearized operator $\mbL  $ on the space orthogonal to the modified approximate slow space $\mathcal Z_*$.

\begin{thm}\label{lem-coer}
Suppose $\rho>0$ is suitably small. Then there exists $\varep_0>0$, dependent upon $\rho$, and a coercivity constant $C$ independent of $\rho$, such that for all $\varep\in(0,\varep_0)$ and all  $w\in \mathcal Z_*^\bot$, 
\begin{equation}\label{coer-bLp}
\left< \mbL  w,w\right>_{L^2}\geq C\rho^{2} \|w\|_{\Htwoin}^2\qquad\hbox{and}\quad \|\mbL   w\|_{L^2}^2\geq  C \rho^{2}\left< \mbL  w,w\right>_{L^2}.
\end{equation}
Moreover, if for any $\mbq$ the associated $Q\in \mcZ_{*}^0$ satisfies $\langle w+Q\rangle_{L^2}=0$, then we have the average estimate
\begin{equation}\label{est-bLpw-mass}
\left|\left<\mbL   w\right>_{L^2}\right| 
\lesssim \varep^{1/2}\|w\|_{L^2}+\varep^{3/2}\|\mathbf q\|_{l^2},
\end{equation}
and in addition
\begin{equation}\label{coer-PibL-p}
\begin{aligned}
C \varep^3 \|\mathbf q\|_{l^2}^2+\left<\Pi_0\mbL   w, \mbL   w\right>_{L^2}\geq   \|\mbL   w\|_{L^2}^2. 
 \end{aligned}
\end{equation}
\end{thm}
\begin{proof}   
To establish \eqref{coer-bLp}, we introduce
\beqs
\mathrm L_{1} := -\varep^2\Delta +W''(\Phi_{\mathbf p} )-\frac{1}{2}\varep\eta_1,
\eeqs
and rewrite the linearized operator $\mbL  $ defined by \eqref{def-bLp} in the form $\mbL  =\left(\mrL_{1}  \right)^2+\varep \mathrm R$
where
\begin{equation*}
\quad \mathrm R=-\frac{\varep \eta_1^2}{4} -\frac{W'''(\Phi_{\mathbf p} )}{\varep} \Big(\varep^2\Delta \Phi_{\mathbf p} -W'(\Phi_{\mathbf p} )\Big)+ \eta_d W''(\Phi_{\mathbf p} ).
\end{equation*}
Since $\mathrm R$ is a multiplier operator with a finite $L^\infty$-norm, it follows that
\begin{equation*}
 \left<\mbL   w,w\right>_{L^2} \geq \left<\left(\mrL_{1}  \right)^2 w,w\right>_{L^2}-\varep \|\mathrm R\|_{L^\infty}\|w\|_{L^2}^2,
\end{equation*}
and moreover for some $C>0$ independent of $\varep$,
\begin{equation*}
 \|\mbL  w\|_{L^2}^2\geq  \left\|\left(\mrL_{1}  \right)^2 w\right\|_{L^2}^2-\varep^2 C\|\mathrm R\|_{L^\infty}^2 \|w\|_{L^2}^2.
 \end{equation*}
Imposing the condition $\varep_0\ll \rho^2=o(1)$, then the coercivity estimates   \eqref{coer-bLp} for $\mbL  $  follow from Theorem 2.5 of \cite{HP-15} by replacing the preliminary 
approximate slow space $\mathcal Z$ with the modified approximation $\mathcal Z_*$. It remains to obtain estimates  \eqref{est-bLpw-mass} and \eqref{coer-PibL-p}. From the definition of  $\Pi_0$,
\begin{equation}\label{coer-bLpw-mass-1}
\left<\Pi_0\mbL   w, \mbL   w\right>_{L^2}=\|\mbL   w\|_{L^2}^2-\frac{1}{|\Omega|}\left(\int_\Omega \mbL  w\dd   x\right)^2.
\end{equation}
To estimate the averaged term we turn to the definition, \eqref{def-bLp}, of $\mbL  $ which implies
\begin{equation}\label{bL-pw-mass-1}
\begin{aligned}
\int_\Omega \mbL  w\dd   x&=\int_\Omega \Big[\left(\varep^2\Delta-W''(\Phi_{\mathbf p}  ) +\varep \eta_1\right)\left( \varep^2\Delta -W''(\Phi_{\mathbf p} ) \right)w \\
&\qquad-\left(\varep^2 \Delta \Phi_{\mathbf p}  -W'(\Phi_{\mathbf p} )\right)W'''(\Phi_{\mathbf p} ) w 
  +\varep \eta_d W''(\Phi_{\mathbf p} ) w\Big]\dd   x.
\end{aligned}
\end{equation}
Since $w$ satisfies periodic boundary conditions, both $\Delta w$ and $\Delta^2 w$ has no mass
which allows us to rewrite \eqref{bL-pw-mass-1} as
\beqs
\int_\Omega \mbL  w\dd   x=\mathcal I_1+\mathcal I_2+\mathcal I_3,
\eeqs
where the terms $\mathcal I_k (k=1,2,3)$ are defined by
\beqs
\begin{aligned}
\mathcal I_1:=-2\varep^2\int_\Omega  W''(\Phi_{\mathbf p} )\Delta w\dd   x,\qquad \qquad 
\mathcal I_2:=\int_\Omega \left(W''(\Phi_{\mathbf p} )\right)^2w\dd   x,\\
 \qquad \mathcal I_3:=-\int_\Omega \left[ \left(\varep^2\Delta \Phi_{\mathbf p}  -W'(\Phi_{\mathbf p} )\right)W'''(\Phi_{\mathbf p} )- \varep( \eta_d -\eta_1)W''(\Phi_{\mathbf p} ) \right]w\dd   x.\quad
\end{aligned}
\eeqs
We address these terms one by one. For the first term we integrate by parts and add a zero term
\begin{equation*}
\mathcal I_1=-2\varep^2\int_\Omega  \left( \Delta W''(\Phi_{\mathbf p} )  \right)w\dd   x=-2 \varep^2\int_\Omega  \Delta\left(W''(\Phi_{\mathbf p} ) -W''(\phi_0^\infty) \right)w\dd   x .
\end{equation*}
Since $\varep^2\Delta(W''(\Phi_{\mbp} )-W''(\phi_0^\infty))$ is bounded in $L^\infty$ and exponentially localized near the interface $\Gamma_{\mbp}$  we obtain
\begin{equation}\label{bL-pw-est-I-1}
|\mathcal I_1|\lesssim \varep^{1/2}\|w\|_{L^2}.
\end{equation}
By the definition of $\Phi_{\mathbf p}$, the quantity $\varep^2 \Delta \Phi_{\mathbf p}- W'(\Phi_{\mathbf p})$ is order of $\varep$ in $L^\infty$, we deduce that
the part of the integrand in the brackets in $\mathcal I_3$ is of order of $\varep$ in $L^\infty$, hence
 \begin{equation}\label{bL-pwest-I-3}
| \mathcal I_3|\lesssim \varep\|w\|_{L^2}.
 \end{equation}
 Finally, to bound  $\mathcal I_2$ we decompose it into near and far-field parts
\beqs
\begin{aligned}
\mathcal I_2&=\int_\Omega \Big[\left(W''(\Phi_{\mathbf p} )\right)^2- \left(W''(\phi_0^\infty)\right)^2\Big]w \dd   x+\left(W''(\phi_0^\infty)\right)^2 \int_\Omega w\dd   x.
\end{aligned}
\eeqs
The mass of $w$ balances with the mass of $Q$, that is, $\left<w\right>_{L^2}=-\left<Q\right>_{L^2}$. From  \eqref{est-Q-mass} we deduce that
\begin{equation}\label{bL-pw-est-I-2}
|\mathcal I_2|\lesssim \varep^{1/2}\|w\|_{L^2} +\varep^{3/2}\|\mathbf q\|_{l^2}.
\end{equation}
 Combining  estimates for $\mathcal I_k(k=1, 2,3)$ in \eqref{bL-pw-est-I-1}-\eqref{bL-pw-est-I-2} yields \eqref{est-bLpw-mass}. We deduce
 \eqref{coer-PibL-p} from these results together with \eqref{coer-bLpw-mass-1}.
\end{proof}


\section{Orbital Stability of the Bilayer Manifold}
\label{s:nonstab-BLM}
The tangent plane of the bilayer manifold $\cMb$ lies approximately in the meander space $\mcZ_*^1$. In this section we construct a nonlinear projection that maps a tubular projection neighborhood of the bilayer manifold onto the bilayer manifold. The projection uniquely decomposes each $u$ in the projection neighborhood into a bilayer \muckmuck parameterized by the meander modes $\mbp$ plus an orthogonal perturbation $v^\bot\in(\mcZ_*^1)^\bot$. The FCH gradient flow \eqref{RCL-flow} weakly excites the pearling modes, which from the coercivity estimates of Lemma\,\ref{lem-eigen-bM} are weakly damped when the pearling stability condition \eqref{cond-P-stab} holds. Accommodating the weak damping necessitates extracting the pearling modes from the remainder and tracking their evolution dynamically. This is accomplished by further decomposing the orthogonal perturbation  $v^\bot$ in its components in the $Q=Q(\mbq)$ in the pearling slow space $\mcZ_*^0$ and the fast modes $w\in\mcZ_*^\bot.$

We rewrite the flow as an evolution in these variables, and show that for initial data sufficiently close to the the bilayer manifold whose projected meander parameters lie within a set $\cO_\delta \subset\cD_\delta$, then the solution $u=u(t)$ remains close to $\cM_b$ so long as $\mbp$ remains inside of  a slightly bigger set $\cO_{2,\delta}\subset \cD_\delta$. In a companion paper, \cite{CP-nonlinear}, we consider a circular base point interface associated to an equilibrium of the flow and construct classes of initial data for which $\mbp$ remains inside of $\cO_{2,\delta}$ for all time and derive a curvature driven flow that captures the leading order evolution of the meander parameters.

\subsection{
Decomposition of the flow}

 We say that a base interface $\Gamma_0$ and a scaled system mass $M_0$ introduced in \eqref{def-M0} are an \textit{admissible base-point pair} if   $\Gamma_0\in \cG_{K_0, 2\ell_0}^4$ and  the system mass balances with the length of $\Gamma_0$ in the sense that 
 $$\left|M_0- m_0|\Gamma_0| \right|\lesssim  1,$$ 
 where $m_0$ is the mass per unit length of bilayer, defined in \eqref{def-sigma1*}. The collection of admissible pairs, the admissible set, is denoted $\mathcal A(K_0,\ell_0)$. This condition enforces that the far-field value of $\Phi_\mbp$ lies within $O(\varep)$ of $b_-$, and hence that the bulk parameter $|\sigma|\lesssim 1$, see \eqref{e:ConsMass}-\eqref{mass-Phi-p}.  

For each admissible pair $(\Gamma_0, M_0)$, we  introduce an $N_1$-dimensional  bilayer manifold $\cM_b=\cM_b(\Gamma_0, M_0;\rho)$ as given in Definition \ref{def-bM0}, where the $\rho$ dependence arises through $N_1=N_1(\rho)$, see Definition \ref{def-slow-space}. With the $H^2$ inner norm defined in \eqref{H2-inner}, we
construct a projection onto the bilayer manifold $\cM_b$ defined on the tubular projection neighborhood $\mcU$ of the bilayer  manifold $\cM_b$, 
\beq 
\mcU (\cM_b):=\left\{u \in H^2(\Omega)\, \Bigl |\, \inf_{\mbp\in\cD_\delta} \|u-\Phi_\mbp(\sigma)\|_{\Htwoin} \leq \delta \varep,\,
 \langle u-b_-\rangle_{L^2} = \frac{\varep M_0}{|\Omega|} \right\},
\eeq
where $\Phi_\mbp(\sigma)$ is defined in Lemma \ref{lem-def-Phi-p}  with $\sigma=\sigma(\mbp)$ given by \eqref{def-hatla}. 
\begin{defn}\label{def-cM-projection}
For $u\in \mcU(\cM_b)$, we say $\Pi_{\cM_b}u:=\Phi_\mbp(\sigma)$ is the projection onto $\cMb$ and $\Pi_{\cM_b}^\bot u:= v^\bot $ is its compliment if there exist unique  $\mbp\in \cD_\delta$ and mass-free orthogonal perturbation $v^\bot\in (\mcZ_{*}^1)^\bot$ such that
\beq\label{decomp-u-Mb}
u=\Phi_\mbp+v^\bot.
\eeq
{In this case we introduce $Q(\mbq):=\Pi_{\mcZ^0_*}v^\bot $, the projection of the orthogonal perturbation onto $\mcZ^0_*$ and $ w:=\Pi_{\mcZ^0_*}^\bot  v^\bot$, the projection onto the fast modes. 
 We call $(\mbp,\mbq)$ the projected parameters of $u$.}
\end{defn}
 The following lemma establishes the existence of a projection of $\mcU$ to $\cM_b$ and $\mcZ^0_*$.   

\begin{lemma}\label{lem-Manifold-Projection}
Let $\cM_b=\cM_b(\Gamma_0,M_0)$ be the bilayer manifold  as defined in Definition  \ref{def-bM0}. Then for $\delta,\varep_0>0$ sufficiently small the projection
$\Pi_{\cMb}$ is well posed on $\mcU$ for all $\varep\in(0,\varep_0)$. Moreover, for $u\in\mcU$ of the form $u=\Phi_{\mbp_0}+v$ with $\mbp_0\in\cD_\delta$ and massless perturbation $v\in H^2$ satisfying $\|v\|_{\Htwoin}\leq \delta \varep$, then $u$'s projected parameters $(\mbp, \mbq)$ and its orthogonal and fast perturbations, $v^\bot$ and $w$, satisfy
\begin{equation*}
\begin{aligned}
&  
\|\mbq\|_{l^2}+\eps^{-1/2}\|\mbp-\mbp_0\|_{l^2} 
\lesssim \|v\|_{L^2}; \qquad \|v^\bot\|_{\Htwoin}\lesssim \|w\|_{\Htwoin} +\|Q\|_{\Htwoin}\lesssim 
\|v\|_{\Htwoin}. 
\end{aligned}
\end{equation*}

\end{lemma}

The proof of this Lemma is postponed to the appendix.


  Let $u=u(t)$ be a solution of the flow \eqref{eq-FCH-L2} corresponding to initial data $u_0 \in \mcU(\cM_b)$.
So long as $u(t)\in \mcU(\cM_b)$ then $u$ admits the decomposition 
\begin{equation}\label{decomp-u}
u(  x,t) = \Phi_{\mathbf p} (  x;  \sigma) +v^\bot( x, t;\mbq),  \quad v^\bot\in (\mathcal Z_*^1)^\bot ,\quad \int_\Omega v^\bot\dd   x=0, 
\end{equation} 
where the projected parameters $(\mbp,\mbq)=(\mbp(t),\mbq(t))$ and the bulk density parameter $\sigma=\sigma(\mbp(t))$ defined by \eqref{def-hatla} are all time dependent. Substituting the ansatz \eqref{decomp-u} into 
the equation \eqref{eq-FCH-L2} leads to an equation for $\Phi_\mbp$ and $v^\bot$:
\begin{equation}\label{eq-v-Phi}
\p_t \Phi_{\mathbf p} +\p_t v^\bot=-\Pi_0 \mathrm F(\Phi_{\mathbf p} ) -\Pi_0 \mbL   v^\bot-\Pi_0  \mathrm N( v^\bot),
\end{equation}
where $\mbL$ is the linearization of $\mrF$ about $\Phi_\mbp$ introduced in \eqref{def-bLp}, and $\mathrm N(v^\bot)$ is the nonlinear term defined by
\begin{equation}\label{def-N}
\mathrm N( v^\bot):= \mathrm F(\Phi_{\mathbf p} + v^\bot)-\mathrm F(\Phi_{\mathbf p} ) - \mbL  v^\bot.
\end{equation}
To exploit the strong coercivity of $\Pi_0\mbL_\mbp$ on $\mcZ_*^\bot$ and its $O(\varep)$-weak coercivity on $(\mcZ_{*}^1)^\bot$ we follow Definition\,\ref{def-cM-projection} and decompose the orthogonal  perturbation $ v^\bot$ into its pearling and fast mode sub-components
\begin{equation}\label{decomp-v}
 v^\bot=Q(  x, t)+w(  x, t), 
 \quad w\in \mathcal Z_*^\bot(\mbp, \rho).
\end{equation}

In the following, we make a priori assumptions that bound the rate of change of $\mbp$ induced by the flow and norm estimates on $\mbp$ that subsume those of in $\cD_\delta$, defined in \eqref{A-00}:
\begin{equation}\label{A-0}
|\mathrm p_0(t)|+\|\hat\mbp\|_{\mbV_1} \leq   \delta, \qquad \|\hat\mbp\|_{\mbV_2}+\varep \|\hat\mbp\|_{\mbV_4^2}\leq  1  ,   \qquad \|\dot{\mathbf p}\|_{l^2}\leq     \varep^3.
\end{equation}
Here $\delta$ are is in the definition of $\cD_\delta.$  In sub-section\,\ref{ss-MTheom} these assumptions will be refined to eliminate the condition on $\dot{\mbp}.$ 

 In the remainder of this section we develop bounds on $w$ and $\mathbf q$, which require an $L^2$-bound on the nonlinear term $\mathrm N(v^\bot)$. The projection of the solution $u$ onto the manifold involves the approximate tangent spaces $\mathcal Z_{*}(\mbp)$. Although the flow of $\mbp$ is slow in the sense of \eqref{A-0}, it induces temporal variation of the tangent plane that must be accounted for.  We emphasize that the linear operator $\mbL=\mbL_\mbp$ and the spaces $\mathcal Z_{*}^0(\mbp)$ and $\mcZ_{*}^1(\mbp)$ are independent of $\mbq$. The linearization and tangent spaces are defined along the bilayer  manifold $\cMb$.  More significant is the fact that the space $\mathcal Z_{*}(\mbp)$ is only approximately invariant under the action of the linearized operator. This produces terms whose control is crucial to the closure of the estimates. Indeed these terms motivate the introduction of the modified approximate slow space $\mathcal Z_{*}(\mbp)$. 

\subsection{Energy estimate for $w$}
We derive an $H^2$-bound on $w$ under the flow induced by \eqref{eq-FCH-L2} assuming the a priori estimates \eqref{A-0} on $\mbp$ and $\dot{\mbp}$.  
We decompose $ v^\bot$ as in \eqref{decomp-v} to rewrite \eqref{eq-v-Phi} as an evolution for the fast modes $w$, 
\begin{equation}\label{eq-w}
\p_t w+\Pi_0 \mbL   w=-\p_t \Phi_{\mathbf p} -\p_t Q-\Pi_0 \mathrm F(\Phi_{\mathbf p} ) -\Pi_0\mbL   Q-\Pi_0\mathrm N( v^\bot).
\end{equation}
\begin{lemma}\label{lem-est-w} 
Let $\varep\in(0,\varep_0)$ and  the a priori assumptions \eqref{A-0} hold, the function $w\in\mathcal Z_{*}^\bot$, obeys
\begin{equation}\label{est-w-2}
\begin{aligned}
\frac{\mrd }{\mrd t}\left<\mbL   w, w\right>_{L^2}+\|\mbL  w\|_{L^2}^2 & \lesssim \varep^{-1} \|\dot{\mathbf p}\|_{l^2}^2 +\varep^2\rho^{-4}(\|\mathbf q\|_{l^2}^2+ \|\dot{\mathbf q}\|_{l^2}^2)+\varep^5|\sigma-\sigma^*|^2\\
&\qquad +\varep^{7}(1+ \|\hat\mbp\|_{\mbV_4^2}^2)+\|\mathrm N(v^\bot)\|_{L^2}^2.
\end{aligned}
\end{equation}
provided that $\varep_0$ small enough depending on $\rho$. 
\end{lemma}
\begin{proof}
 Since the linearized operator  $\mbL  $ depends on time through $\mbp$, we have
\begin{equation}\label{est-H2-w-1}
\frac{\mrd }{\mrd t}\left<\mbL  w, w\right>_{L^2}=2\left<\p_t w, \mbL  w\right>_{L^2}+\left<\p_t(\mbL )w, w\right>_{L^2}.
\end{equation}
Considering the last term on the right-hand side, the definition \eqref{def-bLp} of $\mbL $ provides the expansion
\begin{equation*}
\begin{aligned}
\p_t (\mbL  )=&-\left(\varep^2 \Delta -W''+\varep \eta_1\right)\left(W'''\p_t \Phi_{\mathbf p} \right)-W^{'''}\p_t \Phi_{\mathbf p} \left(\varep^2 \Delta  -W''\right) \\
&-\left(\varep^2 \Delta \Phi_{\mathbf p}  -W'\right)W^{(4)}\p_t\Phi_{\mathbf p} -W'''\left(\varep^2\Delta -W''\right)\p_t\Phi_{\mathbf p}  +\varep\eta_d W'''\p_t \Phi_{\mathbf p} ,
\end{aligned}
\end{equation*}
where the potential well  $W$ is evaluated at $\Phi_\mbp $. Since $\Phi_{\mbp} $ is uniformly bounded in $L^\infty$ and in $L^2$ after action by powers of $\varep^2\Delta$, we identify the
upper bound on the bilinear form generated by $ \p_t (\mbL  )$
\begin{equation}\label{est-p-tLw,w-1} 
\begin{aligned}
\left<\p_t(\mbL  )w, w\right>_{L^2}&\lesssim\left( \left\|\dot {\mathbf p}\cdot \nabla_{\mathbf p}\Phi_{\mathbf p} \right\|_{L^\infty}+\|\dot{\mathbf p}\cdot\nabla_{\mathbf p} (\varep^2\Delta\Phi_{\mathbf p} )\|_{L^\infty}\right)\left(\|w\|_{L^2}^2+\|\varep^2 \Delta w\|_{L^2}^2\right).
\end{aligned}
\end{equation}
Utilizing the bounds of $\Phi_\mbp$ established in the Appendix Lemma \ref{lem-Phi_t-p}, assumption \eqref{A-0} and the coercivity estimate  \eqref{coer-bLp}, we obtain the upper bound on the bilinear term
\beqs
\left<\p_t(\mbL  )w, w\right>_{L^2}\lesssim \varep \rho^{-4}\|\mbL   w\|_{L^2}^2\leq \varep^{1/2}\|\mbL   w\|_{L^2}^2
\eeqs
by choosing  $\varep_0$ small enough depending on $\rho$. Returning to \eqref{est-H2-w-1}, substituting  \eqref{eq-w} for $\p_t w$, using the coercivity estimate \eqref{coer-PibL-p} 
and bounding the second term via the bilinear estimate above,  leads to 
\begin{equation}\label{est-H2-w-2}
\begin{aligned}
\frac{\mrd }{\mrd t}\left<\mbL   w,w\right>_{L^2}+ \|\mbL  w\|_{L^2}^2\leq &\,-2\left<\p_t\Phi_{\mathbf p} +\p_t Q+\Pi_0 \mbL  Q, \mbL  w\right>_{L^2}\\
&-2\left<\Pi_0 \mathrm F(\Phi_{\mathbf p} )+\Pi_0\mathrm N(v^\bot), \mbL  w\right>_{L^2}.
\end{aligned}
\end{equation}
Here we also used $\varep\in(0,\varep_0)$ with $\varep_0$ small enough depending on domain, system parameters and $(\Gamma_0, M_0)$. Considering the terms on the right-hand side of \eqref{est-H2-w-2}, we apply H\"older's inequality to the last term
\begin{equation}\label{est-w-RHS-1}
\left| \left<\Pi_0\mathrm N(v^\bot), \mbL   w\right>_{L^2}\right|\lesssim \|\mathrm N(v^\bot)\|_{L^2}\|\mbL   w\|_{L^2}.
\end{equation}
Utilizing H\"older's inequality and $L^2$-bound of $\p_t\Phi_\mbp$, we establish the bound
\begin{equation}\label{est-w-RHS-2}
\left|\left<\p_t\Phi_{\mathbf p}  , \mbL  w\right>_{L^2}\right|\lesssim \varep^{-1/2}\|\dot{\mbp}\|_{l^2}\|\mbL   w\|_{L^2}.
\end{equation}
For the $\Pi_0\mbL   Q$ term we project onto {$\mathcal Z_{*}$} and its complement,  use {$Q\in \mathcal Z_{*}$}, Theorem \ref{thm-coupling est}, and finally  the coercivity of Lemma \ref{lem-coer}
to establish
   \begin{equation}\label{est-w-RHS-3-1}
\begin{aligned}
\left|\left<\Pi_0 \mbL   Q, \mbL   w\right>_{L^2}\right|&=\left|\left<\Pi_{\mathcal Z_*^\bot} \Pi_0\mbL  Q, \mbL  w\right>_{L^2}+\left<\Pi_{\mathcal Z_*^\bot}\mbL  \Pi_{\mathcal Z_*} \Pi_0\mbL  Q, w\right>_{L^2}\right|\\
&\lesssim (\varep^2+\varep^2\|\hat{\mbp}\|_{\mbV_4^2})\|\mathbf q\|_{l^2} (\|w\|_{L^2} +\|\mbL  w\|_{L^2})  \\
&\lesssim \varep \rho^{-2} \|\mathbf q\|_{l^2} \|\mbL  w\|_{L^2}.
\end{aligned}
\end{equation}
For the second term on right-hand side of \eqref{est-H2-w-2} requires an investigation of $\p_t Q$. By the definition of $Q$ we calculate
\begin{equation}\label{def-p-tQ}
\p_t Q=\sum_{j\in \Sigma_0} \dot{\mathrm q}_jZ_{\mathbf p, *}^{0j}+\sum_{j\in \Sigma_0} \mathrm q_j \p_t Z_{\mathbf p, *}^{0j}.
\end{equation}
Note that the second term can  be written as
\begin{equation}\label{est-w-RHS-3-2temp}
\begin{aligned}
&\quad \left<\p_t Q, \mbL   w\right>_{L^2}=\left<\Pi_{\mcZ_*^\bot} \mbL \p_t Q, w \right>_{L^2}.
\end{aligned}
\end{equation}
By employing  relation \eqref{def-p-tQ},  statement (2) of Theorem \ref{thm-coupling est} and estimate \eqref{est-L2-p-tZ*} we may bound
\begin{equation}\label{est-w-RHS-3-temp}
\|\Pi_{\mathcal Z_*^\bot}\mbL   \p_t Q\|_{L^2}\lesssim (\varep^2+\varep^2 \|\hat{\mbp}\|_{\mbV_4^2}) \|\dot{\mathbf q}\|_{l^2}+\varep^{2} \|\mathbf q\|_{l^1}.
\end{equation}
Here by the $l^2$-$l^1$ estimate and scaling of $N_0$ from \eqref{est-N0}, we have
\beq\label{est-l-12-q}
\|\mbq\|_{l^1}\leq \varep^{-1/2}\|\mbq\|_{l^2}. 
\eeq
Using Holder's inequality,  the a priori assumptions \eqref{A-0} and the coercivity to bound $\|w\|_{L^2}$ by $\|\mbL  w\|_{L^2}$ we deduce from \eqref{est-w-RHS-3-2temp} - \eqref{est-l-12-q}
\begin{equation}\label{est-w-RHS-3-2}
\begin{aligned}
&\quad \left| \left<\p_t Q, \mbL   w\right>_{L^2}\right|  \lesssim \varep \rho^{-2}\|\mbL  w\|_{L^2}(\|\dot{\mathbf q}\|_{l^2}+\|\mathbf q\|_{l^2}).
\end{aligned}
\end{equation}
Combining the estimates \eqref{est-w-RHS-1} and \eqref{est-w-RHS-2}-\eqref{est-w-RHS-3-2}  with \eqref{est-H2-w-2} and using Young's inequality, yields the
 estimate
 \beq\label{est-H2-w-3}
 \begin{aligned}
 \frac{\dd}{\dd t} \left<\mbL w, w\right>_{L^2} +\|\mbL w\|_{L^2}^2 \leq  &C\left( \varep^{-1/2}\|\dot{\mbp}\|_{L^2}+\varep \rho^{-2}(\|\mbq\|_{l^2} +\|\dot\mbq\|_{l^2})+\|\mrN(v^\bot)\|_{L^2}\right)\|\mbL w\|_{L^2} \\
 & -2\left<\Pi_0 \mrF(\Phi_\mbp), \mbL w\right>_{L^2}.
 \end{aligned}
 \eeq

It remains to bound the $\mrF(\Phi_\mbp)$ term on the right-hand side of the above inequality.  Using Lemma \ref{lem-est-residual} to bound the $L^2$-norm of $\Pi_0\mrF(\Phi_\mbp)$  terms yields
\begin{equation}\label{est-F-bLw}
\begin{aligned}
\left| \left<\Pi_0\mathrm F(\Phi_{\mathbf p} ), \mbL   w\right>_{L^2}\right| \lesssim &\varep^{5/2}|\sigma-\sigma^*|\|\mbL   w\|_{L^2}+\varep^{7/2}(\|\hat{\mathbf p}\|_{\mbV_4^2}+1)\|\mbL   w\|_{L^2}.
\end{aligned}
\end{equation}
Combining the above estimate  \eqref{est-F-bLw} with \eqref{est-H2-w-3} and using Young's inequality, yields the
 estimate \eqref{est-w-2}. The proof is complete. 
\end{proof}

\subsection{Estimates on the pearling parameters $\mathbf q(t)$} We derive  $l^2$ estimates of $\mathbf q$ and $\dot{\mathbf q}$ subject to the a priori assumptions \eqref{A-0}.
  We rewrite \eqref{eq-w} as an evolution for $Q$,
\begin{equation}\label{eq-Q}
\p_t Q+\Pi_0\mbL   Q=\mathscr R[\mbp, w, \mrN],
\end{equation}
where $\mathscr R[\mbp, w, \mrN]$ is the pearling remainder contributed by $\mbp, w,$ and the nonlinear terms $\mathrm N(v^\bot)$,  specifically
\beq
\label{e:R-def}
\mathscr R[\mbp, w, \mrN]:=-\p_t \Phi_{\mathbf p}  -\p_t w-\Pi_0 \mathrm F(\Phi_{\mathbf p} )-\Pi_0 \mbL   w-\Pi_0 \mathrm N( v^\bot).
\eeq
We derive the evolution of $\dot{\mbq}$ by projecting this system onto the slowly evolving space {$\mcZ_*^1(\mbp(t)).$}
\begin{lemma}\label{lem-est-q} Assuming the a priori estimates \eqref{A-0} and the pearling stability condition \eqref{cond-P-stab} hold, $\varep\in(0,\varep_0)$ with $\varep_0$ suitably small,  then there exists $C>0$ independent of $\varep, \rho$ such that the pearling parameters $\mathbf q=(\mathrm q_k(t))_{k\in \Sigma_0}$ obey
\begin{equation*}
\begin{aligned}
&\|\dot{\mathbf q}\|_{l^2}^2\lesssim \|\mbq\|_{l^2}^2+\varep^2\|w\|_{L^2}^2 +\varep \|\dot\mbp\|_{l^2}^2+\|\mathrm N(v^\bot)\|_{L^2}^2+\varep^9+\varep^{9}\|\hat{\mathbf p}\|_{\mbV_4^2}^2;\\
&\p_t\|\mathbf q\|_{l^2}^2+C\varep\|\mathbf q\|_{l^2}^2 \lesssim \varep\|w\|_{L^2}^2+\|\dot\mbp\|_{l^2}^2+\varep^{-1}\|\mathrm N(v^\bot)\|_{L^2}^2+\varep^8+\varep^8\|\hat{\mathbf p}\|_{\mbV_4^2}^2.
\end{aligned}
\end{equation*}
\end{lemma}
\begin{proof}
Taking the $L^2$-inner product of equation \eqref{eq-Q} with $Q$ yields
\begin{equation}\label{est-q-1}
\begin{aligned}
\left<\p_t Q, Q\right>_{L^2}+\left<\Pi_0 \mbL  Q, Q\right>_{L^2}=&\left<\mathscr R[\mbp, w], Q\right>_{L^2}.
\end{aligned}
\end{equation}
Using  \eqref{def-p-tQ} we rewrite the first term on the left-hand side as
\begin{equation}\label{est-p-tQ-Q-0}
\left<\p_t Q, Q\right>_{L^2}=\sum_{i,j\in \Sigma_0} \dot{\mathrm q}_i\mathrm q_j\left<Z_{\mathbf p, *}^{0i}, Z_{\mathbf p, *}^{0j}\right>_{L^2}+\sum_{i,j\in \Sigma_0} \mathrm q_i\mathrm q_j\left<\p_t Z_{\mathbf p, *}^{0i}, Z_{\mathbf p, *}^{0j}\right>_{L^2}.
\end{equation}  
The partial orthogonality of the basis $\{Z_{\mathbf p, *}^{0j}\}_{j\in \Sigma_0}$, from \eqref{ortho-Z*}, and the a priori estimate   $\|\hat{\mbp}\|_{\mbV_2^2}\lesssim 1$ yield
\begin{equation}\label{est-p-tQ-Q-1}
\begin{aligned}
\sum_{i,j\in \Sigma_0} \dot{\mathrm q}_i\mathrm q_j\left<Z_{\mathbf p, *}^{0i}, Z_{\mathbf p, *}^{0j}\right>_{L^2}\geq &(1+\mrp_0) \frac{\p_t \|\mathbf q\|_{l^2}^2}{2}-C\varep^2\|\mathbf q\|_{l^2}\|\dot{\mathbf q}\|_{l^2}.
\end{aligned}
\end{equation}
Applying H\"older's inequality to the second term on the right-hand side of \eqref{est-p-tQ-Q-0}, the estimate \eqref{est-L2-p-tZ*} on $\p_t Z_{\mathbf p, *}^{0j}$ and the $l^1$-$l^2$ estimate \eqref{est-l-12-q} yields the bound
\begin{equation}\label{est-p-tQ-Q-2}
\sum_{i,j\in \Sigma_0} \mathrm q_i\mathrm q_j\left<\p_t Z_{\mathbf p, *}^{0i}, Z_{\mathbf p, *}^{0j}\right>_{L^2}\lesssim \varep^{2}\|\mbq\|_{l^2}\|\mbq\|_{l^1}\lesssim \varep^{3/2}\|\mbq\|_{l^2}^2.
\end{equation}
 Combining estimates \eqref{est-p-tQ-Q-1}-\eqref{est-p-tQ-Q-2} with \eqref{est-p-tQ-Q-0} and applying Young's inequality yield
\begin{equation*}
\frac{ 1+\mrp_0}{2}\p_t \|\mathbf q\|_{l^2}^2-C\varep^{3/2}\|\mathbf q\|_{l^2}^2-C\varep^{3/2}\|\dot{\mathbf q}\|_{l^2}^2\leq \left<\p_t Q, Q\right>_{L^2},
\end{equation*}
which when substituted into \eqref{est-q-1} after multiplying by $2$ implies
\begin{equation}\label{est-l2-q-1}
\begin{aligned}
(1+\mrp_0)\p_t \|\mathbf q\|_{l^2}^2+2\left<\Pi_0 \mbL  Q, Q\right>_{L^2}\leq C\varep^{3/2}\|\mathbf q\|_{l^2}^2+C\varep^{3/2}\|\dot{\mathbf q}\|_{l^2}^2+2\left<\msR[\mbp, w, \mrN], Q\right>_{L^2}.
\end{aligned}
\end{equation}
The last term on the right hand side can be bounded by H\"older's inequality 
\beqs
\left<\msR[\mbp, w, \mrN], Q\right>_{L^2}\leq \|\Pi_{\mcZ_*^0} \msR[\mbp, w, \mrN]\|_{L^2}\|\mbq\|_{l^2}.
\eeqs 
With  the projection of  the remainder $\msR[\mbp,  w, \mrN]$ bounded in next Lemma \ref{lem-est-R-q}, we derive
\begin{equation}\label{est-l2-q-2}
\begin{aligned}
(1+\mrp_0)\p_t \|\mathbf q\|_{l^2}^2+2\left<\Pi_0 \mbL   Q, Q\right>_{L^2}\lesssim& \varep^{3/2} \|\dot{\mathbf q}\|_{l^2}^2+\varep^{3/2} \|\mathbf q\|_{l^2}^2+\varep \|w\|_{L^2}\|\mbq\|_{l^2} \quad  \\
\qquad  &+\Big(\varep^{1/2}\|\dot\mbp\|_{l^2}+ \|\mathrm N(v^\bot)\|_{L^2}+\varep^{9/2} +\varep^{9/2}\|\hat{\mathbf p}\|_{\mbV_4^2}\Big)\|\mathbf q\|_{l^2}.
\end{aligned}
\end{equation}
Since the system is pearling stable, Lemma \ref{lem-eigen-bM} implies the existence of $C>0$ independent of $\varep$ for which,
\begin{equation*}
\left<\Pi_0 \mbL  Q, Q\right>_{L^2} = \mathbb \mbq^T \mbM^*(0,0) \mbq \geq C\varep\|\mathbf q\|_{l^2}^2.
\end{equation*} 
 The a priori estimates imply that $|\mrp_0| $ is  small and hence $(1+\mrp_0)$ is bounded away from zero. Dividing both sides of \eqref{est-l2-q-2} by $(1+\mrp_0)>0$ and applying Young's inequality to the right-hand side of the resulting inequality yields
\begin{equation}\label{est-l2-q-3}
\p_t \|\mathbf q\|_{l^2}^2+C\varep\|\mathbf q\|_{l^2}^2\lesssim \varep^{3/2}\|\dot{\mathbf q}\|_{l^2}^2+\varep \|w\|_{L^2}^2+\|\dot \mbp\|_{l^2}^2+\varep^{-1}\|\mathrm N(v^\bot)\|_{L^2}^2+\varep^8+\varep^8\|\hat{\mathbf p}\|_{\mbV_4^2}^2
\end{equation}
for $\varep\in(0,\varep_0)$ provided that $\varep_0$ small enough depending on domain, system parameters and $K_0,\ell_0$.  

It remains to bound $\|\dot{\mathbf q}\|_{l^2}$.  Taking the $L^2$ inner product of equation \eqref{eq-Q} with $\sum_{j\in \Sigma_0}\dot{\mathrm q}_j Z_{\mathbf p, *}^{0j}$ implies
\beq\label{est-qdot-1}
\left<\p_t Q, \sum_{j\in \Sigma_0} \dot\mrq_jZ_{\mbp,*}^{0j} \right>_{L^2} =  -\left<\Pi_0\mbL   Q, \sum_{j\in \Sigma_0}\dot{ \mathrm q}_jZ_{\mathbf p,*}^{0j}\right>_{L^2} +\left<\msR[\mbp, w, \mrN], \sum_{j\in \Sigma_0} \dot\mrq_jZ_{\mbp,*}^{0j}\right>_{L^2}. 
\eeq
The term on the left can be dealt with similarly as we deal with $\left<\p_t Q, Q\right>_{L^2}$ in \eqref{est-p-tQ-Q-0}-\eqref{est-p-tQ-Q-2}, which gives us
\beqs
(1+\mrp_0)\|\dot\mbq\|_{l^2}^2 \leq \left<\p_t Q, \sum_{j\in \Sigma_0} \dot\mrq_jZ_{\mbp,*}^{1j} \right>_{L^2} +C\varep^{3/2}\|\dot \mbq\|_{l^2}\|\mbq\|_{l^2}
\eeqs
for some numerical constant $C$ independent of $\varep$ and $\rho$. 
In light of Lemma \ref{lem-est-R-q}, the last term of \eqref{est-qdot-1} which includes the remainder projection, can be bounded by 
\beqs
\begin{aligned}
\left<\msR[\mbp, w, \mrN], \sum_{j\in \Sigma_0} \dot\mrq_jZ_{\mbp,*}^{1j}\right>_{L^2} &\leq \|\Pi_{\mcZ_*^0} \msR[\mbp, w, \mathrm N]\|_{L^2} \|\dot \mbq\|_{l^2}.
\end{aligned}
\eeqs 
Adding the   two estimates above with \eqref{est-qdot-1},  and applying Lemma \ref{lem-est-R-q}  we obtain
\begin{equation}\label{est-q'-1}
\begin{aligned}
(1+\mrp_0)\|\dot{\mathbf q}\|_{l^2}^2\leq&-\left<\Pi_0\mbL   Q, \sum_{j\in \Sigma_0}\dot{ \mathrm q}_jZ_{\mathbf p,*}^{0j}\right>_{L^2}+C\Big(  \varep^{3/2}\|\mathbf q\|_{l^2}+\varep \|w\|_{L^2}+\|\mathrm N(v^\bot)\|_{L^2}+\varep^{9/2}\\
&\h{140pt}+\varep^{1/2}\|\dot\mbp\|_{l^2}+\varep^{9/2}\|\hat{\mathbf p}\|_{\mbV_4^2}\Big)\|\dot\mbq\|_{l^2}.
\end{aligned}
\end{equation}
From the definition \eqref{def-Q}  of $Q$  and the estimate \eqref{est-bM*ij} on  $\mathbb M^*$ we rewrite the term involving $\Pi_0\mbL   Q$  as
\begin{equation}\label{est-LQ-p-tQ-0}
\left| \left<\Pi_0\mbL   Q, \sum_{j\in \Sigma_0}\dot{\mathrm q}_jZ_{\mathbf p,*}^{0j}\right>_{L^2}\right|=\left| \sum_{i,j\in \Sigma_0}\mathrm q_i\dot{\mathrm q}_j\mathbb M^*_{ij}(0,0)\right|\lesssim \|\mbq\|_{l^2}\|\dot \mbq\|_{l^2}.
\end{equation}
We establish the $l^2$-estimate of $\dot{\mathbf q}$ by returning this bound to \eqref{est-q'-1},  applying Young's inequality and  assumption \eqref{A-0} on $\mrp_0$. The estimate on $\p_t\|\mathbf q\|_{l^2}^2$ follows from \eqref{est-l2-q-3}. 
\end{proof}

The proof of Lemma \ref{lem-est-q} requires the following estimate on the projection of the remainder to {$\mcZ_{*}^0$}. 

\begin{lemma}\label{lem-est-R-q} Under the assumptions of \eqref{A-0}, the projection of the remainder, defined in \eqref{e:R-def}, to the pearling slow space can be bounded by 
\beqs
\|\Pi_{\mcZ_*^0} \mathscr R[\mbp, w, \mrN]\|_{L^2} \lesssim  \varep \|w\|_{L^2} +\varep^3\|\mbq\|_{l^2}+\varep^{1/2} \|\dot\mbp\| +\varep^{9/2} +\varep^{9/2}\|\hat\mbp\|_{\mbV_4^2} +\|\mathrm N(v^\bot)\|_{L^2}.
\eeqs
\end{lemma}
\begin{proof}
Since  $\|Q\|_{L^2}\sim \|\mbq\|_{l^2}$, see Theorem \ref{thm-coupling est}, it suffices to establish the following inequality for any $Q=\sum_{j\in \Sigma_0} \mrq_j Z_{\mbp,*}^{0j}\in\mcZ_*^0$,
\beq\label{est-R-q-1}
\left<\msR[\mbp, w, \mrN], Q\right>_{L^2}\lesssim \left( \varep \|w\|_{L^2} +\varep^3\|\mbq\|_{l^2}+\varep^{1/2} \|\dot\mbp\|  +\varep^{9/2}\|\hat\mbp\|_{\mbV_4^2} +\|\mathrm N(v^\bot)\|_{L^2}\right) \|\mbq\|_{l^2}.
\eeq
By the definition of $\msR[\mbp, w, \mrN]$, we expand
\beqs
\left<\msR[\mbp, w, \mrN], Q\right>_{L^2}=- \left<\p_t \Phi_\mbp, Q\right>_{L^2}- \left<\Pi_0 \mrF(\Phi_\mbp), Q\right>_{L^2} - \left<\p_t w, Q\right>_{L^2} - \left<\Pi_0 \mbL w, Q\right>_{L^2} - \left<\Pi_0 \mrN(v^\bot), Q\right>_{L^2}.
\eeqs
We deal with these terms one by one in the following. First, in order to bound the term involving $\p_t\Phi_\mbp$ we use its $L^2$ projection estimate to $\mcZ_*^0$. We use Lemma \ref{lem-Phi_t-p} from the appendix.  Since $Q\in \mcZ_*^0$, we have   
\begin{equation}\label{est-p-tPhi-p-Q}
\begin{aligned}
\left|\left<\p_t \Phi_{\mbp} , Q\right>_{L^2}\right| & \lesssim  \|\Pi_{\mcZ_*^0} \p_t\Phi_\mbp\|_{L^2}\|Q\|_{L^2}\lesssim \varep^{1/2}  \|\dot{\mbp}\|_{l^2}\|\mathbf q\|_{l^2}.
\end{aligned}
\end{equation}
 Second,  since $w\in \mathcal Z_*^\bot$ we may apply the expansion  \eqref{def-p-tQ}  of $\p_t Q$ to deduce
\beqs
\left<\p_t w, Q\right>_{L^2} =-\left<w,\p_tQ\right>_{L^2}=\sum_{j\in \Sigma_0} \left<w, \mrq_j\p_t Z_{\mbp,*}^{0j}\right>_{L^2}.
\eeqs
From the estimate \eqref{est-L2-p-tZ*}, H\"older's inequality, and the $l^2$-$l^1$ estimate of $\mbq$  \eqref{est-l-12-q} we obtain
\beq\label{est-p-tw-Q}
\left|\left<\p_t w, Q\right>_{L^2}\right| \lesssim \varep^{-1}\|\dot\mbp\|_{l^2}\|w\|_{L^2}\|\mbq\|_{l^1}\lesssim \varep^{3/2}\|w\|_{L^2}\|\mbq\|_{l^2}.
\eeq
For the inner product with the nonlinear term,   so H\"older's inequality yields
\begin{equation}\label{est-p-tPhi-N-Q-1}
\left| \left<\Pi_0 \mathrm N(v^\bot), Q\right>_{L^2}\right| \lesssim \|\mathrm N(v^\bot)\|_{L^2} \|\mathbf q\|_{l^2}.
\end{equation}
For the $\mbL  w$ term, Corollary \ref{cor-est-w-Z*} and a priori assumptions \eqref{A-0} yield
\begin{equation}\label{est-bLw-Q}
\begin{aligned}
\left| \left<\Pi_0 \mbL  w, Q\right>_{L^2}\right| & \lesssim \left(\varep^2+\varep^2 \|\hat{\mathbf p}\|_{\mbV_4^2}\right) \|w\|_{L^2}\|\mathbf q\|_{l^2} +\varep^3\|\mathbf q\|_{l^2}^2\\
&\lesssim \varep  \|w\|_{L^2}\|\mbq\|_{l^2}+\varep^3\|\mbq \|_{l^2}^2.
\end{aligned}
\end{equation}
The estimate of the term involving the residual, $\mrF(\Phi_\mbp)$, is deferred to \eqref{proj-PiF-Q} of Lemma \ref{Lem-Q-resid}.  As a consequence of \eqref{est-p-tw-Q}--\eqref{est-bLw-Q} and  \eqref{proj-PiF-Q}, the estimate \eqref{est-R-q-1} follows, and hence the proof is complete.
\end{proof}
To complete the estimation of the projection of $Q$ and $\mrN$ to the pearling space  we require the following simple lemma which exploits the high in-plane wave number of the pearling modes. This affords better bounds on the coupling of the residual to the pearling modes that compensates for the weaker coercivity they experience. 
\begin{lemma}[High pearling wavenumber]\label{lem-h-2}
Let $h=h(\bm \gamma_\mbp'')$ in the sense of Notation \ref{Notation-h}, then   there exists a unit vector $(e_j) $ such that
\beqs
\int h(\bm \gamma_\mbp'') \tilde \Theta_j \dd \tilde s_\mbp = O(\varep, \varep \|\hat\mbp\|_{\mbV_4^2}) e_j, \qquad j\in \Sigma_0.
\eeqs
\end{lemma}
\begin{proof}
The proof closely follows that of Lemma  \ref{lem-Theta} for the case $i=0, j\in \Sigma_0$, we omit the details.
\end{proof}

With Lemma\,\ref{lem-h-2} we obtain an improved bound on the coupling of the residual with the pearling modes.

\begin{lemma}\label{Lem-Q-resid}  Assuming \eqref{A-0}, the projection of the residual to pearling space satisfies the estimate
\begin{equation}\label{proj-PiF-Q}
\begin{aligned}
\left|\left<\Pi_0\mathrm F(\Phi_{\mathbf p} ), Q\right>_{L^2}\right|\lesssim &\varep^{9/2}(1+\|\hat{\mathbf p}\|_{\mbV_4^2} )\|\mathbf q\|_{l^2}.
\end{aligned}
\end{equation}
\end{lemma}
\begin{proof}
Subtracting off the far-field value $\mrF^\infty$ of the residual and using the definition of  $\Pi_0$,  we have
\begin{equation}\label{proj-PiF-Q-0}
\left<\Pi_0\mathrm F(\Phi_{\mathbf p}  ), Q\right>_{L^2}=\left<\mathrm F(\Phi_{\mbp} )-\mrF^\infty, Q\right>_{L^2}-\frac{1}{|\Omega|}\int_\Omega \left(\mathrm F(\Phi_{\mbp} )-\mrF^\infty\right)\dd x \int_\Omega Q\dd x.
\end{equation}
Using Lemma \ref{lem-mass-F} and the estimate \eqref{est-Q-mass}, the second term on the right-hand side satisfies 
\begin{equation}\label{proj-PiF-Q-1}
\frac{1}{|\Omega|}\int_\Omega \left(\mathrm F(\Phi_{\mbp} )-\mrF_m^\infty\right)\dd x \int_\Omega Q\dd x=O\left(\varep^{11/2}\|\mathbf q\|_{l^2}\right).
\end{equation}
We use the expansion of $\mathrm F(\Phi_{\mathbf p} )$ given in Lemma \ref{lem-def-Phi-p} to estimate the first term on the right-hand side of \eqref{proj-PiF-Q-0}. 
Examining the $L^2$-inner product of $\mrF_2$  and $Q$,
since $\mathrm F_2=(\sigma_1^*-\sigma)\kappa_{\mathbf p}f_2(z_{\mathbf p})$ with $f_2$ odd, the leading order vanishes since $\psi_0$ has even parity in $z_{\mathbf p}$. Integrating out $z_{\mathbf p}$ we then deduce from Lemma \ref{lem-h-2} that
\begin{equation}\label{est-proj-F2-Q}
\begin{aligned}
\left|\left<\mathrm F_2, Q\right>_{L^2}\right|&=\varep^{3/2}\left|(\sigma_1^*-\sigma) \sum_{j\in \Sigma_0}\int_{\msI_\mbp}h(\bm \gamma_{\mathbf p}'') \mathrm q_j \tilde \Theta_j\dd \tilde s_{\mathbf p}\right|\lesssim \varep^{5/2}|\sigma_1^*-\sigma| (1+ \|\hat{\mathbf p}\|_{\mbV_4^2})\|\mathbf q\|_{l^2}.
\end{aligned}
\end{equation}
Using the form of $\mrF_3$ for Lemma \ref{lem-def-Phi-p} we rewrite 
\begin{equation}\label{est-proj-F3-Q-0}
 \left<\mathrm F_3-\mrF_3^\infty, Q\right>_{L^2}=-\left<\phi_0'\Delta_{s_{\mbp}}\kappa_{\mbp}, Q\right>_{L^2}+\left<f_3(z_{\mbp}, \bm \gamma_{\mbp}'')-f_3^\infty, Q\right>_{L^2}.
 \end{equation} 
Applying Lemma \ref{lem-h-2} again, we bound the  lower order  term on the right
  \begin{equation}\label{est-F3-Q-1}
  \begin{aligned}
\left| \left<f_3(z_{\mbp}, \bm \gamma_{\mbp}'')-f_3^\infty, Q\right>_{L^2}\right|
 \lesssim &\varep^{3/2}\|\mathbf q\|_{l^2}(1+\|\hat{\mbp}\|_{\mbV_4^2}).
 \end{aligned}
 \end{equation}
For the higher order term including the  second derivative of curvature $\Delta_{s_\mbp}\kappa_\mbp$, since $\phi_0'$ is perpendicular to $\psi_0$ in $L^2(\mathbb R_{2\ell})$, the leading order vanishes yielding
\begin{equation}\label{est-Deltakappa-Z0-F3}
\left<\phi_0' \Delta_{s_{\mathbf p}}\kappa_{\mathbf p}, Q \right>_{L^2}=\varep^{3/2}\sum_{j\in \Sigma_0}\mathrm q_j\int_{\msI_\mbp} \left( h_1(\bm \gamma''_{\mathbf p})\Delta_{s_{\mathbf p}}\kappa_{\mathbf p}\tilde \Theta_j + h_2(\bm \gamma''_{\mathbf p})\Delta_{s_{\mathbf p}}\kappa_{\mathbf p}\varep \tilde \Theta_j' \right)\dd \tilde s_{\mathbf p}
\end{equation}  
for some $h_1, h_2$ satisfying Notation \ref{Notation-h}. Note that $h_k(\bm \gamma_\mbp'')$ for $k=1,2$ lies in $L^\infty$ since $\hat\mbp\in \mbV_2$, and  utilizing the curvature $H^2(\msI_\mbp)$ bound in Lemma \ref{lem-Gamma-p} yields
\begin{equation}\label{est-Deltakappa-Z0-0}
\begin{aligned}
\left|\left<\phi_0' \Delta_{s_{\mathbf p}}\kappa_{\mathbf p}, Q \right>_{L^2}\right|\lesssim \varep^{3/2}(1+\|\hat{\mathbf p}\|_{\mbV_4^2})\|{\mathbf q}\|_{l^2}.
\end{aligned}
\end{equation} 
Combining the above  estimate and  \eqref{est-F3-Q-1}  with \eqref{est-proj-F3-Q-0} and multiplying by $\varep^3$ implies
\begin{equation}\label{est-proj-F3-Q}
\left|\left<\varep^3(\mathrm F_3-\mrF_3^\infty), Q\right>_{L^2}\right|\lesssim \varep^{9/2}(1+ \|\hat{\mbp}\|_{\mbV_4^2})\|\mathbf q\|_{l^2}.
\end{equation}
In a similar manner we have
\begin{equation*}
\left|\left<\varep^{4}(\mathrm F_4-\mrF_4^\infty), Q\right>_{L^2}\right|  
\lesssim \varep^{9/2}(1+\|\hat{\mathbf p}\|_{\mbV_4^2})\|\mathbf q\|_{l^2},
\end{equation*}
which combined with the estimates on $\mathrm F_2$ given by \eqref{est-proj-F2-Q}  and $\mathrm F_3$ from \eqref{est-proj-F3-Q} yields
\begin{equation}\label{est-proj-F-Q}
\begin{aligned}
\left| \left<\mathrm F(\Phi_{\mathbf p} )-\mrF_m^\infty, Q\right>_{L^2}\right|\lesssim &\varep^{9/2}(1+\|\hat{\mathbf p}\|_{\mbV_4^2})\|\mathbf q\|_{l^2}.
 \end{aligned}
\end{equation}
 Combining estimates \eqref{est-proj-F-Q} and \eqref{proj-PiF-Q-1} with \eqref{proj-PiF-Q-0} completes the Lemma.

\end{proof}
\begin{remark}

It is essential to separate the pearling modes $Q$ from the fast modes, $w$.   The linear operator has a weaker coercivity on the pearling slow space, which is compensated for by the high-wave number estimates available for the pearling modes in Lemma\,\ref{lem-h-2}. These decrease the coupling of the residual to the pearling modes.  It is instructive to compare \eqref{est-F-bLw} with \eqref{proj-PiF-Q}.
\end{remark}

\subsection{Estimates on the Nonlinearity}
The estimates of Lemmas \ref{lem-est-w}, \ref{lem-est-q}, incorporate $L^2$-bounds of the nonlinear term $\mathrm N(v^\bot)$. The following lemma affords these bounds on $\mathrm N(v^\bot)$  in terms of $w$ and $\mbq$. 
\begin{lemma}\label{lem-est-N} If $\|v^\bot\|_{L^\infty(\Omega)}$ is bounded independent of $\varep$, then
\begin{equation}
\label{eq-Nest}
\|\mathrm N( v^\bot)\|_{L^2}\lesssim\varep^{-1}\Big(\rho^{-2}\left<\mbL   w, w\right>_{L^2} +\|\mbq\|_{l^2}^2\Big),
\end{equation}
Moreover, decomposing $v^\bot=w+Q$ as in \eqref{decomp-v}, we have the bound
\begin{equation*}
\| v^\bot\|_{L^\infty}\lesssim \varep^{-1}\left(\rho^{-1}\left<\mbL   w, w\right>_{L^2}^{1/2} +\|\mathbf q(t)\|_{l^2}\right).
\end{equation*}
\end{lemma}
\begin{proof}
From the definition  \eqref{def-N} of the nonlinear term $\mathrm N(v^\bot)$, with $\mathrm F$ given by \eqref{eq-FCH-L2-p} and $\mbL  $ given by \eqref{def-bLp}, 
 some rearrangements  lead to the equality 
\begin{equation*}
\begin{aligned}
\mathrm N(  v^\bot)=&-\Big(W''(u)-W''\Big)\Big(\varep^2 \Delta  v^\bot-{W'' v^\bot} \Big) -(\varep^2\Delta -W''+\varep\eta_2)\Big( W'(u) -W' -W'' v^\bot\Big) \\
&-\Big(W''(u)-W''-W''' v^\bot\Big) \Big(\varep^2\Delta \Phi_{\mathbf p} - W'(\Phi_{\mathbf p} )\Big),
\end{aligned}
\end{equation*}
where $W', W'', W'''$ are evaluated at $\Phi_{\mathbf p}$ unless  otherwise specified  and $u=\Phi_{\mathbf p} +v^\bot$.  The function $u$ is uniformly bounded in $L^\infty$ since $v^\bot$ is by assumption and $\varep^k\nabla^k\Phi_\mbp\in L^\infty$ is uniformly bounded for $k=1,\ldots 4,$ since $\Phi_\mbp$ is smooth in the inner variables. We deduce that the nonlinear term $\mathrm N$ satisfies the pointwise bound
\begin{equation*}
|\mathrm N( v^\bot)|\lesssim \|W\|_{C^6_c} \Big( \varep^2 |\nabla  v^\bot|^2+\varep^2|\Delta  v^\bot|| v^\bot|+|v^\bot|^2\Big),
\end{equation*}
which yields  the $L^2$ estimate
\begin{equation*}
\|\mathrm N( v^\bot)\|_{L^2}\lesssim  \varep^2\|v^\bot\|_{L^4}^2+\|v^\bot\|_{L^\infty} \varep^2\|\Delta v^\bot\|_{L^2}+\|v^\bot\|_{L^4}^2.
\end{equation*}
 In two space dimensions the Gargliardo-Nirenberg inequalities imply
\begin{equation}\label{ineq-G-N}
\|\nabla  v^\bot\|_{L^4}^2\lesssim \|\nabla^2  v^\bot\|_{L^2}\| v^\bot\|_{L^\infty}\quad \hbox{and}\quad \| v^\bot\|_{L^\infty}\lesssim \| v^\bot\|_{L^2}^{1/2}\| v^\bot\|_{H^2}^{1/2},
\end{equation}
and the $L^2$-estimate of $\mathrm N(v^\bot)$ reduces to
\begin{equation}\label{est-L2-N-1}
\begin{aligned}
\|\mathrm N( v^\bot)\|_{L^2} &\leq C\| v^\bot\|_{L^\infty}\left(\varep^2\|\nabla^2  v^\bot\|_{L^2}+\|v^\bot\|_{L^2}\right)\\
&\leq C\| v^\bot\|_{L^2}^{1/2}\varep^2\|\Delta  v^\bot\|_{L^2}^{3/2} +C\|v^\bot\|_{L^2}^{3/2}\|v^\bot\|_{H^2}^{1/2}\\
&\leq C\varep^{-1} \Big(\| v^\bot\|_{L^2}+\varep^2\|\Delta  v^\bot\|_{L^2}\Big)^2.
\end{aligned}
\end{equation}
From the decomposition $ v^\bot=w+Q$, we have
\begin{equation*}
\| v^\bot\|_{L^2}\lesssim \|w\|_{L^2}+ \|\mathbf q(t)\|_{l^2}, \quad \varep^2 \|\Delta  v^\bot\|_{L^2}\lesssim  \varep^2\|\Delta w\|_{L^2}+\|\mathbf q(t)\|_{l^2},
\end{equation*}
where we used the fact that $\varep^2\Delta$  is a uniformly bounded operator on $\mathcal Z_*^0$ in $L^2$ and hence on $Q$, see \eqref{eq-Lap-induced} and \eqref{def-Sigma}. 
The estimate \eqref{eq-Nest} follows from the coercivity Lemma \ref{lem-coer}. Applying the estimate \eqref{ineq-G-N} leads to
\begin{equation*}
\begin{aligned}
\|v^\bot\|_{L^\infty}&\leq  \varep^{-1}(\|w\|_{L^2}+\|\mathbf q\|_{l^2})^{1/2}(\varep^2 \|\Delta w\|_{L^2}+\|\mathbf q\|_{l^2})^{1/2}\\
&\lesssim \varep^{-1}\left(\rho^{-1} \left<\mbL  w, w\right>_{L^2}^{1/2}+\|\mathbf q\|_{l^2}\right).
\end{aligned}
\end{equation*}
The proof is complete.
\end{proof}

\subsection{Main Theorem}
\label{ss-MTheom}
In this sub-section we introduce thinner tubular neighborhoods $\mcV_R(\cM_b,\cO_\delta)\subset\mcU(\cM_b)$ of thickness $R$ defined over the open base $\cO_\delta \subset\cD_\delta$. We show that solutions of the gradient flow \eqref{eq-FCH-L2} that start inside of $\mcV_{R_1}(\cMb, \cO_\delta)$ remain in a slightly thicker neighborhood $\mcV_{R_2}(\cM_b, \cO_{2,\delta})$  so long as $\mbp$ remains in the slightly larger base $\cO_{2,\delta}$.   For $R\in(0,\varep^2]$, the tubular neighborhood with width $R$ and domain $\cO_{\delta}$ is defined as
\beq
\label{def-B-circ}
\mcV_R(\cM_b, \cO_{\delta}) =\left\{ u\in H^2(\Omega)\,\Bigl |\,
\inf_{\mbp\in\cO_{\delta} } \|u-\Phi_\mbp(\sigma)\|_{\Htwoin} <   R,\, \langle u-b_-\rangle_{L^2} = \frac{\varep M_0}{|\Omega|} \right\}.
\eeq
We introduce the nested base domains $\cO_{m,\delta}$ as the subsets of $ \cD_\delta$ that satisfy
\beq\label{assump-Ap}
\begin{aligned}
  \cO_{m,\delta} :=\left\{\mbp\in \mbR^{N_1} \,\bigl |\,  |\mrp_0| + \|\hat\mbp\|_{\mbV_1}< m\delta; \quad \|\hat\mbp\|_{\mbV_2}+ \varep\|\hat\mbp\|_{\mbV_4^2}< m \right\}, \qquad m=1,2.
\end{aligned}
\eeq
When $m=1$, we denote $\cO_{1,\delta}$ by $\cO_\delta$. The parameter $\delta$ will be chosen sufficiently small that Lemma \ref{lem-sigma-2} holds.  The condition on $\mrp_0$ insures that the pearling stability condition $(\mathbf{PSC})$ holds uniformly, see Lemma \ref{lem-sigma-2}; the uniform bound on $\varep\|\hat\mbp\|_{\mbV_4^2}$ insures the smoothness of the perturbed curve $\Gamma_\mbp$. 

From Lemma \ref{lem-est-V}, each of the a priori bounds on $\hat\mbp$ in \eqref{assump-Ap} are inferred from the single, stronger bound 
$\|\hat\mbp\|_{\mbV_3^2}\leq m\delta. $ Hence we introduce a parallel set of smaller but more easily defined domains,
\beq\label{assump-Ap'}
\cO^\circ_{m,\delta}:= \begin{aligned}
\left\{\mbp\in \mbR^{N_1} \;\bigl |\; |\mrp_0| + \|\hat\mbp\|_{\mbV_3^2}<  m  \delta\right\}\subset \cO_{m,\delta}.
\end{aligned}
\eeq

The equilibrium pearling stability condition arises from replacing $\sigma$ in the pearling stability condition \eqref{cond-P-stab} with its leading order equilibrium value $\sigma_1^*$, defined in \eqref{def-sigma1*}, 
\beq
\label{equil-PSC}
(\mathbf{PSC^*}) \qquad \sigma_1^* S_1+\eta_d \la_0>0.
\eeq 
The next lemma shows that if $(\mathbf{PSC^*})$ holds, then for a suitable admissible pair $(\Gamma_0, M_0)$ the $(\mathbf{PSC})$ holds uniformly for all $\mbp\in \cO_{2,\delta}$ provided that $\delta$ is sufficiently small. 

\begin{lemma}\label{lem-sigma-2}
Suppose that the equilibrium pearling stability condition \eqref{equil-PSC} holds and that $(\Gamma_0, M_0)$ is a admissible pair satisfying 
\beq\label{PSC-M0}
\Bigl|M_0- m_0|\Gamma_0| -B_2^\infty |\Omega|\sigma_1^*\Bigr|\leq \delta. 
\eeq
Then for $\mbp\in \cO_{2,\delta}$, the bulk parameter
 $\sigma=\sigma(\mbp)$ defined in \eqref{def-hatla} is uniformly bounded, i.e. $|\sigma|\lesssim 1$, and
 the pearling stability condition $(\mathbf{PSC})$ from \eqref{cond-P-stab}  holds uniformly for all $\mbp\in\cO_{2,\delta}$ provided that $\delta, \varep_0$  is sufficiently small, in terms of the domain, the system parameters, and $K_0,\ell_0$. 
\end{lemma}
\begin{proof}
From the bound of Lemma \ref{lem-sigma}, we estimate  
\beqs
 \left|\sigma(\mbp)-\frac{M_0-m_0|\Gamma_0|}{B_2^\infty|\Omega|}\right| \lesssim |\mrp_0| +\varep.
\eeqs
The uniform bound on $\sigma$ follows from the assumption on $\mrp_0$ since $\mbp\in \cO_{2,\delta}$. By assumption \eqref{PSC-M0},  $|\sigma(\mbp)-\sigma_1^*|\leq \delta$ and the pearling stability condition \eqref{cond-P-stab} holds uniformly. 
\end{proof}

\begin{lemma}\label{lem-pdot-l2}

For $\mbp\in \cO_{2,\delta}$, the temporal derivative of $\mbp$ satisfies the bound 
 \beqs
 \|\dot\mbp\|_{l^2}\lesssim  \varep^3
 +\varep^{3/2}\|v^\bot\|_{L^2} +\varep^{1/2}\|\mathrm N(v^\bot)\|_{L^2}+ \varep^{-1}\|v^\bot\|_{L^2} \|\dot\mbp\|_{l^2}. 
 \eeqs
\end{lemma}
\begin{proof}
We rewrite the equation \eqref{eq-v-Phi} as
\beqs
\p_t \Phi_\mbp =-\Pi_0 \mrF(\Phi_\mbp) -\mathrm{Re}(v^\bot), \qquad \mathrm{Re}[v^\bot]:= \p_t v^\bot + \Pi_0 \mbL v^\bot +\Pi_0\mathrm N(v^\bot). 
\eeqs
With a use of Lemmas \ref{lem-Phi_t-p}, \ref{lem-est-residual} and the a priori assumption on $\|\mbp\|_{\mbV_4^2}$ we derive
\beq\label{dotp-l^2}
\begin{aligned}
\|\dot\mbp\|_{l^2}&\lesssim \varep^{1/2}\|\Pi_{\mcZ_*^1}\p_{t}\Phi_\mbp\|_{L^2}\\
&\lesssim \varep^{1/2}\|\Pi_0\mrF(\Phi_\mbp)\|_{L^2}  +\varep^{1/2}\|\Pi_{\mcZ_*^1}\mathrm{Re}[v^\bot]\|_{L^2}\\
&\lesssim \varep^3
+\varep^{1/2}\|\Pi_{\mcZ_*^1}\mathrm{Re}[v^\bot]\|_{L^2}. 
\end{aligned}
\eeq
To estimate the projection of the remaidner $\mathrm{Re}[v^\bot]$, we deal with its terms one by one. First, we rewrite the projection of $\partial_tv^\bot$ as
\beqs
\left<\p_t v^\bot, Z_{\mbp, *}^{1k}\right>_{L^2} = \p_t \left<v^\bot, Z_{\mbp,*}^{1k}\right>_{L^2} - \left<v^\bot, \p_t  Z_{\mbp,*}^{1k}\right>_{L^2}.
\eeqs
The first term on the right hand side is zero since $v^\bot$ is perpendicular to the meandering slow space $\mcZ_{*}^1$; and the second term can be bounded with the aid of \eqref{est-L2-p-tZ*}. Combining these, we deduce 
\beqs
|\left<\p_t v^\bot, Z_{\mbp, *}^{1k}\right>_{L^2}| \lesssim  \varep^{-1} \|\dot\mbp\|_{l^2}\|v^\bot\|_{L^2} \qquad \forall k\in\Sigma_1,
\eeqs
which combined with a typical $l^2$-$l^\infty$ estimate and the  $N_1\lesssim \varep^{-1}$  implies
\beqs
\|\Pi_{\mcZ_*^1}\p_t v^\bot\|_{L^2}\lesssim  \varep^{-1}N_1^{1/2}\|\dot\mbp\|_{l^2}\|v^\bot\|_{L^2}\lesssim \varep^{-3/2}\|\dot\mbp\|_{l^2}\|v^\bot\|_{L^2}.
\eeqs
 Second,  we apply  Lemma \ref{thm-coupling est} to bound the projection of the linear term $\Pi_0 \mbL_\mbp v^\bot$,
\beqs
\left\|\Pi_{\mcZ_*^1}\Pi_0 \mbL_\mbp v^\bot\right\| \lesssim (\varep^2 + \varep^2 \|\hat\mbp\|_{\mbV_{4}^2}) \|v^\bot\|_{L^2}.
\eeqs
Finally, the projection of the nonlinear term can be  estimated trivially,
\beqs
\|\Pi_{\mcZ_*^1} \Pi_0 \mathrm N(v^\bot)\|_{L^2}\lesssim \|\mathrm N(v^\bot)\|_{L^2}.
\eeqs
These three estimates imply
\beqs
\|\Pi_{\mcZ_*^1}\mathrm{Re}[v^\bot]\|_{L^2}\lesssim \varep^{-3/2}\|v^\bot\|_{L^2}\|\dot\mbp\|_{l^2} + (\varep^2 + \varep^2 \|\hat\mbp\|_{\mbV_{4}^2}) \|v^\bot\|_{L^2}+\|\mathrm N(v^\bot)\|_{L^2}, 
\eeqs
which combined with \eqref{dotp-l^2} completes the proof. 
\end{proof}

\begin{lemma}\label{lem-IC-est}
Fix $K_0,\ell_0$, and assume $(\Gamma_0,M_0)$ is a admissible pair from $\mathcal A(K_0,\ell_0)$. Then there exists $\varep_0$ sufficiently small depending on $\delta$, and   a positive $T_0$ independent of $\delta, \rho$ and $\eps_0$ such that for all initial data $u_0\in \mcV_{\varep^{5/2}}(\cM_b, \cO_{\delta})$,
the projection parameter $\mbp(t)$ corresponding to the solution $u=u(t)$ remains in the open set $\cO_{2,\delta}$ for all $t\in[0,T_0\rho^{-1}]$ so long as $u$ remains in the tubular projection neighborhood $\mcU(\cM_b)$ for which the projection $\Pi_{\cM_b}$ is well-defined.
\end{lemma}
\begin{proof}
Since $u_0\in \mcV_{\varep^{5/2}}(\cM_b, \cO_\delta)$, then there exists $\mbp_0 \in \cO_\delta$ and  $v_0\in L^2$ satisfying $\|v_0\|_{\Htwoin}\leq  \varep^{5/2}$ such that $u_0=\Phi_{\mbp_0}+v_0$. We first note that for $\varep_0$ small enough, Lemma \ref{lem-Manifold-Projection} applies to $u_0$, and hence there exists $\mbp(0)\in \cD_\delta$ such that $\Phi_{\mbp(0)}=\Pi_{\cM_b} u_0$ satisfying 
\beq\label{IC-est-1}
\|\mbp(0)-\mbp_0\|_{l^2}\lesssim \varep^{3}. 
\eeq 
By the Fundamental Theorem of Calculus we bound the difference
\beqs
|\mrp_k(t)-\mrp_k(0)|\leq \int_0^t |\dot\mrp_k| \dd \tau \qquad t>0, 
\eeqs
for any  $k\in \Sigma_1$, which together  with the H\"older's inequality and the a priori assumption \eqref{A-0} implies
\beqs
\|\mbp(t)-\mbp(0)\|_{l^2} \lesssim  t\|\dot\mbp\|_{l^2}\lesssim \varep^3t. 
\eeqs
Combining with \eqref{IC-est-1} with the aid of triangle inequality implies
\beqs
\begin{aligned}
\|\mbp(t)-\mbp_0\|_{l^2} &\leq  \|\mbp(t)-\mbp(0)\|_{l^2} +\|\mbp_0-\mbp(0)\|_{l^2}\\
&\lesssim \varep^{3}+ \varep^3 t. 
\end{aligned}
\eeqs
Note that $\mbp_0\in\cO_\delta$, from the estimate above the length parameter $\mrp_0(t)$, as the first component of $\mbp(t)$, satisfies $|\mrp_0(t)|< 2\delta$ for $\varep_0$ small enough. It suffices to bound the difference of $\hat\mbp$ and $\hat\mbp_0$ in $\mbV_1$, $\mbV_2$ and $\mbV_4^2.$ By the embedding Lemma \ref{lem-est-V} with $N_1\lesssim \varep^{-1}\rho^{1/4}$ from  \eqref{est-N1}, we derive 
\beqs
\|\hat \mbp(t)-\hat \mbp_0\|_{\mbV_1} \lesssim \varep^{-1} \rho^{1/4} \|\hat \mbp(t)-\hat \mbp_0\|_{l^1} \lesssim \varep^{-3/2}\rho^{1/4}\| \mbp(t)- \mbp_0\|_{l^2}\lesssim \varep^{3/2} +\varep^{3/2}t,
\eeqs
 and  the following higher weighted estimate 
 \beqs
 \|\hat \mbp -\hat\mbp_0\|_{\mbV_2}+\varep\|\hat\mbp-\hat\mbp_0\|_{\mbV_4^2} \lesssim \varep^{-3}\rho \|\hat\mbp-\hat\mbp_0\|_{l^2} \lesssim \rho  +\rho t. 
 \eeqs
Noting $\mbp_0\in \cO_\delta$,  there exists $T_0>0$, independent of $\varep, \rho$, such that $\mbp \in \cO_{2,\delta}$ for any $t\in[0, T_0/\rho]$.  
\end{proof}

The following theorem presents the stability of the bilayer manifold up to its boundary. We recall that $(\Gamma_0, M_0)$ is a admissible pair with associated $N_1(\rho)$-dimensional bilayer manifold  $\cM_b(\Gamma_0, M_0; \rho)$ defined in Definition \ref{def-bM0}, $\rho$ is the spectral cut-off introduced in Definition \ref{def-slow-space} and is sufficiently small as required by Theorem\,\ref{thm-coupling est} and Lemma \ref{lem-IC-est}. The slow spaces $\mcZ_*^k,\mcZ_*$ are defined in \eqref{def-Z*} and $\mcV_R(\cM_b, \cO_{\delta})$ are the tubular neighborhoods with $\Htwoin$-width $R$ and base $\cO_\delta$ defined in \eqref{def-B-circ}-\eqref{assump-Ap}. The parameter $\delta>0$ is a fixed sufficiently small as required by Lemma\,\ref{lem-Manifold-Projection}.

\begin{thm}
\label{thm:Main}  
 Consider the mass-preserving flow \eqref{eq-FCH-L2} subject to periodic boundary  conditions on the  domain $\Omega=[-L, L]^2$.
  Assume that the equilibrium pearling stability condition $(\mathbf{PSC^*)}$--\eqref{equil-PSC}, holds for the given system parameters. Fix $K_0,\ell_0$, then there exists an $\varep_0$ and a $C>0$ such that 
  for each admissible pair $(\Gamma_0,M_0)$ from $\mathcal A(K_0, \ell_0)$, and for all $\varep\in(0,\varep_0)$, the bilayer manifold $\cM_b(\Gamma_0,M_0)$ has the following properties. Each solution $u=u(t)$ corresponding to initial data   $u_0 \in  \mcV_{\varep^{5/2} }(\cM_b, \cO_\delta)$ remains in the slightly larger tubular neighborhood $\mcV_{C\varep^{5/2}}(\cM_b, \cO_{2,\delta})\subset \mcU(\cM_b)$ 
  so long as its projected meander parameters $\mbp$ remain in $\cO_{2,\delta}$.  
Denoting this interval of residency as $[0,T]$, then $T>0$ and during this interval $u$ admits the dynamic decomposition
\beqs
u(t)=\Phi_\mbp(t;\sigma)+v^\bot, \qquad v^\bot= Q(t; \mbq)+ w(t), \qquad \forall t\in [0,T],
\eeqs
 where $Q=\Pi_{\mcZ_*^0}v^\bot\in \mcZ_{*}^0(\mbp,\rho), w\in \mcZ_*^\bot(\mbp,\rho)$.  In particular, the orthogonal perturbation $v^\bot$ and its fast and pearling decomposition satisfy  
\beq\label{MT-est-v}
\|v^\bot\|_{\Htwoin}\lesssim \|w\|_{\Htwoin}+ \|Q\|_{\Htwoin}\leq C\varep^{5/2} \rho^{-2}.
\eeq
\end{thm}
\begin{proof}
Since $u_0\in \mcV_{\varep^{5/2}}(\cM_b,\cO_\delta)\subset \mcU(\cM_b)$, Lemma \ref{lem-Manifold-Projection} implies the existence of the decomposition $u_0=\Phi_{\mbp(0)}+v_0^\bot=\Phi_{\mbp(0)}+Q_0(\mbq(0))+w_0$ with $Q_0\in \mcZ_*^0(\Gamma_{\mbp(0)})$ and $w_0\in\mcZ_*^\perp(\Gamma_{\mbp(0)})$ satisfying,
\begin{equation}
\label{MT-init-est}
\begin{aligned}
 \|w_0\|_{\Htwoin}\lesssim \varep^{5/2}, \qquad \|\mathbf q(0)\|_{l^2}\lesssim \varep^{5/2}.\h{20pt}
\end{aligned}
\end{equation} 
We establish the existence of  positive constants  $K_1, K_2$ independent of $\varep,\rho,\delta$ and $T>0$ for which the bounds 
\begin{equation}\label{assump-A}
({\bf A})\qquad 
\begin{aligned}
& \left<\mbL  w, w\right>_{L^2}\leq K_1 \varep^5\rho^{-2}, \qquad \|\mbq\|_{l^2}^2\leq  K_2 \varep^{5}\rho^{-4},  
\end{aligned}
\end{equation}
 hold uniformly for all $t\in[0,T]$ as long as $\mbp(t)\in \cO_{2,\delta}$ on the interval.  In the argument below we modify the notation of Section\,\ref{ssec-Notation} writing `$A\lesssim B$'  to denote '$A\leq CB$' for a constant $C$ that is independent of $K_1, K_2$ as well as the small parameters $\varep, \rho, \delta.$  The existence of a $T>0$ is assured by  \eqref{MT-init-est} and Lemma \ref{lem-IC-est}. 
\medskip

First, applying the coercivity Theorem \ref{lem-coer} and assumption $(\mathbf A)$ implies
\beq\label{main-est-w-L2}
\|w\|_{\Htwoin}^2\lesssim \rho^{-2}\left<\mbL w, w\right>_{L^2} \lesssim K_1\varep^5\rho^{-4}. 
\eeq
Then from  the relation  $\|Q\|_{\Htwoin}\sim \|\mbq\|_{l^2}$, Lemma \ref{lem-est-N} and assumption (${\bf A}$) we   bound the $L^2$-norm of $v^\bot=Q+w$ and nonlinear term $\mathrm N(v^\bot)$ as 
\begin{equation}\label{main-est-N}
\begin{aligned}
\|v^\bot\|_{\Htwoin}^2&\lesssim \|w\|_{\Htwoin}^2 + \|\mbq\|_{l^2}^2,
\qquad \|\mathrm N(v^\bot)\|_{L^2}^2 
\lesssim  (K_1^2+K_2^2) \varep^{8}\rho^{-8}.
\end{aligned}
\end{equation}
It suffices to verify the assumption $(\mathbf A)$ for all $t\in[0,T]$, on which $\mrp(t)\in \cO_{2,\delta}$, in order to establish the main estimate \eqref{MT-est-v} in the Theorem.  
We first note from Lemma \ref{lem-pdot-l2}  for $K_1, K_2>1$, 
\beq
\|\dot\mbp\|_{l^2}^2\lesssim \varep^6+ (K_1^2+K_2^2)\varep^8\rho^{-8} +(K_1+K_2)\varep^3\rho^{-4}\|\dot\mbp\|_{l^2}^2. 
\eeq
Since $K_1, K_2$ are independent of $\varep$, for $\varep\in(0,\varep_0)$ with  $\varep_0$ small enough depending on $\rho$ we have 
\beq \label{est-pdot-l2}
\|\dot\mbp\|_{l^2}^2\lesssim \varep^6,
\eeq 
independent of $K_1, K_2$, and the a priori assumption \eqref{A-0} holds for $\mbp\in \cO_{2,\delta}$.  The pearling stability condition $(\mathbf{PSC})$ \eqref{cond-P-stab} holds uniformly for all $t\in[0,T]$ by Lemma \ref{lem-sigma-2}. In particular Lemma \ref{lem-est-q} applies.  We restate the key estimates of Lemmas \ref{lem-est-w} and  \ref{lem-est-q} as
\begin{equation}\label{main-est-w-q-1}
\begin{aligned}
&\frac{\mrd }{\mrd  t}\left<\mbL w,w\right>_{L^2}+\|\mbL   w\|_{L^2}^2\lesssim  \varep^{-1}\|\dot\mbp\|_{l^2}^2+\varep^5+\varep^2\rho^{-4}(\|\mbq\|_{l^2}^2+\|\dot{\mbq}\|_{l^2}^2) +\varep^7(1+\|\hat\mbp\|_{\mbV_4^2}^2) +\|\mathrm N(v^\bot)\|_{L^2}^2;\\
&\p_t \|\mbq \|_{l^2}^2+C\varep\|\mbq\|_{l^2}^2\lesssim \varep \|w\|_{L^2}^2+\|\dot \mbp\|_{l^2}^2+\varep^{-1}\|\mathrm N(v^\bot)\|_{L^2}^2+\varep^8+\varep^8\|\hat\mbp\|_{\mbV_4^2}^2,
\end{aligned}
\end{equation}
and  $\|\hat \mbp\|_{\mbV_{4}^2} \lesssim \varep^{-1} $ for $\mbp\in \cO_{2,\delta}$. From the $l^2$-bound of the pearling modes $\mbq$ from (${\bf A}$), the estimates of the fast modes $w$ and the nonlinear terms $\mathrm N$ from \eqref{main-est-w-L2}-\eqref{main-est-N},  and $\mbp\in \cO_{2,\delta}$ we reduce the $l^2$ bound of $\dot\mbq$ in Lemma \ref{lem-est-q} to 
\begin{equation}\label{main-est-q'-2}
\begin{aligned}
\|\dot{\mbq}\|_{l^2}^2
&\lesssim \varep^7+(K_1+K_2)\varep^5\rho^{-4} +(K_1^2+K_2^2)\varep^8\rho^{-8},
\end{aligned}
\end{equation}
 where  the first  term  on the right-hand side comes from the a priori assumptions on $\|\hat\mbp\|_{\mbV^2_4}^2$ and estimate of $\|\dot \mbp\|_{l^2}$ in \eqref{est-pdot-l2}. 
Combining this with the first inequality in \eqref{main-est-w-q-1}, and reusing \eqref{assump-A}-\eqref{main-est-N}  and $\mbp\in \cO_{2,\delta}$  yields for $K_1, K_2>1$
\begin{equation}\label{main-est-w-1}
\frac{\mrd }{\mrd t}\left<\mbL   w,w\right>_{L^2} + C\rho^2\left<\mbL   w,w\right>_{L^2}\lesssim  \varep^5+(K_1^2+K_2^2)\varep^7\rho^{-12}, 
\end{equation}
 where the term with the dominant power of $\varep$ arises from the inhomogeneous term on the right-hand  side of \eqref{main-est-w-q-1}. 
Integrating this estimate in time we obtain
 \begin{equation}\label{main-est-w}
 \begin{aligned}
 \left<\mbL   w,w\right>_{L^2} \leq & \left< \mbL_{\mbp(0)}  w_0,w_0\right>_{L^2} \, e^{-C \rho^2 t} +C \varep^5\rho^{-2} +C(K_1^2+K_2^2) \varep^7 \rho^{-14},\\
 \lesssim & \|w_0\|_{\Htwoin}^2  +  \varep^5\rho^{-2}+(K_1^2+K_2^2) \varep^7 \rho^{-14}\\
  \leq & \, C_1( \varep^5\rho^{-2}+(K_1^2+K_2^2) \varep^7 \rho^{-14}),
 \end{aligned}
 \end{equation}
for some positive constant $C_1$ independent of $\varep, \rho$ and $K_1, K_2$. Here we used  \eqref{MT-init-est} to bound $w_0$. 
 
Turning to the $\mbq$ estimate in  \eqref{main-est-w-q-1} and utilizing   \eqref{main-est-w-L2}-\eqref{main-est-N}, \eqref{est-pdot-l2}  to bound the first three terms on the right-hand side,  we obtain the bound, valid for $\mbp\in \cO_{2,\delta}$,
\beqs
\p_t \|\mbq\|_{l^2}^2+C \varep\|\mbq\|_{l^2}^2\lesssim K_1 \varep^6\rho^{-4}+(K_1^2+K_2^2) \varep^7,
\eeqs
 where the dominant $\varep$-term arises from $\|w\|_{L^2}$.
We integrate this inequality and apply the initial value estimates \eqref{MT-init-est} to $\|\mbq_0\|_{l^2}$ to obtain 
\begin{equation}\label{main-est-q}
\begin{aligned}
\|\mbq\|_{l^2}^2 \leq e^{-C\varep t}\|\mbq(0)\|_{l^2}^2+C K_1 \varep^5\rho^{-4}  +C(K_1^2+K_2^2)\varep^6  \leq C_2\Big(K_1 \varep^5 \rho^{-4}  +(K_1^2+K_2^2) \varep^6\Big).
 \end{aligned}
\end{equation} 
Here $C_2$ is a positive constant independent of $\varep, \rho$ and $K_1, K_2$. Taking $K_1=2C_1, K_2=2C_2K_1$ and $\varep_0$ small enough, we combine \eqref{main-est-w} and \eqref{main-est-q} to establish ({\bf A}). Together with the first inequality in \eqref{main-est-N}, this completes the proof. 
\end{proof}

 \section*{Acknowledgement}  The second author thanks the NSF DMS for their support through grant 1813203. Both authors thank Gurgen Hayrapetyan for sharing preliminary results on this problem for the weak FCH that arose out of his thesis.

\section{Appendix}
We present some technical and intermediate results.    
\subsection{The variation of  the local coordinates with respect to $\mbp$}

\begin{lemma}\label{lem-A}
Assume  $\|\hat\mbp\|_{\mbV_1}\ll 1$ and $\Gamma_0\in \cG_{K_0,2\ell_0}^4$, the normalized length constant $A(\mbp)$ defined in  \eqref{def-A(p)}  admits the approximation
\beqs
A(\mbp)=  1 + O(\|\hat\mbp\|_{\mbV_1})
\eeqs
Furthermore, the rate of change of $A(\mbp)$ with respect to $\mbp$ can be bounded by 
\beqs
\|\nabla_\mbp A(\mbp)\|_{l^2} \lesssim \|\hat\mbp\|_{\mbV_2^2}.
\eeqs
If $\Gamma_0$ is a circle then we have the isopermetric bound
\beqs
A(\mbp)=  1 + O(\|\hat\mbp\|_{\mbV_1}^2). 
\eeqs
\end{lemma}
\begin{proof}
The function  $A(\mbp)$ is the proportional change in length of $\bm\gamma_0$ to the perturbation $\bm \gamma_{\bar p}$  given in \eqref{def-gamma-barp} that excluded the radial perturbation. In light of \eqref{def-ts}, taking the derivative of \eqref{def-gamma-barp} we find
\beq
\bm \gamma_{\bar p}'=(1-\kappa_0 \bar p(\tilde s)) \bm \gamma_0' +  \bar p'(\tilde s) |\bm \gamma_\mbp'| \mbn_0(s).
\eeq
Taking absolute value, and using the orthogonality between and tangent $\bm \gamma_0'$ and normal $\mbn_0$, we deduce
\beq\label{def-|gamma-barp|}
|\bm \gamma_{\bar p}'|= \Big((1-\kappa_0 \bar p(\tilde s))^2 +|\bar p'(\tilde s)|^2 |\bm \gamma'_\mbp|^2\Big)^{1/2}
\eeq
Considering terms including $\bar p$ as small, the right-hand side has leading order term $1-\kappa_0 \bar p(\tilde s)$, and hence
\beq\label{est-A(p)-1}
A(\mbp)=\frac{1}{|\Gamma_0|} \int_{\msI}  (1-\kappa_0(s) \bar p(\tilde s))  \dd s  +O(\|\hat\mbp\|_{\mbV_1}^2).
\eeq
The approximation of $A(\mbp)$ follows if $\Gamma_0$ is a general smooth curve.  

When $\Gamma_0$ is a circle we return to \eqref{est-A(p)-1} and remark that the curvature $\kappa_0(s)=\kappa_0$ is a constant while $\bar p(\tilde s)=\sum_{i=3}^{N_1-1} \mrp_i \tilde \Theta_i(\tilde s)$ inherits a zero-integral with respect to $\dd\tilde s= |\bm \gamma_\mbp'|\dd s$ from the Laplace-Beltrami eigenmodes $\{\tilde \Theta_i(\tilde s)\}_{i\geq 3}$. 
From the definition of $\bm \gamma_\mbp$ given in \eqref{def-gamma-p}, identity \eqref{def-|gamma-barp|}, and the general form expansion \eqref{est-A(p)-1} of $A(\mbp)$ we find
\beqs
\begin{aligned}
|\bm \gamma_\mbp'|&=\frac{1+\mrp_0}{A(\mbp)} | \bm \gamma_{\bar p}'| =1+O(\|\hat\mbp\|_{\mbV_1}). 
\end{aligned}
\eeqs
Changing variables from $\dd s$ to $\dd\tilde s$ and using the zero
average of $\bar p$ with respect to $\dd\tilde s$ we derive the isoperimetric  bound which shows that circles are critical points of perturbations $\bar p$ that do not change the effective radius. 

 It remains to  estimate the rate of change of $A$ with respect to $\mrp_j$. By plugging  $|\bm \gamma_\mbp'|=(1+\mrp_0) |\bm \gamma_{\bar p}'|$ into the right hand side of \eqref{def-|gamma-barp|} and solving for $|\bm \gamma_\mbp'|$ we find
\beqs
|\bm \gamma_{\bar p}'|^2 = \frac{(1-\kappa_0 \bar p)^2}{1-(1+\mrp_0)^2|\bar p'|^2}.
\eeqs
To calculate the derivative with respect to $\mrp_j$, we need the derivative of $\bar p$ and $\bar p'$. In fact,
\beq\label{de-gamma-barp}
2|\bm \gamma_{\bar p}'| \p_{\mrp_j}|\bm \gamma_{\bar p}'| =  \frac{(1-\kappa_0 \bar p) \p_{\mrp_j} \bar p}{1-(1+\mrp_0)^2|\bar p'|^2} +\frac{(1-\kappa_0 \bar p)^2}{(1-(1+\mrp_0)^2|\bar p'|^2)^2} \Big[2\delta_{j0}(1+\mrp_0) |\bar p'|^2 +(1+\mrp_0)^2 2\bar p' \p_{\mrp_j} \bar p'  \Big].
\eeq
On the other hand, by the definition of $\bar p$, 
\beqs
\begin{aligned}
\p_{\mrp_j} \bar p  &= \tilde \Theta_j \bm 1_{\{j\geq 3\}} +\bar p'   (\p_{\mrp_j} \tilde s - \delta_{j0}\tilde s/(1+\mrp_0)); \\
\p_{\mrp_j} \bar p' &= \tilde \Theta_j' \bm 1_{\{j\geq 3\}} +\bar p''  (\p_{\mrp_j} \tilde s - \delta_{j0}\tilde s/(1+\mrp_0)). 
\end{aligned}
\eeqs
We note $\|\nabla_\mbp\tilde s\|_{l^2}\lesssim 1$ by its definition in \eqref{def-ts}.  The gradient estimate of $A$ follows by combining  these identities with \eqref{de-gamma-barp},  integrating by parts, the a priori bound on $\hat\mbp$, and Lemma \ref{Notation-e_i,j}. 
 \end{proof}

The following lemma estimates the $\mbp$-variation of the  whiskered coordinates  associated to  $\Gamma_{\mathbf p}$. It controls the difference between the local coordinates 
$(s_{\mathbf p}, z_{\mathbf p})$ for $\Gamma_\mbp$ and $(s,z)$ for $\Gamma_0$  in terms of the perturbation $\mathbf p$.  
\begin{lemma}
\label{lem-change-of-coord}
Let $(s_{\mathbf p}, z_{\mathbf p})$ be the local coordinates subject to $\Gamma_{\mathbf p}$ with whisker length $2\ell$. Under assumption \eqref{A-00} the tangent coordinate $s_\mbp$ has bounded variation on the domain $|\varep z_\mbp|\ll 1$, satisfying 
\beqs
\|\nabla_\mbp s_{\mbp}\|_{L^2(\msI)}\lesssim 1.
\eeqs 
The  normal local coordinate $z_\mbp$ varies quickly with respect to $\mrp_j$ satisfying

$$ \frac{\p z_\mbp}{\p \mathrm p_0} =-\frac{1}{\varep}\left(\frac{R_0+\bar p }{A} \left(1-(1+\mrp_0)\p_{\mrp_0}\ln A  \right)- \frac{\tilde s_\mbp \bar p'}{A(1+\mrp_0)} \right)\mathbf n_0\cdot \mathbf n_{\mathbf p},\quad  j=0,
$$


$$\hspace{0.1in} \frac{\p z_{\mathbf p}}{\p \h{0.5pt} \mathrm p_j}= -\frac{\bE_j\cdot{\mathbf n_{\mathbf p}}}{\varep\sqrt{2\pi R_0}},\hspace{2.9in} j=1, 2,
$$

 
$$\frac{\p z_{\mathbf p}}{\p \h{0.5pt} \mathrm p_j}   =-\frac{1}{\varep}\left(\tilde \Theta_j-\frac{(1+\mrp_0)\p_{\mrp_j}\ln A}{A} (R_0+\bar p) \right)\mathbf n_0\cdot \mathbf n_{\mathbf p},  \hspace{0.9in} j\geq 3. \hspace{0.0in}
$$

Moreover, 
we have the estimates
\begin{equation}\label{est-diff-s-z-p}
\|s_{\mathbf p}- s\|_{L^\infty(\Gamma_\mbp^{2\ell})}\lesssim \|\mathbf p\|_{\mbV_0},\quad \|z_{\mathbf p}-z\|_{L^\infty(\Gamma_\mbp^{2\ell})}\leq \varep^{-1}\|\mathbf p\|_{\mbV_0}.
\end{equation}
\end{lemma}
\begin{proof}
Any $x\in \Gamma_{\mathbf p}^{2\ell}$ can be expressed in the local coordinates of both $\Gamma_\mbp$ and $\Gamma_{0}$. Equating these yields
%
\begin{equation}\label{eq-szp}
\bm\gamma_0(s)+\varep z\mathbf n_0(s)=\bm \gamma_{\mathbf p}(s_{\mathbf p})+\varep z_{\mathbf p}\mathbf n_{\mathbf p}(s_{\mathbf p}).
\end{equation}
 Taking the derivative of \eqref{eq-szp} with respect to $\mathrm p_j$, the $j$-th component of the vector $\mathbf p$, yields
\begin{equation}\label{eq-szp-1}
\begin{aligned}
0&=\frac{\p\bm \gamma_{\mathbf p}}{\p \h{0.5pt}\mathrm p_j}(s_{\mathbf p})+\bm \gamma_{\mathbf p}'\frac{\p s_{\mathbf p}}{\p \h{0.5pt}\mathrm p_j} +\varep \frac{\p z_{\mathbf p}}{\p \h{0.5pt}\mathrm p_j} \mathbf n_{\mathbf p}(s_{\mathbf p})+\varep z_{\mathbf p}\frac{\p\h{0.5pt} \mathbf n_{\mathbf p}}{\p\h{0.5pt} \mathrm p_j}(s_{\mathbf p})+\varep z_{\mathbf p} \mathbf n_{\mathbf p}'  \frac{\p s_{\mathbf p}}{\p \h{0.5pt} \mathrm p_j}.
\end{aligned}
\end{equation}
The vectors $\bm \gamma_{\mathbf p}'$ and $\mathbf n_{\mathbf p}$ are perpendicular to each other since $\bm \gamma_{\mathbf p}'$ lies in 
the tangent space while $\mathbf n_{\mathbf p}$ is the normal vector. Taking the dot product of \eqref{eq-szp-1}) with $\bm \gamma_{\mathbf p}'$ and rearranging, we obtain 
\begin{equation}\label{eq-sp-1}
\frac{\p s_{\mathbf p}}{\p \h{0.5pt} \mathrm p_j}=\frac{1}{(1-\varep z_{\mathbf p} \kappa_{\mathbf p})|\bm \gamma_{\mathbf p}'|^2} \left(-\frac{\p \bm \gamma_{\mathbf p}}{\p \h{0.5pt}\mathrm p_j}\cdot \bm\gamma'_{\mathbf p}+\varep z_{\mathbf p}\frac{\p \bm\gamma_{\mathbf p}'}{\p \h{0.5pt} \mathrm p_j}\cdot\mathbf n_{\mathbf p}\right).
\end{equation}
Here we used that  $\bm \gamma_{\mathbf p}'\cdot \mathbf n_{\mathbf p}$ is zero so that $\p_{\mathrm p_j}(\bm \gamma_{\mathbf p}'\cdot \mathbf n_{\mathbf p})=0$,  and definition of 
$\kappa_{\mathbf p}$ given in Lemma \ref{lem-Gamma-p}.  Taking the dot product of identity \eqref{eq-szp-1} with $\mathbf n_{\mathbf p}(s_{\mathbf p})$ and arguing as above  we find 
\begin{equation}\label{id-z'-p}
\frac{\p z_{\mathbf p}}{\p \h{0.5pt} \mathrm p_j}=-\frac{1}{\varep}\frac{\p \bm\gamma_{\mathbf p}}{\p\h{0.5pt} \mathrm p_j} \cdot \mathbf n_{\mathbf p}. 
\end{equation} 
From the definition of $\bm \gamma_{\mathbf p}$ in \eqref{def-gamma-p} and \eqref{def-barp},  for $j\geq 3$ it holds that
\begin{equation}\label{de-gamma-p-j}
\begin{aligned}
\frac{\p \bm\gamma_{\mathbf p}}{\p\h{0.5pt} \mathrm p_j}&= \tilde \Theta_j(\tilde s_{\mathbf p}) \mathbf n_0 - \frac{\bm \gamma_0+\bar p(\tilde s_\mbp) \mbn_0}{A^2}(1+\mrp_0) \p_{\mrp_j} A;\\
 \frac{\p \bm \gamma'_{\mathbf p}}{\p \h{0.5pt} \mathrm p_j}&=\left(  \tilde \Theta_j'(\tilde s_{\mathbf p}) |\bm \gamma_\mbp'|  -\frac{\bar p'(\tilde s_\mbp) }{A^2}(1+\mrp_0) \p_{\mrp_j}A  \right)\mbn_0 \\
 &\qquad - \left( \kappa_0 \tilde \Theta_j(\tilde s_\mbp) + \frac{1-\kappa_0 \bar p(\tilde s_\mbp)}{A^2}\right) (1+\mrp_0)\p_{\mrp_j}A\bm \gamma_0'.
\end{aligned}
\end{equation}
Here to obtain the second identity, we can take derivative with respect to $s_\mbp$ directly and $\mbn_0'=-\kappa_0 \mbn_0$. Secondly, it's a bit more complicate for the case $j=0$ due to the dependence of $\tilde \Theta_j$ on $|\Gamma_\mbp|=(1+\mrp_0)|\Gamma_0|$. Indeed,  $\p_{\mrp_0}\tilde \Theta_j=- \tilde s_\mbp \tilde \Theta_j'/(1+\mrp_0)^2$ by  its definition in \eqref{def-barp} which  furthermore implies $\p_{\mrp_0} \bar p=-\tilde s_\mbp\bar p'/(1+\mrp_0)^2$, and hence
\beq\label{de-gamma-p-0}
\begin{aligned}
\frac{\p \bm \gamma_\mbp}{\p \mrp_0}= (1-(1+\mrp_0) \p_{\mrp_0} \ln A)\frac{\bm \gamma_0 +\bar p(\tilde s_\mbp)\mbn_0}{A} -\frac{1+\mrp_0}{A} \frac{1}{(1+\mrp_0)^2} \tilde s_\mbp \bar p' \mbn_0. 
\end{aligned}
\eeq
Since $\bm \gamma_{\mathbf p}'$ is independent of $\mathrm p_1$ and $\mathrm p_2$, we have $\p_{\mathrm p_j}\bm \gamma_\mbp'=0$ for $j=1,2$. 
In addition, from \eqref{def-gamma-p} we have  
\beq\label{de-gamma-p-j=1,2}
\p_{\mathrm p_j}\bm \gamma_{\mathbf p}=\bE_j/\sqrt{2\pi R_0}, \qquad j=1,2.
\eeq
The expressions for the derivatives of $s_{\mathbf p}$ and  $z_\mbp$  with respect to $\mathrm p_j$ follow by plugging \eqref{de-gamma-p-j}, \eqref{de-gamma-p-0} or \eqref{de-gamma-p-j=1,2} into \eqref{eq-sp-1}-\eqref{id-z'-p} with the aid of $\bm \gamma_0=R_0\mbn_0$ and $\kappa_0=-1/R_0$. The  estimates \eqref{est-diff-s-z-p} follow directly from the Mean Value Theorem.
\end{proof}

\begin{lemma}\label{lem-de-hatla}
For $\mbp\in \cD_\delta$ introduced in \eqref{A-00},  the sensitivity of $\sigma$ defined in\eqref{def-hatla} to $\mbp$  can be bounded by
\begin{equation*}
\begin{aligned}
 \|\nabla_{\mbp} \sigma(\mbp)\|_{l^2}\lesssim 1.
\end{aligned}
\end{equation*}
\end{lemma}
\begin{proof}
This is a direct result of the Lemma \ref{lem-change-of-coord} and the definition of $\sigma(\mbp)$, details are omitted.
\end{proof}

\begin{lemma}\label{lem-Phi_t-p} The bilayer \muckmuck $\Phi_\mbp$ defined in  Lemma \ref{lem-def-Phi-p}, satisfies the expansion 
\beqs
\frac{\p \Phi_\mbp}{\p \mrp_j} = \frac{1}{\varep}  (\phi_0' +\varep \phi_1')\xi_j(s_\mbp)+\varep \mathrm R_j,
\eeqs
where $\xi_j(s_\mbp):=\varep \frac{\p z_\mbp}{ \p \mrp_j}$  and the remainder $\mathrm R=(\mathrm R_j)$ can be  bounded as a vector function in $L^2$   
\beqs
\|\mathrm R\|_{L^2}\lesssim 1; \qquad \|\Pi_{\mcZ_*^1}\mathrm R\|_{L^2}\lesssim  \varep^{1/2}.
\eeqs
Moreover, the quantity $\p_t \Phi_\mbp$ satisfies the estimates $L^2$ and $L^\infty$ estimates
 \beqs
 \|\p_t \Phi_\mbp\|_{L^2}\sim \|\Pi_{\mcZ_*^1}\p_t\Phi_\mbp\| \sim \varep^{-1/2} \|\dot\mbp\|_{l^2}; \qquad \|\Pi_{\mathcal Z_{ *}^0} \p_t \Phi_\mbp\|_{L^2}\lesssim \varep^{1/2}\|\dot \mbp\|_{l^2},
 \eeqs
 \begin{equation*} 
\begin{aligned}
\left\|\dot {\mathbf p}\cdot \nabla_{\mathbf p}\Phi_{\mathbf p} \right\|_{L^\infty}+\|\dot{\mathbf p}\cdot\nabla_{\mathbf p} (\varep^2\Delta\Phi_{\mathbf p} )\|_{L^\infty} \lesssim  \varep^{-3/2}\|\dot{\mbp}\|_{l^2}.
\end{aligned}
\end{equation*}
 \end{lemma}

\begin{proof}
From the definition of $\Phi_\mbp$ in Lemma \ref{lem-def-Phi-p}, we calculate 
\beqs
\frac{\p \Phi_\mbp}{\p \mrp_j} = (\phi_0'+\varep \phi_1'+\varep^2 \p_{z_\mbp} \phi_2 +\varep^3\phi_3') \frac{\p z_\mbp}{\p \mrp_j} +\varep^2 \p_{s_\mbp} \phi_2 \frac{\p s_\mbp}{\p \mrp_j} +\varep \p_{\sigma} (\phi_1+\varep \phi_2) \frac{\p \sigma}{\p \mrp_j}.
 \eeqs
Combining with Lemmas \ref{lem-change-of-coord} and \ref{lem-de-hatla}  we obtain the expressions for derivatives of $\Phi_\mbp$ with respect to $\mrp_j$,
\beqs
\mathrm R_j:=\varep^2( \p_{z_\mbp} \phi_2 +\varep\phi_3') \frac{\p z_\mbp}{\p \mrp_j} +\varep^2 \p_{s_\mbp} \phi_2 \frac{\p s_\mbp}{\p \mrp_j} +\varep \p_{\sigma} (\phi_1+\varep \phi_2) \frac{\p \sigma}{\p \mrp_j}. 
\eeqs
We remark that the leading order term comes from $\p_\sigma \phi_1=B_{\mbp,2}$ which is nonzero as $|z_\mbp|\to \infty$. This fact combined with Lemma \ref{lem-de-hatla} and Lemma \ref{lem-change-of-coord} yields the $L^2$ estimate of $\mathrm R$. The estimate on the projection of $\mathrm R$ to the meandering slow space is similar with the exception that the functions in $\mcZ_*^1$ are localized, decaying exponentially fast to zero as $|z_\mbp|\to \infty$. This contributes an extra factor of $\varep^{1/2}$ to the bound. 

The time derivative of $\Phi_{\mathbf p}$ satisfies the chain rule  
\begin{equation}\label{def-ptPhi}
\p_t \Phi_{\mathbf p} =\dot{\mathbf p}\cdot \nabla _{\mathbf p}\Phi_{\mathbf p}.
\end{equation} 
The first $L^2$ estimate on this quantity follows directly from the expressions of $\frac{\p\Phi_\mbp}{\p \mrp_j}$, the orthogonality of  $\{\tilde \Theta_j\}$, and $l^2$ bound of $\nabla_\mbp A$ in Lemma \ref{lem-A}. Considering the $L^2$ estimate of the projection to the pearling space,  the leading order of  $\p_{\mrp_j}\Phi_\mbp$ has odd parity in $z_\mbp$ while the leading order terms in  $Z_{\mbp,*}^{0j}$ for $j\in \Sigma_0$ have even parity. This renders the projection higher order of $\varep$, and establishes the bound. 

The final $L^\infty$-norm bound is obtained from the form of $\Phi_\mbp$  and the fact that $\varep^2 \Delta_{s_\mbp}$ is a bounded operator when acting on  $\Phi_\mbp$. More explicitly, 
\begin{equation*} 
\begin{aligned}
\left\|\dot {\mathbf p}\cdot \nabla_{\mathbf p}\Phi_{\mathbf p} \right\|_{L^\infty}+\|\dot{\mathbf p}\cdot\nabla_{\mathbf p} (\varep^2\Delta\Phi_{\mathbf p} )\|_{L^\infty} \lesssim  \varep^{-1} \|\dot{\mbp}\|_{l^1}.
\end{aligned}
\end{equation*}
Using the $l^1$-$l^2$ estimate, $\|\dot \mbp\|_{l^1}\lesssim N_1^{1/2} \|\dot \mbp\|_{l^2}$ and bounding $N_1$ by $\varep^{-1}$ finishes the proof.
\end{proof}

\begin{lemma}\label{lem-Z*} The rate of change of the basis functions of the slow space $\mcZ_*$ with respect to $\mbp$ can be bounded by
\beqs
\|\nabla_\mbp Z_{\mbp, *}^{I(j)j}\|_{L^2}\lesssim \varep^{-1}.
\eeqs
Under the assumption $\|\dot\mbp\|_{l^2}\lesssim \varep^3$ of \eqref{A-0} we have
 \begin{equation}\label{est-L2-p-tZ*}
\|\p_t Z_{\mathbf p, *}^{I(j)j}\|_{L^2}\lesssim \varep^{-1}\|\dot{\mathbf p}\|_{l^2}\lesssim \varep^2.
\end{equation}
 
\end{lemma}
\begin{proof}
The basis functions of the tangent plane satisfy
 \beqs
 \p_{\mrp_i} Z_{\mathbf p, *}^{I(j)j}=\varep^{-1} \tilde \Theta_i\tilde\psi_{I(j)}' \tilde \Theta_j+O(1).
\eeqs
The result follows from the orthogonality of $\tilde \Theta_i$ on $L^2(\msI_\mbp)$ and $l^1$-$l^2$ estimate with $N\lesssim \varep^{-1}$. 
\end{proof}

\begin{lemma}\label{lem-mass-F} For $\mbp\in \cD_\delta$,  the mass of residual can be estimated by 
\begin{equation*}
\int_\Omega \left(\mathrm F(\Phi_{\mbp} ) -\mrF^\infty\right)\dd x= O\left(\varep^4, \varep^5\|\hat{\mbp}\|_{\mbV_4^2}\right).
\end{equation*}
\end{lemma}
\begin{proof}
We expand $\mathrm F(\Phi_{\mbp} )$  in  Lemma \ref{lem-def-Phi-p}, subtract $\mrF_m^\infty$ and integrate,
\begin{equation}
\label{Lemma3-5-eqn}
\begin{aligned}
\int_\Omega \left(\mathrm F(\Phi_{\mbp} )-\mrF^\infty\right)\dd x=&\varep^2\int_\Omega \mathrm F_2(\Phi_{\mbp} )\dd x+\varep^3\int_\Omega ( \mathrm F_3(\Phi_{\mbp} ) -\mrF_3^\infty)\dd x+\varep^4\int_\Omega (\mathrm F_{\geq 4}(\Phi_{\mbp} )-\mrF_{\geq 4}^\infty)\dd x.  
\end{aligned}
\end{equation}
Recalling form of $\mathrm F_2$ in \eqref{F-234}, where $f_2(z_{\mbp})$ is odd in $z_{\mbp}$, we deduce
\begin{equation}\label{F-mass}
\begin{aligned}
\varep^2 \int_\Omega  \mathrm F_2(\Phi_{\mbp} )\dd x&=-\varep^4(\sigma-\sigma_1^*)\int_{\mathbb R_{2\ell}}\int_{\msI}  \kappa_{\mbp}^2 z_{\mbp}f_2(z_{\mbp})\dd \tilde s_{\mbp}\mrd z_{\mbp},\\
&=C\varep^4(\sigma-\sigma_1^*)\int_{\msI} h(\bm \gamma_{\mbp}'')\dd \tilde s_{\mbp},
\end{aligned}
\end{equation}
where the function $h=h(\bm \gamma_{\mbp}'')$ follows Notation\,\ref{Notation-h}. Since $h$ is bounded in $L^\infty$ by Lemma \ref{lem-e_ij}, we have
\begin{equation}\label{F2-mass}
\begin{aligned}
\varep^2 \int_\Omega \mathrm F_2(\Phi_{\mbp} )\dd x=O(\varep^4|\sigma-\sigma_1^*|).
\end{aligned}
\end{equation}
From the form of $\mathrm F_3$ given in \eqref{F-234}, the odd parity of $\phi_0'$ with respect to $z_{\mbp}$ implies
\begin{equation}\label{F3-mass}
\begin{aligned}
\varep^3 \int_\Omega ( \mathrm F_3(\Phi_{\mbp} ) -\mrF_3^\infty)\dd x&=\varep^4\int_{\mathbb R_{2\ell}}\int_{\msI} \left( \phi_0'\Delta_{s_{\mbp}}\kappa_{\mbp}+(f_3(z_{\mbp},\bm \gamma_{\mbp}'') -f_3^\infty)\right)\mrJ_\mbp\dd s_{\mbp}\mrd z_{\mbp},\\
&=\varep^4 \left(\varep \int_{\msI} h(\bm \gamma_{\mbp}'')\Delta_{s_{\mbp}}\kappa_{\mbp} \dd s_{\mbp} +\int_{\msI} h(\bm \gamma_{\mbp}'') \dd s_{\mbp} \right),\\
&=O\left(\varep^5\|\hat{\mbp}\|_{\mbV_4^2}\right)+O(\varep^4).
\end{aligned}
\end{equation}
Here we used the $H^2(\msI_\mbp)$ bound of the curvature from Lemma \ref{lem-def-gamma-p}. Similar estimates show that
\begin{equation}\label{F4-mass}
\begin{aligned}
&\varep^4 \int_\Omega ( \mathrm F_{ 4}(\Phi_{\mbp} ) -\mrF_4^\infty)\dd x 
=O\left(\varep^5\|\hat{\mbp}\|_{\mbV_4^2},\varep^4\right).
\end{aligned}
\end{equation}
Combining  \eqref{F2-mass}-\eqref{F4-mass}  with \eqref{F-mass} yields \eqref{Lemma3-5-eqn}.
\end{proof}

\subsection{The Decomposition}

In the section, we prove Lemma\,\ref{lem-Manifold-Projection}. It suffices to consider  $\mbp_0=0$. 
Define $\{\msF_k\}_{k=0}^{N_1-1}$ to be real-valued  functionals of $v$ and  the parameters $\mathbf p$, explicitly  given by
 \begin{equation*}
 \mathscr F_k(v,\mathbf p)= \int_\Omega \Big(\Phi_0 +v(  x)-\Phi_{\mathbf p}  \Big) Z_{\mathbf p, *}^{1k}\, \mrd   x, 
 \end{equation*}
 for $i\in \Sigma_1(\Gamma_{\mathbf p}, \rho)$. Note that $\msF_k(0, \bm 0)=0$ and from mean value theorem 
 \beqs
 \Phi_0-\Phi_\mbp = -\mbp\cdot \nabla_\mbp \Phi_{\la \mbp}, \qquad \hbox{for some $\la=\la(\mbp)\in[0,1]$}.  
  \eeqs
If we introduce $\{\mathbf e_k\}$ as the canonical basis of  $\mbR^{N_1}$ and the $\mbp$ dependent notation $\mbT_{kj}^\la=\mbT_{kj}^\la(\mbp)$ as
  \beq\label{def-T}
  \mathbb T_{kj}^\la(\mbp):= \int_\Omega \p_{\mrp_j}\Phi_{\la \mbp} \, Z_{\mbp,*}^{1k} \dd x, 
  \eeq
 then the functional $\msF_k$ can be rewritten as 
 \beq\label{def-msF-T}
 \msF_k(v,\mathbf p)= \int_\Omega v(  x) \, Z_{\mathbf p, *}^{1k}\, \mrd x - \left<\mbT^\la(\mbp) \mbp,\mathbf e_k \right>, 
 \eeq
 and the gradient of $\msF:=(\msF_k)$ with respect to $\mbp$ at $(v,\mbp)=(0, \bm 0)$ can be represented by
 \beqs
\nabla_\mbp \msF(0, \bm 0)= - \mbT^T(\bm 0). 
 \eeqs
 Here $\mbT(\bm 0)= \mbT^1(\bm 0)$ and the superscript $T$ denotes matrix transpose. We will show in Lemma \ref{lem-T} that $\mbT(\bm0)$ is invertible.   We use the contraction mapping theorem to establish the existence of $\mbp$ such that $\msF_k(v, \mbp)=0$ for some $v$. We define the function
 \beqs
 \mathscr G(\mbp; v)= (\mbT(\bm 0))^{-1} \Big(\mbT(\bm 0)\mbp- \msF(v;\mbp)\Big), 
 \eeqs
and show that $\mathscr G$ is a contraction of $\mbp$ near the origin when $v$ is small in $L^2$. We observe $\msG(\mbp; v)=\mbp$ is equivalent to $\msF(v, \mbp)=\bm0$, that is, any fixed point of $\msG(\mbp)$ for a given $v$ is a zero  of $\msF(v, \mbp)$.   With $\msF$  written in terms of $\mbT$ as in \eqref{def-msF-T}, we rewrite 
 \beq\label{msg-k}
 \mathscr G_k(\mbp; v)= \left<(\mbT(\bm 0))^{-1} \Big(\mbT(\bm 0)- \mbT^\la(\mbp)\Big) \mbp, \mathbf e_k \right>_{l^2} + (\mbT(\bm 0))^{-1} \int_\Omega v(x) \, Z_{\mbp, *}^{1k} \dd x
 \eeq
The following properties of $\mbT^\la(\mbp)$  show that $\msG(\mbp; v)$ is a contraction mapping of $\mbp$ near zero.   

\begin{lemma}\label{lem-T}
$\mbT(\bm 0)$ is invertible and satisfies the bound
\beqs
\|(\mbT(\bm0))^{-1}\|_{l^2_*}\lesssim \varep^{1/2}.
\eeqs
Moreover the difference between  $\mbT^{\la_1}(\mbp_1)$ and $\mbT^{\la_2}(\mbp_2)$, for $\la_k:=\la(\mbp_k)$,  
satisfies
 \beqs
 \|\left(\mbT^{\la_1}(\mbp_1)-\mbT^{\la_2}(\mbp_2)\right)\mbp_l\|_{l^2}\lesssim \varep^{-2} \|\mbp_l\|_{l^2}\|\mbp_1-\mbp_2\|_{l^2}, \quad l=1,2. 
 \eeqs
\end{lemma}
\begin{proof} 
By the definition of $\mbT^\la_{kj}(\mbp)$ in \eqref{def-T} with $Z_{\mbp,*}^{1k}$ replaced by its definition and $\p_{\mrp_j}\Phi_\mbp$ given in Lemma \ref{lem-Phi_t-p}, there exists a matrix with $l_*^2$-norm one such that
 \beq\label{est-bT-0}
 \mbT_{kj}^\la(\mbp)=\frac{1}{\varep^{1/2}} \int_{\mbR_{2\ell}} \int_{\msI_\mbp} (\phi_0'(z_{\la\mbp}) +\varep \phi_1'(z_{\la\mbp}))\frac{\phi_0'(z_\mbp)}{m_1} \xi_j(s_{\la\mbp})\tilde \Theta_k(s_\mbp)\dd \tilde s_\mbp \dd z_\mbp +\varep^{3/2}\mbE_{ij}. 
 \eeq
 Here $(s_{\la \mbp}, z_{\la \mbp})$ denotes the scaled local coordinates near $\Gamma_{\la \mbp}$ for $\la \in [0,1]$. Applying Lemma \ref{lem-change-of-coord} and the $l^1$-$l^2$ estimate with $N_1\lesssim \varep^{-1}$ implies
 \beqs
 |z_{\la \mbp}-z|+\varep^{-1}|s_{\la \mbp}-s|\lesssim \varep^{-1} \|\mbp\|_{\mbV_0}\lesssim \varep^{-1}N_1^{-1/2}\|\mbp\|_{l^2}\lesssim\varep^{-3/2}\|\mbp\|_{l^2}, 
 \eeqs 
 where we recall $\mbV_0$ is equivalent to $l^1(\mbR^{N_1})$.  Thus the estimate on the  difference of $\mbT^{\la_1}(\mbp_1)$ and $\mbT^{\la_2}(\mbp_2)$ follows from standard calculations.
We now estimate the inverse of $\mbT(\bm 0)$, i.e.  $\mbT^\la(\bm 0)$ for $\la=1$.  
Fomr the definition of $\phi_0, \phi_1$ in \eqref{def-phi0n2}, \eqref{def-phi1},  we deduce the useful but straightforward identity
\beq\label{est-bT-1}
\int_{\mathbb R} (\phi_0'+\varep\phi_1')\phi_0' \dd z
= m_1^2+\varep(\sigma m_2+\eta_d m_3^2), 
\eeq
 where $m_1$ is introduced  in \eqref{def-m1}, and $m_2, m_3$ are constants defined by \beq\label{def-m23}
m_2=\frac{1}{2} \int_{\mathbb R} \mrL_0^{-1}(z\phi_0') \dd z; \qquad m_3= \frac{1}{2} \int_{\mathbb R} |z\phi_0'|^2 \dd z.
\eeq 
 Applying \eqref{est-bT-1} and taking $\mbp=\bm0$, \eqref{est-bT-0} implies  
  \beq\label{est-bT-2}
 \mathbb T_{kj}(\bm 0)=\frac{m_1^2 +\varep(\sigma m_2+\eta_d m_3^2)}{\varep^{1/2}}  \int_{\msI}  \xi_j(s) \tilde \Theta_k(\tilde s)   \dd \tilde s +O(\varep^{3/2} \mbE_{kj}),
 \eeq
 where $\xi_j(s)=-\Theta_j $ and $\tilde \Theta_k(\tilde s)=\Theta_k(s)$ since $\tilde s=s$.  Hence the  orthogonality of $\Theta_j$ in $L^2(\msI)$ implies
 \beqs
\mbT(\bm 0)= -\frac{m_1^2+\varep (\sigma m_2+\eta_d m_3^2)}{\varep^{1/2}}\mathbb I+O(\varep^{3/2}) \mbE.
\eeqs
Since $\mbE$ has $l^2_*$-norm one, the first term dominates and hence $\mbT(\bm 0)$ is invertible. In fact, for some uniform constant $C$ it holds that
\beqs
\|\mbT(\bm 0)\mbp\|_{l^2} \geq C\varep^{-1/2}\|\mbp\|_{l^2}. 
\eeqs
The $l^2_*$-bound of the inverse $\mbT(\bm 0)^{-1}$ follows, which completes the proof. 
\end{proof}

\begin{lemma}\label{lem-IFT-1} Let $\Gamma_0\in\cG^4_{K_0,2\ell_0}$ with local coordinates $(s, z)$ and $\varep\in(0,\varep_0)$. Let $0\leq r \leq 1$ and $\|v\|_{L^2} \leq  \delta \varep$ for some $\delta, \varep_0$ small enough, 
 then  there exists $\mathbf p=\mathbf p(v)\in l^2$ such that $\msF(v, \mathbf p(v))=0$ and 
\begin{equation*}
\|\mathbf p(v)\|_{l^2}\lesssim \varep^{1/2}  \|v\|_{L^{2}}.
\end{equation*}
The smallness of $\delta,\varep_0$ depend on domain, system parameter and $(\Gamma_0,M_0)$
\end{lemma}  
 \begin{proof}
 We check $\msG(\mbp; v)=(\msG_k(\mbp; v)):l^2\to l^2$ is a contraction mapping  for $\mbp\in l^2$ satisfying $\|\mbp\|_{l^2}\leq \delta  \varep^{3/2}$ for some $\delta$ small enough independent of $\varep$. In fact, employing Lemma \ref{lem-T} yields
 \beq\label{est-msG}
 \|\msG(\mbp; v)\|_{l^2}\lesssim \varep^{1/2}\|v\|_{L^2} +\varep^{-3/2}\|\mbp\|_{l^2}^2, 
 \eeq
 which lies in the ball $B_{\delta \varep^{3/2}}(\bm 0)\subset l^2$ provided that $\delta, \varep$ are suitably small. Hence $\msG$ is a closed map on the small ball  $B_{\delta\varep^{3/2}}(\bm 0)$. It remains to show the mapping is contractive. Indeed, for any $\mbp_1, \mbp_2\in B_{\delta \varep^{3/2}}(\bm 0)$ from  \eqref{msg-k} we have 
 \beqs
 \begin{aligned}
 \msG_k(\mbp_1) -\msG_k (\mbp_2)=&\left<  \left(\mbT(\bm 0)\right)^{-1} \left[-\left(\mbT^{\la_1}(\mbp_1) -\mbT(\bm0)\right)(\mbp_1-\mbp_2) +\left( \mbT^{\la_2}(\mbp_2) -\mbT^{\la_1}(\mbp_1)\right)\mbp_2\right], \mathbf e_k \right>_{l^2}  \\
 & + \left(\mbT(\bm 0)\right)^{-1} \int_\Omega v(x) \left(Z_{\mbp_1, *}^{1k}-Z_{\mbp_2, *}^{1k}\right) \dd x
 \end{aligned}
 \eeqs 
From the gradient estimate of $Z_{\mbp,*}^{I(k)k}$ in Lemma \ref{lem-Z*}, we bound
\beqs
\|Z_{\mbp_1, *}^{1k}-Z_{\mbp_2, *}^{1k}\|_{L^2} \lesssim \varep^{-1} \|\mbp_1-\mbp_2\|_{l^2}, \qquad \forall k\in \Sigma_1.   
\eeqs
 Combining this with Lemma \ref{lem-T} and $\mbp_1,\mbp_2\in B_{\delta\varep^{3/2}}$ with $\delta$ suitably small  yields
 \beqs
 \begin{aligned}
 \|\msG(\mbp_1)-\msG(\mbp_2)\|_{l^2} &\lesssim \left(\varep^{-3/2}\|\mbp_1\|_{l^2} +\varep^{-3/2}\|\mbp_2\|_{l^2} +\varep^{-1/2}N_1^{1/2}\|v\|_{L^2}\right) \|\mbp_1-\mbp_2\|_{l^2}\\
 &\leq \frac{1}{2}\|\mbp_1-\mbp_2\|_{l^2}.
 \end{aligned}
 \eeqs
Hence $\msG$ is a contraction mapping on the space $B_{\delta\varep^{3/2}}(\bm 0)\subset l^2$. The existence of $\mbp \in B_{\delta\varep^{3/2}}(\bm 0)$ such that
\beqs
\msG(\mbp; v)=\mbp
\eeqs
follows from contraction mapping principle. And then the bound of $\mbp$ in terms of $L^2$-norm of $v$  follows from \eqref{est-msG} provided with $\varep_0$ suitably small.
 \end{proof}


 Now we prove the projection Lemma \ref{lem-Manifold-Projection}. 
 \begin{proof}[Proof of Lemma \ref{lem-Manifold-Projection}] Without loss of generality we assume $\mbp_0=\bm0$. 
With a use of Lemma \ref{lem-IFT-1}, there exists $\mbp=\mbp(v)$ such that 
 \beqs
 u=\Phi_{\mbp}+v^\bot, \qquad v^\bot \in (\mcZ_*^1)^\bot, \qquad \int_\Omega v^\bot \dd x=0
 \eeqs
 and $\mbp$ in $l^2$ can be bounded by 
 \beqs
\|\mbp\|_{l^2}\lesssim  \varep^{1/2}\|v\|_{L^2}.
 \eeqs
 While $u=\Phi_0 +v^\bot$, the mean value theorem and the bound of $\mbp$ afford the estimate
 \begin{equation*}
\begin{aligned}
\|v^\bot\|_{H^{2k}}&\lesssim \sup_{\la\in [0,1]}\sum_j\left(\left\|\mathrm p_j\frac{\p \Phi_{\mathbf p} }{\p \h{0.5pt}  \mathrm p_j}\right\|_{H^{2k}}\right)\bigg|_{\mathbf p=\la\mathbf p}+\|v\|_{H^{2k}}\\
&\lesssim \varep^{-2k-1/2}\|\mathbf p\|_{l^2}
+\|v_0\|_{H^{2k}}\\
&\lesssim \varep^{-2k}\|v_0\|_{L^2} +\|v_0\|_{H^{2k}}.
\end{aligned}
\end{equation*}
 By Lemma \ref{lem-proj-Z0}, $v^\bot$ can be further decomposed as
 \beqs
 v^\bot=Q + w, \qquad Q=\sum_{j\in \Sigma_0}\mrq_j \in \mcZ^0_*, 
 \eeqs
 where the coefficient vector $\mbq$ satisfy 
 \beqs
 \|\mbq\|_{l^2}\lesssim \|v^\bot\|_{L^2}.
 \eeqs
 Finally, with $w=v^\bot-Q$ we bound
 \beqs
 \|w\|_{H^{2k}}\lesssim \|v^\bot\|_{H^{2k}} +\|Q\|_{H^{2k}}\lesssim \|v\|_{H^{2k}} +\varep^{-2k}\|\mbq\|_{l^2}\lesssim \varep^{-2k}\|v\|_{L^2}+\|v\|_{H^{2k}}. 
 \eeqs
 The proof is completed by translating the base point from $0$ to $\mbp_0$ and $\mbp$ to $\mbp-\mbp_0$.  
  \end{proof}

\bibliographystyle{amsplain}

\end{document}